\documentclass{amsart}

\numberwithin{equation}{section}

\usepackage{enumerate}
\usepackage{amsmath, amssymb}
\usepackage{graphicx}
\usepackage{mathrsfs}
\usepackage{color}

\usepackage[normalem]{ulem}
\usepackage{blkarray}
\usepackage{nicefrac}
\usepackage{soul}
\usepackage{color}

\usepackage[normalem]{ulem}
\usepackage{blkarray}
\usepackage{nicefrac}
\usepackage{ytableau}
\usepackage{mathtools}
\usepackage{hyperref}

\newcommand{\tr}{\text{tr}}
\newcommand{\sign}{\text{sgn}}
\newcommand{\morph}{\text{End}}
\newcommand{\stab}{\text{stab}}
\newcommand{\orb}{\text{orb}}

\newcommand{\matrixspace}[1]{M_{#1 \times #1}}

\newtheorem{thm}{Theorem}[section]
\newtheorem{cor}[thm]{Corollary}
\newtheorem{prop}[thm]{Proposition}
\newtheorem{lem}[thm]{Lemma}
\newtheorem*{conj}{Conjecture}

\begin{document}
\title{Determinants of Random Unitary Pencils}
\author{Michael T. Jury}
\address{Department of Mathematics, University of Florida, PO Box 118105, Gainesville, FL 326211-8105, USA}
\email{mjury@ufl.edu} 
\author{George Roman}
\address{Department of Mathematics, University of Florida, PO Box 118105, Gainesville, FL 326211-8105, USA}
\email{g.roman@ufl.edu} 
\thanks{Jury partially supported by National Science Foundation Grant DMS-2154494. Roman partially supported by National Science Foundation Grant DMS-2154494 during a visit to Washington University in St. Louis, Spring 2025}

\begin{abstract}
We investigate determinants of random unitary pencils (with scalar or matrix coefficients), which generalize the characteristic polynomial of a single unitary matrix. In particular we examine moments of such determinants, obtained by integrating against the Haar measure on the unitary group. We obtain an exact formula in the case of scalar coefficients, and conjecture an asymptotic formula in the general case, and prove a special case of the conjecture. 
\end{abstract}

\maketitle

\section{Introduction}\label{sec:intro}

Let $U(d)\subset M_{d\times d}(\mathbb C)$ denote the $d\times d$ unitary group. $U(d)$ is a compact Lie group, equipped with a unique, unimodular Haar probability measure which we denote $dU$. By a {\it unitary pencil} we mean an expression of the form
\begin{equation}\label{eqn:intro-pencil-definition}
  L_{\mathcal X}(\mathcal U)\vcentcolon= I_k\otimes I_d + \sum_{j=1}^g X_j \otimes U_j
\end{equation}
where $\mathcal X\vcentcolon=(X_1, \dots, X_g)$ is a fixed tuple of $k\times k$ matrices, the $\mathcal U\vcentcolon=(U_1, \dots, U_g)$ are unitary matrices, and $I_m$ is the $m \times m$ identity matrix (or just $I$ when the size is clear from context). The symbol $\otimes$ denotes the usual Kronecker tensor product. We may consider the $U_j$ as independent random variables, sampled according to the Haar measure on $U(d)$, in which case we will call (\ref{eqn:intro-pencil-definition}) a {\it random unitary pencil}.

We will be interested in the determinants of random unitary pencils:
\begin{equation}\label{eqn:det-pencil}
  \det L_{\mathcal X}(\mathcal U) = \det \left(I_k\otimes I_d + \sum_{j=1}^g X_j \otimes U_j\right).
\end{equation}
When $g=1$ and $k=1$, this expression reduces to $\det(I+xU)$, which is essentially the characteristic polynomial of $U$, up to the change of variable $U\to -U^*$ (which leaves Haar measure invariant). For $g=1$ and $k>1$, we obtain a product of characteristic polynomials---indeed, in that case one sees that if $x_1, \dots, x_k$ are the eigenvalues of the coefficient matrix $X$, then we have
\begin{equation}\label{eqn:matrix-form-det}
  \det (I_k\otimes I_d +X\otimes U) = \prod_{l=1}^k \det(I_d+x_lU)
\end{equation}
Returning to the general expression (\ref{eqn:det-pencil}), if we think of the $U_j$ as fixed, we may view $\det L_{\mathcal X}(\mathcal U)$ as a function of the coefficient matrices $\mathcal X=(X_1, \dots, X_g)$, of varying sizes $k$. From this point of view it can be seen as a kind of ``joint characteristic polynomial'' of the tuple $\mathcal U=(U_1, \dots, U_g)$. 

There has been extensive work on the statistics of random characteristic polynomials drawn from various random matrix ensembles, a small sample is \cite{Baik-Rains-2001, Basor-Conrey-2024, Borodin-Strahov-2006, Bump-Gamburd-2006, Conrey-Farmer-Keating-Rubinstein-Snaith-2003, Keating-Snaith-2000}. This work also makes significant contact with other areas of mathematics, notably combinatorics and number theory. In particular the behavior of random characteristic polynomials for single unitary matrices has received a great deal of attention because of its connections to the Riemann zeta function \cite{Keating-Snaith-2000,Conrey-Farmer-Keating-Rubinstein-Snaith-2003}. In that context, what is wanted are expressions for correlations between products (and quotients) of random characteristic polynomials, i.e. one would like to understand quantities like
\begin{equation}\label{eqn:classical-problem}
  \int_{U(d)} \prod_{l=1}^k \det (I_d+x_lU)\prod_{l^\prime=1}^{k^\prime} \det(I_d+\overline{y_{l^\prime}}U^*)\, dU
\end{equation}
for complex scalars $x_1, \dots, x_k, y_1, \dots, y_{k^\prime}$, where the integral is performed with respect to the normalized Haar measure on $U(d)$. It turns out that exact formulas for the expression (\ref{eqn:classical-problem}) are known (see for example \cite[Formula 2.16]{Conrey-Farmer-Keating-Rubinstein-Snaith-2003} or \cite[Proposition 4]{Bump-Gamburd-2006}; the integral can also be expressed in terms of Toeplitz determinants via the so-called ``Heine-Szeg\H{o} identity'', see for example \cite{Bump-Diaconis-2002} and its references). These formulas can become somewhat complicated but often become simpler in the large size limit (as the size $d$ of the unitaries $U$ goes to infinity). For example in (\ref{eqn:classical-problem}) one can show that 
\begin{equation}\label{eqn:classical-limit}
  \lim_{d\to \infty} \int_{U(d)} \prod_{l=1}^k \det (I_d+x_lU)\prod_{l^\prime=1}^{k^\prime} \det(I_d+\overline{y_{l^\prime}}U^*)\, dU = \prod_{l=1}^k\prod_{l^\prime=1}^{k^\prime} \frac{1}{1-x_l\overline{y_{l^\prime}}}.
\end{equation}
where the limit exists only under the additional assumption that all the $|x_l|, |y_{l^\prime}|$ are less than $1$. (For example, one can prove this by combining the Heine-Szeg\H{o} identity with the Strong Szeg\H{o} Limit Theorem, which can be found in \cite{Bump-Diaconis-2002}.) In the special case of single factors ($k=k^\prime=1$) we have the exact formula
\begin{equation}\label{eqn:HS-ident}
		\int_{U(d)} \det(1 + x U) \det(1 + \overline{y} U^\ast) \, dU = 1 + x\overline{y} + \cdots + (x \overline{y})^d
\end{equation}
for any $x,y \in \mathbb{C}$, and the limiting formula
\begin{equation}\label{eqn:HS-limit}
		\lim_{d \to \infty} \int_{U(d)} \det(1 + x U) \det(1 + \overline{y} U^\ast) \, dU = \frac{1}{1 - x\overline{y}}
\end{equation}
when $|x|, |y|<1$.

Observe that using the identity (\ref{eqn:matrix-form-det}), we can rewrite (\ref{eqn:classical-limit}) in the more compact form
\begin{equation}\label{eqn:classical-matrix-limit}
  \lim_{d\to \infty} \int_{U(d)} \det(I_k\otimes I_d+X\otimes U) \overline{\det(I_{k^\prime}\otimes I_d+Y\otimes U)}\, dU = \det\left( I_k\otimes I_{k^\prime} -X\otimes \overline{Y} \right)^{-1},
\end{equation}
assuming that all the eigenvalues of the $X$ and $Y$ matrices are less than $1$, that is, that $X$ and $Y$ have spectral radius less than $1$. (Here $\overline{Y}$ denotes the {\it entrywise} complex conjugate of the matrix $Y$.)

The goal of the present paper is to formulate and investigate the following conjecture about the expression corresponding to (\ref{eqn:classical-matrix-limit}) in the multivariable case $g>1$. For a $g$-tuple of matrices $\mathcal X = (X_1, \dots, X_g)$ we define the {\it outer spectral radius} $\textbf{rad}(\mathcal X)$ to be
\[
\textbf{rad}(\mathcal X) \vcentcolon= rad \left(\sum_{j=1}^g X_j\otimes \overline{X_j}\right)^{1/2},
\]
where $rad(X)$ denotes the ordinary spectral radius of a matrix.

We have:

\begin{conj}\label{conj}
  Let $\mathcal X, \mathcal Y$ be $g$-tuples of square matrices of size $k\times k$, $k^\prime \times k^\prime$ respectively, with $\textbf{rad}(\mathcal X)$, $\textbf{rad}(\mathcal Y)<1$. Then
  \begin{equation}\label{eqn:intro-conj}
    \lim_{d\to \infty} \int_{U(d)^g} \det (L_{\mathcal X}(\mathcal U))\overline{\det (L_{\mathcal Y}(\mathcal U))}\, d\mathcal U = \det\left( I_k\otimes I_{k^\prime} -\sum_{j=1}^g X_j\otimes \overline{Y_j}\right)^{-1}.
  \end{equation}
\end{conj}
\noindent where $d\mathcal{U} \vcentcolon= dU_1 \times \cdots \times dU_g$ is the $g$-fold product measure on $U(d)^g$.

The outer spectral radius hypothesis is natural in the sense that if $\textbf{rad}(\mathcal X)$, $\textbf{rad}(\mathcal Y)<1$ then the spectral radius of $\sum_{j=1}^g X_j\otimes \overline{Y_j}$ is strictly less than $1$, so that the right hand side exists, and can be expanded into a suitably convergent series (Lemmas~\ref{lem:cauchy-identities-redux} and \ref{lem:target-expansion}, but see also the remark following Proposition~\ref{prop:CS-for-outer}). We will show that the conjecture holds in a formal sense (by treating both sides as formal power series in a suitable way), and then give a rigorous proof of the conjecture in an important special case. First, in the very special case of scalar coefficients ($k=k^\prime =1$), we are able to obtain an explicit expression for the integral at each finite sized $d$, analogous to (\ref{eqn:HS-ident}). Precisely, we shall prove
\begin{thm}\label{thm:scalar-closed-form}
  For vectors  $x=(x_1, \dots, x_g)$ and $y=(y_1, \dots, y_g)\in \mathbb C^g$, we have for all $d\geq 1$
  \begin{equation}\label{eqn:scalar-closed-form}
    \int_{U(d)^g} \det (L_x(\mathcal U))\overline{\det (L_y(\mathcal U))}\, d\mathcal U =  \sum_{n=0}^d \sum_{ |\alpha| = n } c(d, \alpha) \binom{n}{\alpha} x^\alpha \overline{y^\alpha}
  \end{equation}
	
	\noindent where
	
	\begin{equation}
		c(d, \alpha) = \binom{d}{n} \binom{n}{\alpha} \prod_{j=1}^g \binom{d}{\alpha_j}^{-1}.
	\end{equation}
\end{thm}
(Here we have used the usual multi-index notation $x^\alpha$.) It will be easy to show that the factors $c(d, \alpha)$ increase to $1$ as $d\to \infty$, which will give
\begin{cor}\label{cor:scalar-limit}
	
	For any $x, y \in \mathbb{C}^g$ such that $\|x\|_2 , \|y\|_2 < 1$, we have
	
	\begin{equation}\label{eqn:scalar-limit}
		\lim_{d \to \infty} \int_{ U(d)^g } \det ( L_x( \mathcal{U} ) ) \; \overline{ \det ( L_y( \mathcal{U} ) ) } \, d\mathcal{U} = \frac{1}{1 - \langle x , y \rangle} .
	\end{equation}
	
\end{cor}
\noindent Here $\langle \cdot, \cdot\rangle$ and $\|\cdot \|_2$ denote the usual Euclidean inner product and norm on $\mathbb C^g$.

This establishes the conjecture (\ref{eqn:intro-conj}) (in sharper form, with an exact formula for each $d$) in the case $k=k^\prime=1$.

For higher values of $k, k^\prime$ we are currently able to prove the conjecture only under an additional assumption on the coefficients $\mathcal X, \mathcal Y$:

\begin{thm}\label{thm:triangular-limit} The conjecture (\ref{eqn:intro-conj}) holds if all of the coefficient matrices $\mathcal X$, $\mathcal Y$ are upper triangular.
  \end{thm}
Here we do not obtain a very explicit expression at finite sizes $d$, though the coefficients can be written down in terms of some rather complicated combinatorial sums. Of course, it suffices to assume only that each of the tuples $\mathcal X, \mathcal Y$ is jointly similar to an upper triangular tuple, since the determinant is unaffected by an overall similarity. In particular, since a commuting system of matrices can be put into simultaneous triangular form, we obtain:
\begin{cor}\label{cor:commuting}
  The conjecture(\ref{eqn:intro-conj}) holds if each of the tuples of coefficient matrices $\mathcal X=(X_1, \dots, X_g)$, $\mathcal Y=(Y_1, \dots, Y_g)$ commute among themselves.
  \end{cor}

From elementary properties of the determinant, one sees that the triangular case immediately reduces to the diagonal case, and if we write out the diagonal case in concrete terms we obtain
\begin{cor}\label{cor:diagonal-cor}

  For vectors
  \[ x_l = (x_l^{(1)}, \dots, x_l^{(g)}),\  l=1, \dots, k, \quad\quad y_{l^\prime}=(y_{l^\prime}^{(1)}, \dots, y_{l^\prime}^{(g)}), \  l^\prime =1, \dots k^\prime
  \]
 in $\mathbb C^g$  with $\|x_l\|_2, \|y_{l^\prime}\|_2<1$,  we have
  \begin{align}\label{eqn:diagonal-cor}
    \lim_{d\to \infty} &\int_{U(d)^g} \prod_{l=1}^k \det(I+\sum_j x_l^{(j)} U_j)\prod_{m=1}^{k^\prime}\overline{ \det(I+\sum_j y_{l^\prime}^{(j)} U_j)}\, d\mathcal{U} \\
    &= \prod_{l=1}^k\prod_{l^\prime=1}^{k^\prime} \frac{1}{1-\langle x_l, y_{l^\prime}\rangle} \nonumber
  \end{align}

  \end{cor}
This bears an obvious resemblance to (\ref{eqn:classical-limit}) (and, of course, recovers it exactly when $g=1$). The key thing to observe is that the Corollaries~\ref{cor:scalar-limit} and \ref{cor:diagonal-cor} identify the correct domain of convergence for the coefficients in these multivariable cases built from linear factors: it is the open unit ball in $\mathbb C^g$. The conjecture speculates that the correct domain of convergence in the general case is the {\it outer spectral radius ball} $\textbf{rad}(\mathcal X)<1$.

In the course of proving Theorem~\ref{thm:triangular-limit} we also obtain moment bounds on the random variables $|\det(I+\sum x_j U_j)|$ (in the case of scalar coefficients $x_j$). Precisely:

\begin{prop}\label{prop:moment-bounds} For each integer $k\geq 1$ and each $0\leq r<1$, there is a number $C(r,k)$ such that for all $d\geq 1$ and all vectors $x=(x_1, \dots, x_g)\in \mathbb C^g$ with $\|x\|_2\leq r$,
  \begin{equation}\label{eqn:moment-bound}
    \int_{U(d)^g} \left|\det L_x(\mathcal U)\right|^{2k}\, d\mathcal U \leq C(r,k).
  \end{equation}
\end{prop}
\noindent (It is trivial that for fixed $d$ and any $x$, one has finite moments of all orders; the content of the proposition is that there are bounds independent of $d$ once $x$ is restricted to the compact set $\|x\|_2\leq r<1$.)

One might conjecture that an analog of this proposition should hold for matrix coefficients $\mathcal X$; indeed such a statement would imply (\ref{eqn:pencil-limit-conjecture}) (see Proposition~\ref{prop:L2-bound-is-sufficient} below).

\subsection{Discussion and overview of the proofs}

Our original interest in this problem was motivated by the identity (\ref{eqn:HS-limit}), the right-hand side being the well known {\em Szeg\H{o} kernel} $k(x,y)=(1-x\overline{y})^{-1}$, which is the reproducing kernel for the Hardy space $H^2$ over the unit disk (the Hilbert space of holomorphic functions with square-summable power series). There has been much activity in recent years around a particular multivariable analog of the Hardy space, called the {\em Drury-Arveson space}, which is the reproducing kernel Hilbert space over the unit ball $\mathbb B^g\subset \mathbb C^g$ with kernel $k(x,y)=(1-\langle x,y\rangle)^{-1}$. From many points of view (but primarily from the perspective of multivariable operator theory), the Drury-Arveson space is the natural multivariable analog of $H^2$. We refer to \cite{Hartz-2023} for a recent survey, with extensive references. Thus, Corollary~\ref{cor:scalar-limit} shows that the Drury-Arveson kernel also arises as the generalization of the Szeg\H{o} kernel in a context quite different from that of multivariable operator theory. It would be interesting to pursue this connection further.

Despite the evident similarity of the multivariable expressions (\ref{eqn:scalar-limit}, \ref{eqn:diagonal-cor}) to their one-variable counterparts (\ref{eqn:classical-limit}, \ref{eqn:HS-limit}), none of the one-variable proofs carry over directly to the multivariable setting. The fundamental difficulty is that in one variable, the integrands, which are products of characteristic polynomials, depend only on the eigenvalues of $U$, so one has access to the Weyl integration formula, which allows the integral to be computed as an integral over the $d$-torus against a suitable density. In our setting, it is clear that even for scalar coefficients the expression $\det(I+\sum x_j U_j)$ is not a function of the eigenvalues of the individual $U_j$, so the Weyl formula cannot be applied. One can also prove the classical versions given here by using the so-called Heine-Szeg\H{o} identity to express the integral as a Toeplitz determinant, but we know of no comparable representation in the multivariable setting. The strategy that comes closest to working is that used by Bump and Diaconis \cite{Bump-Diaconis-2002} in their proof of the Strong Szeg\H{o} Limit Theorem, which invokes Schur-Weyl duality to re-express the integral over the unitary group in terms of characters and symmetric functions. This is broadly the approach we follow here, but it cannot be applied directly since the orthogonality relations for symmetric functions (coming from the {\em Hall inner product}) do not have exact analogs in the multivariable setting. They {\it do} have asymptotic analogs, however; these were discovered indpendently, in somewhat different contexts, by Mingo, \'Sniady, and Speicher \cite{Mingo-Sniady-Speicher-2007} and by R\u{a}dulescu \cite{Radulescu-2006}. So, the approach can be adapted but it requires either the ability to do an exact calculation (which we carry out in the case of scalar coefficients), or to obtain suitable bounds (uniform in $d$) on the resulting combinatorial expressions (which we do in the $k>1$, triangular case). It is possible to give a different proof of Theorem~\ref{thm:scalar-closed-form} using the so-called {\em Weingarten calculus} (in fact we originally proved the theorem by this technique), though generalizing that proof to $k>1$ seems prohibitively complicated. However we do make use of some of the ideas behind the construction of the Weingarten calculus by Collins \cite{Collins-2003} and Collins and \'Sniady \cite{Collins-Sniady-2006} in our proofs. 

\subsection{Reader's guide}

Section~\ref{sec:background} provides an overview of the necessary background material from combinatorics, representation theory, and matrix analysis which we will need. It also serves to establish notation to be used in the remainder of the paper. Most of this material will be well-known to readers who are familiar for example with the works of Bump and Diaconis \cite{Bump-Diaconis-2002} or Bump and Gamburd \cite{Bump-Gamburd-2006}.

In Section~\ref{sec:formal} we show that the limit in Conjecture~\ref{eqn:pencil-limit-conjecture} holds in a formal sense, by representing both sides in terms of suitable series expansions (Lemmas~\ref{lem:pencil-expansion-matrix} and \ref{lem:target-expansion}), and showing that the coefficients match in the limit (Corollary~\ref{cor:termwise-limit}). The key tool is the asymptotic orthogonality alluded to above (Theorem~\ref{thm:asymptotic-orthogonality}).  We then show that the full conjecture would follow (essentially by a normal families argument) if one could prove suitable bounds on the integrals independent of the size $d$ (Proposition~\ref{prop:L2-bound-is-sufficient}).

In Section~\ref{sec:scalar} we focus on the case of scalar coefficients ($k=1$) and obtain the exact expression (\ref{eqn:scalar-closed-form}) and its limiting version (\ref{eqn:scalar-limit}), these are proved as Theorem~\ref{thm:main} and Corollary~\ref{cor:main-limit} respectively. Making use of the expansions obtained in Section~\ref{sec:formal}, after some manipulations the key computation is contained in Propositions~\ref{prop:trace-projection} and \ref{thm:trace-centalt} which compute the trace of the product of the expectation onto a certain subgroup of the symmetric group $S_n$, with the projection onto an irreducible representation of $S_n$ (specifically, the alternating representation).

Section~\ref{sec:triangular} contains the proof of Theorem~\ref{thm:triangular-limit} (restated there as Theorem~\ref{thm:true-for-triangular}). We first reduce the triangular case to the case of scalar multiplies of the identity matrix (Proposition~\ref{prop:uniform-L2-bound}), by H\"older's inequality and another normal families argument. Once this is done, the broad outline of the proof is similar to that of the scalar case in Section~\ref{sec:scalar}, but instead of just the alternating representation we have to deal with more general irreducible representations of $S_n$. As a result we are unable to obtain a tractable explicit expression, but must instead bound a certain combinatorial quantity (Theorem~\ref{thm:content-ratio-bound}). Since the proof of this bound is rather more involved it is relegated to its own section, Section~\ref{sec:content-ratio-lemma}.

Section~\ref{sec:remarks} contains some final discussion and remarks. In particular we show that the conjecture implies a more general version of itself, where the linear pencils can be replaced by certain more general polynomials in noncommuting arguments. We also apply our results to obtain asymptotics for the moments of characteristic polynomials of sums of random unitaries, reminiscent of the Strong Szeg\H{o} Limit Theorem, and indicate its relation to the {\em Brown measure} of a sum of independent Haar unitaries.

\subsection{Acknowledgments} The authors would like to thank the organizers and participants of the May 2024 Oberwolfach Workshop on Non-commutative Function Theory and Free Probability, where an early version of this work was presented.

\section{Background}\label{sec:background}

In this section we give a review of the material we will need from combinatorics, algebra, and matrix analysis, and also establish the notation to be used for the remainder of the paper. 

\subsection{Multi-indices and Words}\label{sec:words}

A multi-index is a tuple of non-negative integers, $\alpha = (\alpha_1 , \ldots , \alpha_g) \in \mathbb{N}^g$. The \textit{weight} of $\alpha$, denoted $|\alpha|$, is the sum of its entries i.e. $|\alpha| = \sum_{j=1}^g \alpha_j$. When $|\alpha| = n$, the \textit{multinomial coefficient} of $\alpha$ is
\begin{equation*}
	\binom{n}{\alpha} \vcentcolon= \frac{n!}{\alpha_1! \cdots \alpha_g!} = \frac{n!}{\alpha!}.
\end{equation*}

\noindent For $x = ( x_1 , \ldots , x_g ) \in \mathbb{C}^g$, we use the usual notation for monomials:
\begin{equation*}
	x^\alpha \vcentcolon= x_1^{\alpha_1} x_2^{\alpha_2} \cdots x_g^{\alpha_g} .
\end{equation*}

\noindent In particular we have the multinomial theorem
\begin{equation*}
	( x_1 + \cdots + x_g )^n = \sum_{|\alpha| = n} \binom{n}{\alpha} x^\alpha .
\end{equation*}

The {\it free monoid} generated by $\{ \ell_1 , \ldots , \ell_g \}$ is the set of all words (of finite length) in these letters, denoted $\mathbb{W}_g$. A typical word has the form $w = \ell_{i_1} \cdots \ell_{i_n}$ where $1 \leq i_k \leq g$, and we call $n$ the length of $w$. (We include the empty word, which has length $0$.) Let $\mathbb{W}_{g}(n)$ be the set of words of length $n$. We have a group action of the symmetric group $S_n$ on $\mathbb{W}_{g}(n)$ by permuting letters in the obvious way: 
\begin{equation*}
	\sigma(w) = \ell_{i_{\sigma^{-1}(1)}} \ell_{i_{\sigma^{-1}(2)}} \cdots \ell_{i_{\sigma^{-1}(n)}}.
\end{equation*}

\noindent Under this action, two words belong to the same orbit if and only if they have the same amount of each letter. The orbits of $\mathbb{W}_{g}(n)$ thus naturally correspond to multi-indices $\alpha$ of weight $n$, where $w$ contains $\alpha_j$ instances of the letter $\ell_j$. In this case we write (by slight abuse of notation) $w \in \alpha$. The size of the $\alpha$-orbit is $\binom{n}{\alpha}$. Each orbit has a canonical representative $w_\alpha$ which begins with $\alpha_1$-many copies of the letter $\ell_1$, followed by $\alpha_2$-many $\ell_2$'s, and so on. As an example in the free monoid on $\{ a, b, c, d \}$, the word $bacbcb \in \mathbb{W}_{4}(6)$ has letter count $\alpha = ( 1 , 3 , 2 , 0 )$, and the canonical word of this letter count is $w_{ ( 1 , 3 , 2 , 0 ) } = abbbcc$.

For a word $\ell_{i_1} \cdots \ell_{i_n} = w \in \mathbb{W}_{g}(n)$ and tuple of matrices $\mathcal{X} = (X_1 , \ldots , X_g)$, let
\begin{equation}\label{eqn:word-power}
	\mathcal{X}^{\otimes w} \vcentcolon= X_{i_1} \otimes X_{i_2} \otimes \cdots \otimes X_{i_n}.
\end{equation}
For example when $g=2$, we have $\mathcal{X}^{\otimes abba} = X_1 \otimes X_2 \otimes X_2 \otimes X_1$. For multi-index $\alpha$ we will also write
\begin{equation*}
	\mathcal{X}^{\otimes \alpha} \vcentcolon= \mathcal{X}^{\otimes w_\alpha} = X_1^{\otimes \alpha_1} \otimes X_2^{\otimes \alpha_2} \otimes \cdots \otimes X_g^{\otimes \alpha_g}.
\end{equation*}

\subsection{Partitions}\label{sec:partitions}

A {\it partition} $\lambda$ of a positive integer $n$ (denoted $\lambda \vdash n$ or $|\lambda| = n$) is a non-increasing sequence of positive integers $\lambda = ( \lambda_1 \geq \lambda_2 \geq \cdots \geq \lambda_p > 0 )$ such that $\sum_{i=1}^p \lambda_i = n$. Each $\lambda_i$ is called a {\it part} and the number of parts $p$ is interchangeably called the \textit{length} or \textit{height} of $\lambda$. We shall use ``height'' and denote it $ht(\lambda)$.

Given a partition $\lambda \vdash n$, the associated \textit{Young diagram} is a collection of $n$ boxes arranged on a grid so that the $i^{th}$ row from the top contains $\lambda_i$ boxes. For example the diagram of $\lambda = (3 , 2 , 2 , 1)$ is
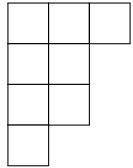
\begin{equation*}
	\ytableausetup{nobaseline,boxsize=1.5em}
	\begin{ytableau}
		*(white) & *(white) & *(white) \\
		*(white) & *(white) \\
		*(white) & *(white) \\
		*(white) \\
	\end{ytableau}
\end{equation*}
Each $\lambda$ has a \textit{conjugate partition} $\lambda^\ast$ whose Young diagram is the transpose of the diagram for $\lambda$. For example $(3,2,2,1)^\ast = (4,3,1)$. We set the \textit{width} of $\lambda$ to be $wd(\lambda) \vcentcolon= ht(\lambda^\ast)$ meaning $wd(\lambda) = \lambda_1$. Thus, the height and width of $\lambda$ are the literal height and width of its Young diagram.

The boxes of $\lambda$ are labeled by coordinates $u = (i,j)$ in the style of a matrix (the first entry for the row, the second for the column). The \textit{hook length} $h_u$ is the number of boxes which are either below $u$ in the same column or to the right of $u$ in the same row, including $u$ itself, which can be computed as $h(i,j) = \lambda_i - i + \lambda^\ast_j - j - 1$. The \textit{content} $c_u$ is simply $c(i,j) = j - i$.

For example, here is the Young diagram of $(5 , 3 , 2)$ with the hook length in each box:
\begin{equation*}
	\ytableausetup{nobaseline,boxsize=1.5em}
	\begin{ytableau}
		7 & 6 & 4 & 2 & 1\\
		4 & 3 & 1\\
		2 & 1 
	\end{ytableau}
\end{equation*}
Similarly with the content in each box: 
\begin{equation*}
	\ytableausetup{nobaseline,boxsize=1.5em}
	\begin{ytableau}
		0 & 1 & 2 & 3 & 4\\
		-1 & 0 & 1\\
		-2 & -1 
	\end{ytableau}
\end{equation*}

Given partition $\lambda$, a \textit{subpartition} $\mu$ is a partition such that $\mu_i \leq \lambda_i$ for each part. In other words, the Young diagram of $\lambda$ covers the diagram of $\mu$ when aligned at their top-left corner. For such $\lambda,\mu$, the \textit{skew diagram} $\lambda / \mu$ is the diagram obtained by subtracting the boxes of $\mu$ from $\lambda$. For example, if $\lambda = (5 , 3 , 2)$ and $\mu = (4 , 1)$ then $\lambda / \mu$ looks like
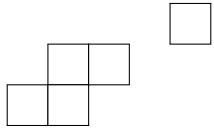
\begin{equation*}
	\ytableausetup{nobaseline,boxsize=1.5em}
	\begin{ytableau}
		\none & \none & \none & \none & *(white) \\
		\none & *(white) & *(white) \\
		*(white) & *(white)
	\end{ytableau}
\end{equation*}

\subsection{Symmetric Group}\label{sec:symmetric-group} 

We let $S_n$ be the symmetric group, i.e. the group of permutations on the set $\{ 1 , 2 , \ldots , n \}$. There is a natural group action of $S_n$ on itself given by conjugation $\sigma \cdot \tau = \sigma \tau \sigma^{-1}$. The orbits of this action are called conjugacy classes and are in natural bijection with partitions $\lambda \vdash n$. Namely, any $\sigma\in S_n$ has a decomposition into disjoint cycles $(i_1\cdots i_{\lambda_1})\cdots (j_1\cdots j_{\lambda_p})$, and the associated partition is the partition formed by the lengths of the cycles. An element $\sigma$ in the orbit labeled by $\lambda$ is said to have {\it cycle type} $\lambda$. For $\sigma$ of type $\lambda$, let $c_\lambda$ be the size of its orbit and $z_\lambda$ be the size of its stabilizer. The orbit-stabilizer theorem tells us that $c_\lambda z_\lambda = n!$, or
\begin{equation}\label{eqn:orb-stab}
	z_\lambda^{-1} = \frac{c_\lambda}{n!}
\end{equation}
as we will use it.

Alongside cycle type, the elements of $S_n$ can be distinguished by their \textit{sign} $\sign(\sigma)$ which is equal to $1$ (resp. $-1$) if $\sigma$ can be written as a product of an even (resp. odd) number of transpositions. An $\ell$-cycle can be written as a product of $\ell-1$ transpositions, namely
\begin{equation*}
	(i_1 \cdots i_\ell) = (i_1 i_\ell) (i_1 i_{\ell-1}) \cdots (i_1 i_3) (i_1 i_2)
\end{equation*}
so $\sign(i_1 \cdots i_\ell) = (-1)^{\ell-1}$. For general $\sigma \in S_n$, applying this to each disjoint cycle reveals that 
\begin{equation}\label{eqn:sign}
	\sign(\sigma) = (-1)^{n - c(\sigma)}
\end{equation}
where $c(\sigma)$ is the number of disjoint cycles in $\sigma$ (in fact $c(\sigma) = ht(\lambda)$ when $\sigma$ is of type $\lambda$).

Given a partition $\mathcal I$ of the set $\{1, \dots, n\}$ into disjoint subsets $\mathcal I =(I_1, \dots, I_p)$, the associated {\it Young subgroup} $S_\mathcal I$ is the subgroup of $S_n$ which leaves each of the subsets $I_j\subset \{1, \dots, n\}$ invariant. If $\alpha_j=|I_j|$ is the cardinality of $I_j$, then $S_{\mathcal I}$ is isomorphic to the group $S_{\alpha_1}\times \cdots \times S_{\alpha_p}$. We single out a distinguished family of Young subgroups as follows:

Given multi-index $\alpha$ with $|\alpha| = n$, let
\[
P_i = \left\{ 1 + \sum_{j=1}^{i-1} \alpha_j , 2 + \sum_{j=1}^{i-1} \alpha_j , \ldots , \alpha_i + \sum_{j=1}^{i-1} \alpha_j \right\} \; , \; i = 1 , \ldots , g
\]
so that $P_j$ contains the first $\alpha_j$ letters of $\{1 , \ldots , n\} \setminus \cup_{i=1}^{j-1} P_i$. Then $\mathcal{P} = (P_1 , \ldots , P_g)$ is a partition of $\{1 , \ldots , n\}$ and the Young subgroup of $\alpha$ is $S_\alpha \vcentcolon= S_{\mathcal{P}}$. Note that the cardinality of $S_\alpha$ is $|S_\alpha| = \alpha_1! \cdots \alpha_g! =\vcentcolon \alpha!$.

\subsection{Symmetric Polynomials}\label{sec:symmetric-polys}

A \textit{symmetric polynomial} in $d$ variables is a polynomial $f \in \mathbb{C}[x_1 , \ldots , x_d]$ that is invariant under permutation of variables. That is, $f(x_1 , \ldots , x_d) = f(x_{\sigma(1)} , \ldots , x_{\sigma(d)})$ for every permutation $\sigma \in S_d$. Given a matrix $A \in \matrixspace{d}$ with eigenvalues $a_1 , \ldots , a_d$ and symmetric polynomial $f$, we write (by abuse of notation) $f(A) \vcentcolon= f(a_1 , \ldots , a_d)$. 

Stanley \cite{Stanley-2024} and MacDonald \cite{Macdonald-2015} both give an extensive treatment of symmetric polynomials and symmetric functions. Here we shall recite only the results which are necessary for our work.

For $n \in \mathbb{N}$, the \textit{power sum symmetric polynomials} are defined as
\begin{equation*}
	p_n(x_1 , \ldots , x_d) \vcentcolon= x_1^n + \cdots + x_d^n
\end{equation*}

\noindent and for a partition $\lambda = (\lambda_1 \geq \cdots \geq \lambda_r)$, put
\begin{equation*}
	p_\lambda \vcentcolon= \prod_{i=1}^r p_{\lambda_i}.
\end{equation*}
In our notation, for a matrix $A$ we have $p_n(A) = \tr(A^n)$ and $p_\lambda(A) =\prod_{i=1}^p \tr(A^{\lambda_i})$.

The \textit{complete symmetric polynomial} $h_n$ is the sum of all monomials of degree $n$ (e.g. $h_2(x,y,z) = x^2 + y^2 + z^2 + xy + yz + xz$), and the \textit{elementary symmetric polynomial} $e_n$ is the sum of all monomials of degree $n$ with no repeated variables (e.g. $e_2(x,y,z) = xy + yz + xz$). These arise in the following products:
\begin{gather}
	\label{eqn:det-to-complete} \prod_{i=1}^{d} (1 - x_i t)^{-1} = \sum_{n=0}^{\infty} h_n t^n \\
	\label{eqn:det-to-elementary} \prod_{i=1}^{d} (1 + x_i t) = \sum_{n=0}^{d} e_n t^n .
\end{gather}
Equation~(\ref{eqn:det-to-elementary}) is obtained by distribution of terms, while (\ref{eqn:det-to-complete}) is formally obtained by writing out the geometric series for each $\frac{1}{1 - x_it}$ and taking the Cauchy product of formal power series.

We also need the relations
\begin{gather}
	\label{eqn:complete-to-power-sum} h_n = \sum_{\lambda \vdash n} z_\lambda^{-1} p_\lambda \\
	\label{eqn:elementary-to-power-sum} e_n = \sum_{\lambda \vdash n} \sign(\lambda) z_\lambda^{-1} p_\lambda 
\end{gather}
where $\sign(\lambda) \vcentcolon= \sign(\sigma)$ for any $\sigma \in S_n$ of cycle type $\lambda$ (see \cite[equation $(2.14^\prime)$]{Macdonald-2015}). Combining these equations yields the following lemma:

\begin{lem}\label{thm:cauchy-identities}
For any matrix $A \in \matrixspace{d}$,
\begin{equation}\label{eqn:det-to-power-sum}
	\det(I + A) = \sum_{n=0}^d \sum_{\lambda \vdash n} \sign(\lambda) z_\lambda^{-1} p_\lambda(A)
\end{equation}
Moreover, if the spectral radius $rad(A)<1$, then 
\begin{equation}\label{eqn:dual-cauchy}
  \det(I-A)^{-1} = \sum_{n=0}^\infty \sum_{\lambda \vdash n} z_\lambda^{-1} p_\lambda(A)
\end{equation}
where the series is absolutely convergent in the sense that
\[
\sum_{n=0}^\infty\left| \sum_{\lambda \vdash n} z_\lambda^{-1} p_\lambda(A)\right| <\infty.
\]

\end{lem}
\begin{proof} 
	
	Equations $(\ref{eqn:det-to-elementary})$ and $(\ref{eqn:elementary-to-power-sum})$ together with $t=1$ and $x_i$'s as the eigenvalues of $A$ immediately yield $(\ref{eqn:det-to-power-sum})$. Likewise, $(\ref{eqn:det-to-complete})$ and $(\ref{eqn:complete-to-power-sum})$ give $(\ref{eqn:dual-cauchy})$, but we must justify the absolute convergence.
	
	Writing the eigenvalues of $A$ as $a_1 , \ldots , a_d$, the condition $rad(A) < 1$ simply means $|a_i| < 1$ for $i = 1 , \ldots , d$. Therefore
	\begin{equation*}
		\frac{1}{1 - t a_i } = \sum_{k = 0}^{\infty} a_i^{k} t^{k}
	\end{equation*}
	converges absolutely for each $i$ (and say $|t| \leq 1$). Thus, the Cauchy product of these series also converges absolutely, and we have
	\begin{equation*}
		\det(I - tA)^{-1} = \prod_{i=1}^{d} \frac{1}{1 - t a_i} = \prod_{i=1}^{d} \left(\sum_{k= 0}^{\infty} a_i^{k} t^{k}\right) = \sum_{n=0}^{\infty} h_n(A) t^n
	\end{equation*}
	by definition of $h_n(A)$. Taking $t=1$ and applying (\ref{eqn:complete-to-power-sum}) finishes the proof.
	
\end{proof}

Perhaps the most important symmetric polynomials are the \textit{Schur polynomials}, denoted $s_\lambda(x_1 , \ldots , x_d)$ where $\lambda$ is a partition (of any number). These polynomials have a rich theory and many definitions, but we only need to know a few facts. First, the evaluation when each $x_i = 1$ is known from \cite[Corollary 7.21.4]{Stanley-2024} to be
\begin{equation}\label{eqn:schur-hook-and-content}
	s_{\lambda}(d) \vcentcolon= s_{\lambda}(1 , \ldots , 1) = \prod_{u \in \lambda} \frac{d + c_u}{h_u}.
\end{equation}
The index $u \in \lambda$ signifies that the product is taken over each box $u$ in the Young diagram of $\lambda$. Here $h_u$ and $c_u$ are the hook length and content as defined in Section~\ref{sec:partitions} above. 

Second, the Schur functions form a basis for the symmetric polynomials, so given partitions $\mu,\nu$ there exist coefficients $c^{\lambda}_{\mu,\nu}$ so that
\begin{equation}
	s_\mu s_\nu = \sum_{\lambda} c^{\lambda}_{\mu,\nu} s_\lambda
\end{equation}
as stated in \cite[equation (7.64)]{Stanley-2024}. These $c^{\lambda}_{\mu,\nu}$ are called the \textit{Littlewood-Richardson coefficients}.

\subsection{Representations}\label{sec:representations}

Much of what follows can be found in any introductory textbook on representation theory, such as \cite{Serre-1977} or \cite{Fulton-Harris-1991}.

A \textit{representation} of a group $G$ is a pair $(\rho , V)$ where $V$ is a finite-dimensional vector space and $\rho$ is a homomorphism $\rho : G \to \morph(V)$. The group multiplication becomes composition in $\morph(V)$, but this space also comes equipped with an addition operation. This leads one to consider the \textit{group algebra} $\mathbb{C}[G]$, which is a vector space whose basis is $G$ with multiplication of vectors defined by distribution and the group operation of $G$. Thus, any representation of $G$ determines (and is determined by) an algebra homomorphism $\rho : \mathbb{C}[G] \to \morph(V)$. 

One key representation is the (left) \textit{regular representation} given by $V = \mathbb{C}[G]$ and $\rho_R$ sending $s \in \mathbb{C}[G]$ to the map of left multiplication by $s$, which is determined by $\rho_R(s) : t \mapsto st$ for $t \in \mathbb{C}[G]$.

A subspace $W$ of a representation $(\rho, V)$ is said to be a subrepresentation if $W$ is invariant under $\rho(G)$ (i.e. $\rho(g) (w) \in W$ for all $g \in G$ and $w \in W$). A representation whose only subrepresentations are $\{0\}$ and itself is called an irreducible representation (irrep). When $G$ is a finite group, there are exactly as many irreps as there are conjugacy classes of $G$ (in particular, finitely many), although there is usually no canonical way to associate an irrep to a conjugacy class.

The {\it character} of a representation $\rho$ is the function
\begin{align*}
	\chi : G &\to \mathbb{C} \\
	g &\mapsto \tr(\rho(g)).
\end{align*}
Note that for any representation, $\chi(1_G) = \dim(V)$ since the identity of $G$ is mapped to the identity of $\morph(V)$. If $V_1 , \ldots , V_r$ are the irreps, with associated characters $\chi_1, \ldots , \chi_r$, any representation and character decompose as
\begin{gather*}
	V = m_1 V_1 \oplus \cdots \oplus m_r V_r \\
	\chi = \sum_{i=1}^{r} m_i \chi_i
\end{gather*}
where $m_iV_i$ is the direct sum of $V_i$ with itself $m_i$ times. Serre \cite{Serre-1977} refers to this as the \textit{canonical decomposition}. The canonical decomposition of the regular representation has $m_i = \dim(V_i)$, so every irrep occurs at least once in the regular representation.

In \cite[Theorem 8]{Serre-1977}, an equation for the projection $p_i$ of V onto $m_i V_i$ with respect to this decomposition is given:
\begin{equation}\label{eqn:representation-projection}
	p_i = \frac{\chi_i(1)}{|G|} \sum_{g \in G} \chi_i(g^{-1}) \rho(g) \in \morph(V).
\end{equation}
In particular this means $\sum_i p_i = 1$. We also have that $p_i$ commutes with $\rho(g)$ for each $g \in G$ \cite[Proposition 6]{Serre-1977}.

We are concerned with the representation theory of $S_n$, which is well understood (see \cite{Sagan-2001}). In fact, there is a natural correspondence between conjugacy classes and irreps, so we can label the irreps as $( V_{S_n}^{\lambda} , \rho_{S_n}^{\lambda} )$ for $\lambda \vdash n$. The character of each irrep will simply be denoted $\chi_\lambda$. Given $\lambda$, define $Q_\lambda \in \mathbb{C}[S_n]$ by
\begin{equation}\label{eqn:Q-formula}
	Q_\lambda\vcentcolon=\frac{\chi_\lambda(1)}{n!} \sum_{\sigma} \chi_\lambda(\sigma^{-1}) \sigma.
\end{equation}
The dimension $\chi_\lambda(1)$ is known \cite[Corollary 7.21.6]{Stanley-2024} to be
\begin{equation}\label{eqn:character-hook-length}
	\chi_\lambda(1) = \frac{n!}{\prod_{u \in \lambda} h_u}.
\end{equation}

If $(\rho, V)$ is any representation of $S_n$ canonically decomposed as $V = \bigoplus_{\lambda \vdash n} m_\lambda V_{S_n}^{\lambda}$, then by (\ref{eqn:representation-projection}) we see $\rho(Q_\lambda)$ is the projection of $V$ onto $m_\lambda V_{S_n}^{\lambda}$. By considering the regular representation, we learn
\begin{prop}\label{prop:properties-of-Q}
	Fix $\lambda, \mu \vdash n$ and $s \in \mathbb{C}[S_n]$. Then
	\begin{itemize}
		\setlength\itemsep{-0.5em}
		\item $Q_\lambda s = s Q_\lambda$, \\
		\item $\sum_{\lambda \vdash n} Q_\lambda = 1$, \\
		\item $Q_\lambda Q_\mu = \delta_{\lambda,\mu} Q_\lambda$.
	\end{itemize}
\end{prop}
\begin{proof}
	
	Let $\rho_R : \mathbb{C}[S_n] \to \morph(\mathbb{C}[S_n])$ be the regular representation. Then $\rho_R(Q_\lambda)$ is of the form (\ref{eqn:representation-projection}), so as mentioned previously, $\rho_R(Q_\lambda) \rho_R(\sigma) = \rho_R(\sigma) \rho_R(Q_\lambda)$ for all $\sigma \in S_n$, thus $\rho_R(Q_\lambda) \rho_R(s) = \rho_R(s) \rho_R(Q_\lambda)$. Now for any $t \in \mathbb{C}[S_n]$ we have
	\begin{equation*}
		Q_\lambda s t = \rho_R(Q_\lambda s)t = \rho_R(s Q_\lambda)t = sQ_\lambda t
	\end{equation*}
	and choosing $t = 1$ gives the first claim of the proposition. The other results also follow from applying the properties of projections to the regular representation and evaluating on $t=1$; the second property can be stated as ``the sum of all projections is the identity'' and the third property says ``each projection is idempotent and distinct projections are orthogonal.''
	
\end{proof}

Each $\lambda$ also induces an irrep on $U(d)$ -- written $( V_{U(d)}^{\lambda} , \rho_{U(d)}^{\lambda} )$ -- by \cite[Theorem 36.2]{Bump-2004}. Moreover, this theorem states that the character of this irrep is $U \mapsto s_\lambda(U)$ where $s_\lambda$ is a Schur polynomial. In particular, we have $\dim(V_{U(d)}^{\lambda}) = s_\lambda(d)$.

Next, $S_n$ and $U(d)$ can both be represented on $(\mathbb{C}^d)^{\otimes n}$ via
\begin{gather*}
	\rho_{S_n}^d : S_n \to \morph(\mathbb{C}^d)^{\otimes n} \\
	\rho_{S_n}^d(\sigma) : v_1 \otimes \cdots \otimes v_n \mapsto v_{\sigma^{-1}(1)} \otimes \cdots \otimes v_{\sigma^{-1}(n)}
\end{gather*}
and 
\begin{gather*}
	\rho_{U(d)}^n : U(d) \to \morph(\mathbb{C}^d)^{\otimes n} \\
	\rho_{U(d)}^n(U) : v_1 \otimes \cdots \otimes v_n \mapsto Uv_1 \otimes \cdots \otimes Uv_n.
\end{gather*}

To reduce notational burden, we shall often identify $\rho_{S_n}^d(\sigma) = \sigma$ unless emphasis is necessary. Additionally, because the matrix representation of $\rho_{U(d)}^n(U)$ with respect to the basis of $(\mathbb{C}^d)^{\otimes n}$ induced by the standard basis of $\mathbb{C}^d$ is exactly $U^{\otimes n}$, we write $\rho_{U(d)}^n(U) = U^{\otimes n}$.

These representations of $S_n$ and $U(d)$ commute with each other, giving rise to a natural representation of $\rho_{S_n \times U(d)} : (\sigma, U) \mapsto \sigma U^{\otimes n}$ which decomposes $(\mathbb{C}^d)^{\otimes n}$ as 
\begin{gather}\label{eqn:schur-weyl-duality}	
	(\mathbb{C}^d)^{\otimes n} = \bigoplus_{\substack{\lambda \vdash n \\ \ ht(\lambda) \leq d}} V_{S_n}^{\lambda} \otimes V_{U(d)}^{\lambda} 
\end{gather}
where $S_n \times U(d)$ acts by $\rho_{S_n}^{\lambda} \otimes \rho_{U(d)}^{\lambda}$ on $V_{S_n}^{\lambda} \otimes V_{U(d)}^{\lambda}$. This is known as Schur-Weyl duality, first proven by Weyl \cite{Weyl-1997}, though Bump \cite[Theorem 36.4]{Bump-2004} provides a modern treatment. 

As observed in \cite{Collins-Sniady-2006}, this structure implies that $\rho_{S_n}^d : \mathbb{C}[S_n] \to \morph(\mathbb{C}^d)^{\otimes n}$ is injective when restricted to the subalgebra 
\begin{equation*}
	\mathbb{C}_d[S_n] \vcentcolon= \left( \sum_{\substack{\lambda \vdash n \\ \ ht(\lambda) \leq d}} Q_\lambda \right) \mathbb{C}[S_n] .
\end{equation*}
In particular, when $n \leq d$, every $\lambda \vdash n$ satisfies $ht(\lambda) \leq d$, and since $\sum_{\lambda \vdash n} Q_\lambda = 1$, this means $\mathbb{C}_d[S_n] = \mathbb{C}[S_n]$. Hence $\rho_{S_n}^d$ is injective when $n \leq d$, which will be the case throughout Section~\ref{sec:scalar}.

Because $\rho_{S_n}^d(\sigma) = \rho_{S_n \times U(d)} (\sigma, I_d)$, we see from (\ref{eqn:schur-weyl-duality}) that
\begin{equation}\label{eqn:tensor-representation-decomp}
	(\mathbb{C}^d)^{\otimes n} = \bigoplus_{\lambda\vdash n, \ ht(\lambda)\leq d} s_\lambda(d) V_{S_n}^{\lambda}.
\end{equation}
is the canonical decomposition of $(\mathbb{C}^d)^{\otimes n}$ with respect to $\rho_{S_n}^d$. In terms of characters we get
\begin{equation}\label{eqn:character-of-permutation-representation-decomposed}
	\chi_{S_n}^{d} = \sum_{\lambda \vdash n , \ ht(\lambda) \leq d} s_\lambda(d) \chi_\lambda .
\end{equation}
In fact, $\chi_{S_n}^{d}$ can also be computed directly by \cite[Proposition 1.10]{Procesi-2021}
\begin{equation}\label{eqn:character-of-permutation-representation-direct}
	\chi_{S_n}^{d}(\sigma) = d^{c(\sigma)} .
\end{equation}
Again, $c(\sigma)$ is the number of disjoint cycles of $\sigma$.

To close out this section, we return to the Littlewood-Richardson coefficients. For any representation $\rho$ of a group $G$, we obtain a representation of any subgroup $H \leq G$ by restriction $\rho|_H$. Even when $\rho$ is an irrep of $G$, the restriction is typically not an irrep of $H$, so one wonders how $\rho|_H$ decomposes into $H$-irreps. In the case of $G = S_n$ and $H = S_\alpha$, the Littlewood-Richardson coefficients answer this question. 

When $\alpha = (\alpha_1 , \alpha_2)$, it is known \cite[Equation~(4.33)]{Sagan-2001} that
\begin{equation}\label{eqn:splitting-rule}
	\chi_{\lambda} (\gamma) = \sum_{\substack{\mu\vdash \alpha_1 \\ \nu\vdash\alpha_2}} c_{\mu\nu}^\lambda \ \chi_{\mu} (\gamma_{1}) \chi_{\nu} (\gamma_{2})
\end{equation}
where $\gamma \in S_\alpha$ is written $\gamma = \gamma_{1} \gamma_{2}$ for $\gamma_{i} \in S_{\alpha_i}$ $(i=1,2)$. One can derive an extended splitting rule for general $\alpha$ by repeated application of this base rule. For example when $\alpha = (\alpha_1 , \alpha_2 , \alpha_3)$, we see
\begin{equation*}
	\chi_\lambda = \sum_{ \substack{ \mu^1 \vdash \alpha_1 \\ \nu^1 \vdash \alpha_2 + \alpha_3 } } c^{\lambda}_{\mu^1,\nu^1} \chi_{\mu^1} \chi_{\nu^1} =	\sum_{ \substack{ \mu^1 \vdash \alpha_1 \\ \nu^1 \vdash \alpha_2 + \alpha_3 \\ \mu^2 \vdash \alpha_2 \\ \mu^3 \vdash \alpha_3 } } c^{\lambda}_{\mu^1,\nu^1} c^{\nu^1}_{\mu^2,\mu^3} \chi_{\mu^1} \chi_{\mu^2} \chi_{\mu^3}
\end{equation*}
with the first equality coming from restricting $S_n$ to $S_{\alpha_1} \times S_{\alpha_2+ \alpha_3}$ and then restricting further to $S_{\alpha_1} \times S_{\alpha_2} \times S_{\alpha_3} = S_\alpha$. Proceeding inductively for general $\alpha$ yields
\begin{equation}\label{eqn:splitting-rule-general}
	\chi_\lambda = \sum_{ \substack{ \mu^i \vdash \alpha_i \\ \nu^i \vdash \alpha_{i+1} + \cdots + \alpha_{g} } } c^{\lambda}_{\mu^1 , \nu^1} \left( \prod_{i=2}^{g-2} c^{\nu^{i-1}}_{\mu^i, \nu^i} \right) c^{\nu^{g-2}}_{\mu^{g-1},\mu^{g}} \prod_{i=1}^{g} \chi_{\mu^i}
\end{equation}
The $\mu^i$'s range over $i = 1 ,\ldots , g$ while the $\nu^i$'s range over $i = 1 , \ldots , g-2$.

\subsection{Tools from matrix analysis}\label{sec:matrix-analysis}

For any $m\times m$ matrix $T$, let $L_T$ (resp. $R_T$) denote the linear maps on $m \times \ell$ (resp. $\ell \times m$) matrices $L_TX = TX$ (resp. $R_TX$ =$XT$). Direct calculation reveals
\begin{equation}\label{eqn:left-right}
	\text{vec}(L_A(R_B(X))) = (B^T \otimes A) \text{vec}(X)
\end{equation}

\noindent where $\text{vec}$ is the column vectorization map. (That is, $\text{vec}(X)$ is the column vector formed by stacking the columns of $X$ vertically, beginning with the first column and continuing downward.) Under this map, the basis of matrix units of $M_n(\mathbb C)$ is mapped onto the standard basis of $\mathbb C^{n^2}$. Consequently, $B^T \otimes A$ is a matrix representation of the composed map $L_A \circ R_B$ in the standard basis.

Two other identities we shall use occasionally are
\begin{gather*}
	\tr(A \otimes B) = \tr(A)\tr(B) \\
	(A \otimes B) \circ (C \otimes D) = (A \circ C) \otimes (B \circ D)
\end{gather*}
for matrices $A,B,C,D$ of compatible sizes. We are using $\circ$ to emphasize matrix multiplication/composition of linear maps.

The following proposition is a rephrasing of Procesi \cite{Procesi-2021} Proposition 1.5 (or Kostant \cite{Kostant-1958} Lemma 4.9) which we include for easy reference:

\begin{prop}\label{thm:kostant}
	
	Take $X_1 , \ldots , X_n \in \matrixspace{d}$. If $\sigma = (i_1 \, i_2 \, \ldots \, i_h) \ldots (j_1 \, j_2 \, \ldots \, j_\ell)$ is the disjoint cycle decomposition of $\sigma \in S_n$, then
	\begin{equation*}
		\tr( X_{i_1} X_{i_2} \cdots X_{i_h} ) \cdots \tr (X_{j_1} X_{j_2} \cdots X_{j_\ell} ) = \tr( \rho_{S_n}^d(\sigma^{-1}) \circ (X_1 \otimes X_2 \otimes \cdots \otimes X_n) ).
	\end{equation*}
	
\end{prop}

\noindent In particular, when all $X_i = A \in \matrixspace{d}$ this equation becomes
\begin{equation}\label{eqn:kostant-corollary}
	p_\lambda(A) = \tr( \sigma^{-1} \circ A^{\otimes n} )
\end{equation}
where $\lambda$ is the cycle type of $\sigma$.

From this we obtain an extremely important (for us) consequence, which is that we can rewrite the identities in Lemma~\ref{thm:cauchy-identities} in terms of sums over the symmetric group rather than sums over partitions. In particular we have
\begin{lem}\label{lem:cauchy-identities-redux}
  For any matrix $A \in \matrixspace{d}$,
\begin{equation}\label{eqn:det-to-power-sum-redux}
	\det(I + A) = \sum_{n=0}^d \frac{1}{n!} \sum_{\sigma\in S_n} \sign(\sigma) \tr(\rho_{S_n}^d(\sigma^{-1}) \circ A^{\otimes n}) 
\end{equation}
Moreover, if $rad(A)<1$, then 
\begin{equation}\label{eqn:dual-cauchy-redux}
  \det(I-A)^{-1} = \sum_{n=0}^\infty \frac{1}{n!} \sum_{\sigma\in S_n} \tr(\rho_{S_n}^d(\sigma^{-1}) \circ A^{\otimes n}) 
\end{equation}
where the series is absolutely convergent in the sense that
\begin{equation*}
	\sum_{n=0}^\infty\left| \sum_{\sigma\in S_n} \tr(\rho_{S_n}^d(\sigma^{-1}) \circ A^{\otimes n})  \right| <\infty.
\end{equation*}

\end{lem}
\begin{proof}
	
	By abuse of notation write $\sigma \in \lambda$ to mean that $\sigma \in S_n$ has cycle type $\lambda \vdash n$. Since $S_n$ can be partitioned into cycle types, we see
	
	\begin{align*}
		\frac{1}{n!} \sum_{\sigma\in S_n} \sign(\sigma) \tr( \sigma^{-1} \circ A^{\otimes n}) &= \frac{1}{n!} \sum_{\lambda \vdash n} \sum_{\sigma \in \lambda} \sign(\sigma) \tr( \sigma^{-1} \circ A^{\otimes n}) \\
		&\overset{(\ref{eqn:kostant-corollary})}{=} \frac{1}{n!} \sum_{\lambda \vdash n} \sum_{\sigma \in \lambda} \sign(\lambda) p_\lambda(A) \\
		&= \sum_{\lambda \vdash n} \frac{c_\lambda}{n!} \sign(\lambda) p_\lambda(A) \\
		&\overset{(\ref{eqn:orb-stab})}{=} \sum_{\lambda \vdash n} z_\lambda^{-1} \sign(\lambda) p_\lambda(A) 
	\end{align*}
	and similarly
	\begin{equation*}
		\frac{1}{n!} \sum_{\sigma\in S_n} \tr( \sigma^{-1} \circ A^{\otimes n}) = \sum_{\lambda \vdash n} z_\lambda^{-1} p_\lambda(A) 
	\end{equation*}
	Substitute these identities into Lemma~$\ref{thm:cauchy-identities}$ to finish the proof.
	
\end{proof}

\subsubsection{Outer spectral radius and row contractions} For a $g$-tuple of $k\times k$ matrices $\mathcal{X}=(X_1, \dots, X_g)$, we defined the {\it outer spectral radius} of $\mathcal{X}$ to be
\begin{equation*}
  {\textbf{rad}}(\mathcal{X})\vcentcolon= (rad(\sum_{j=1}^g X_j\otimes \overline{X_j}))^{1/2}
\end{equation*}
where $rad$ on the right hand side denotes the ordinary spectral radius (the maximum modulus of eigenvalues of a matrix). In fact ${\textbf{rad}}(\mathcal{X})$ is the square root of the spectral radius of the linear mapping $T\to \sum_{j=1}^g X_j TX_j^*$.

We define the {\it row norm} of a matrix tuple $\mathcal{X}$ by $\|\mathcal X\|_{row} = {\|X_1X_1^* +\cdots +X_gX_g^*\|}^{1/2}$. We say that $\mathcal X$ is a {\it row contraction} if $\|\mathcal X\|_{row}<1$. We say that two tuples $\mathcal{X}, \mathcal{Y}$ are {\it jointly similar} if there is an invertible matrix $S$ such that $SX_jS^{-1} = Y_j$ for each $j=1, \dots, g$. We have the following important Rota-style theorem relating the outer spectral radius to row contraction (this is a special case of \cite[Theorem 3.8]{Popescu-2014}, see also \cite[Theorem 1.9]{Pascoe-2021} for a concrete treatment in the case of matrices):
\begin{prop}\label{prop:rota-for-outer}
  A tuple $\mathcal{X}$ is jointly similar to a row contraction if and only if ${\textbf{rad}}(\mathcal{X})<1$.
\end{prop}
As a consequence also have a kind of Cauchy-Schwarz inequality for the outer spectral radius:
\begin{prop}\label{prop:CS-for-outer}
  For any $g$-tuples of square matrices $\mathcal X, \mathcal Y$ (possibly of different sizes), we have
  \begin{equation}
    rad\left(\sum_{j=1}^g X_j\otimes \overline{Y_j}\right)\leq {\textbf{rad}}(\mathcal{X}) {\textbf{rad}}(\mathcal{Y})
  \end{equation}
\end{prop}
\begin{proof}
Switching the order of the tensor factors (which does not change the spectral radius), the matrix $\sum_{j=1}^g \overline{Y_j} \otimes X_j$ is the matrix of the linear transformation $\Psi:T\to \sum_{j=1}^gX_jTY_j^*$. By homogeneity it suffices to assume ${\textbf{rad}}(\mathcal{X}),  {\textbf{rad}}(\mathcal{Y})<1$ and prove that the left hand side is strictly less than $1$. But then Proposition~\ref{prop:rota-for-outer} lets us assume, by applying appropriate similarities, that $\mathcal X$ and $\mathcal Y$ are (strict) row contractions, so that the map $\Psi$ is a strict contraction, hence $rad(\Psi)<1$. 
\end{proof}

\section{General considerations and formal convergence}\label{sec:formal}

We recall our conjecture from the introduction: 
\begin{conj}
  Let $\mathcal X, \mathcal Y$ be $g$-tuples of square matrices of size $k\times k$, $k^\prime \times k^\prime$ respectively, with $\textbf{rad}(\mathcal X)$, $\textbf{rad}(\mathcal Y)<1$. Then
  \begin{equation}\label{eqn:pencil-limit-conjecture}
    \lim_{d\to \infty} \int_{U(d)^g} \det (L_{\mathcal X}(\mathcal U))\overline{\det (L_{\mathcal Y}(\mathcal U))}\, d\mathcal U = \det\left( I_k\otimes I_{k^\prime} -\sum_{j=1}^g X_j\otimes \overline{Y_j}\right)^{-1}.
  \end{equation}
\end{conj}

In this section we show that the conjecture is formally true, meaning that the limit exists in the sense of formal power series. Subsequent sections deal with estimates for special cases of $\mathcal{X}, \mathcal{Y}$ which will allow us to prove convergence rigorously. We begin by computing homogeneous expansions of the integral and the target limit in (\ref{eqn:pencil-limit-conjecture}). In light of (\ref{eqn:kostant-corollary}), define
\begin{equation}\label{eqn:general-power-sum-poly}
	p_{\sigma, \alpha}(\mathcal{X}) \vcentcolon= \tr (\rho_{S_n}^{k}(\sigma^{-1}) \circ \mathcal{X}^{\otimes \alpha}).
\end{equation}
This allows us to state the following lemma more concisely.  (We remark that, by a famous theorem of Procesi \cite[Theorem 1.3]{Procesi-1976}, these ``trace monomials'' (as one varies $\sigma$ and $\alpha$) generate the ring of $GL(d)$-invariants of matrix tuples $\mathcal X$.)

\begin{lem}\label{lem:pencil-expansion-matrix}
	
	For $\mathcal{X} \in \matrixspace{k}^g$ and $\mathcal{U} \in \matrixspace{d}^g$, we have
	
	\begin{equation}\label{eqn:pencil-expansion-matrix}
	  \det ( L_\mathcal{X}( \mathcal{U} ) ) =  \sum_{n=0}^{kd }\frac{1}{n!} \sum_{|\alpha| = n} \binom{n}{\alpha} \sum_{\sigma \in S_n} \sign(\sigma) p_{\sigma, \alpha}(\mathcal{X})p_{\sigma, \alpha}(\mathcal{U})          
	\end{equation}
	
\end{lem}
\begin{proof}
	
	We put $A = \sum_{j=1}^g X_j\otimes U_j\in M_{kd\times kd}$ in Lemma~\ref{lem:cauchy-identities-redux} to obtain
	\begin{equation}
		\det ( L_\mathcal{X}( \mathcal{U} ) ) = \sum_{n=0}^{kd} \frac{1}{n!} \sum_{\sigma\in S_n} \sign(\sigma)\tr \left(\rho_{S_n}^{kd}(\sigma^{-1}) \circ \left(\sum_{j=1}^g X_j\otimes U_j\right)^{\otimes n}\right)
	\end{equation}

	Temporarily set $A_j = X_j \otimes U_j$. For \textit{any} $g$-tuple of matrices $\mathcal{A} = (A_1 , \ldots , A_g)$, the $n$-fold tensor product can be expanded as
	\begin{equation*}
		\left( \sum_{j=1}^g A_j \right)^{\otimes n} = \sum_{w \in \mathbb{W}_{g}(n)} \mathcal{A}^{\otimes w}.
	\end{equation*}
  
	We also observe that
	\begin{equation}\label{eqn:word-tensor-conjugate}
	 	\tau \mathcal{A}^{\otimes w} \tau^{-1} = \mathcal{A}^{\otimes \tau(w)}.
	\end{equation}
	
	\noindent This follows directly from the definitions of these operators. Since any word $w$ of letter count $\alpha$ is a permutation of $w_\alpha$ (i.e. $w = \tau_w(w_\alpha)$ for some $\tau_w \in S_n$), we proceed with simplifying:
	\begin{align*}
	 	\det \left( I + \sum_{j=1}^{g} A_j \right) &= \sum_{n=0}^{kd} \frac{1}{n!} \sum_{\sigma\in S_n} \sign(\sigma)\tr \left( \sigma^{-1} \circ \left(\sum_{j=1}^g A_j\right)^{\otimes n}\right) \\
	 	&= \sum_{n=0}^{kd} \frac{1}{n!} \sum_{\sigma\in S_n} \sum_{w \in \mathbb{W}_{g}(n)} \sign(\sigma)\tr \left( \sigma^{-1} \circ \mathcal{A}^{\otimes w} \right) \\
	 	&= \sum_{n=0}^{kd} \frac{1}{n!} \sum_{|\alpha| = n} \sum_{w \in \alpha} \sum_{\sigma\in S_n} \sign(\sigma)\tr \left( \sigma^{-1} \circ \mathcal{A}^{\otimes w} \right) \\
	 	&= \sum_{n=0}^{kd} \frac{1}{n!} \sum_{|\alpha| = n} \sum_{w \in \alpha} \sum_{\sigma\in S_n} \sign(\sigma)\tr \left( \sigma^{-1} \circ \mathcal{A}^{\otimes \tau_w(w_\alpha)} \right) \\
	 	&= \sum_{n=0}^{kd} \frac{1}{n!} \sum_{|\alpha| = n} \sum_{w \in \alpha} \sum_{\sigma\in S_n} \sign(\sigma)\tr \left( \sigma^{-1} \circ \tau_w \mathcal{A}^{\otimes \alpha} \tau_w^{-1} \right) \\
	 	&\overset{(\ast)}{=} \sum_{n=0}^{kd} \frac{1}{n!} \sum_{|\alpha| = n} \sum_{w \in \alpha} \sum_{\sigma\in S_n} \sign(\sigma)\tr \left( \sigma^{-1} \circ \mathcal{A}^{\otimes \alpha} \right) \\
	 	&= \sum_{n=0}^{kd} \frac{1}{n!} \sum_{|\alpha| = n} \binom{n}{\alpha} \sum_{\sigma\in S_n} \sign(\sigma) \tr( \rho_{S_n}^{kd} (\sigma^{-1}) \circ \mathcal{A}^{\otimes \alpha} ) .
	\end{align*}
	The equality $(\ast)$ holds by cyclicity of trace and the change of variables $\sigma \mapsto \tau_w \sigma \tau_w^{-1}$. In summary,
	\begin{equation}\label{eqn:pencil-expansion-step}
		\det ( L_\mathcal{X}( \mathcal{U} ) ) = \sum_{n=0}^{kd} \frac{1}{n!} \sum_{|\alpha| = n} \binom{n}{\alpha} \sum_{\sigma \in S_n} \sign(\sigma) \; \tr( \rho_{S_n}^{kd}(\sigma^{-1}) \circ ( \mathcal{X} \otimes \mathcal{U} )^{\otimes \alpha} ).
	\end{equation}
	where $\mathcal{X} \otimes \mathcal{U} \vcentcolon= (X_1 \otimes U_1 , \ldots , X_g \otimes U_g)$.
	
	The operators $( \mathcal{X} \otimes \mathcal{U} )^{\otimes \alpha}$ and $\rho_{S_n}^{kd}(\sigma)$ both act on $( \mathbb{C}^k \otimes \mathbb{C}^d )^{\otimes n}$. An elementary tensor in this space has the form
	\begin{equation*}
		v_1^{(k)} \otimes v_1^{(d)} \otimes \cdots \otimes v_n^{(k)} \otimes v_n^{(d)} = \bigotimes_{j=1}^{n} ( v_j^{(k)} \otimes v_j^{(d)})
	\end{equation*}
	where $v_j^{(k)} \in \mathbb{C}^k$ and $v_j^{(d)} \in \mathbb{C}^d$. The map $S : ( \mathbb{C}^k \otimes \mathbb{C}^d )^{\otimes n} \to (\mathbb{C}^k)^{\otimes n} \otimes (\mathbb{C}^d)^{\otimes n}$ defined by
	\begin{equation*}
		S : \bigotimes_{i=1}^{n} ( v_i^{(k)} \otimes v_i^{(d)}) \mapsto \left( \bigotimes_{i=1}^{n} v_i^{(k)} \right) \otimes \left( \bigotimes_{i=1}^{n} v_i^{(d)} \right)
	\end{equation*}
	is invertible (as is any permutation of tensor factors). Moreover, one can directly verify that $S \circ ( \mathcal{X} \otimes \mathcal{U} )^{\otimes \alpha} \circ S^{-1} = \mathcal{X}^{\otimes \alpha} \otimes \mathcal{U}^{\otimes \alpha}$ and $S \circ \rho_{S_n}^{kd}(\sigma) \circ S^{-1} = \rho_{S_n}^{k}(\sigma) \otimes \rho_{S_n}^{d}(\sigma)$.
	
	Thus
	\begin{align*}
		\tr( \rho_{S_n}^{kd}(\sigma^{-1}) \circ ( \mathcal{X} \otimes \mathcal{U} )^{\otimes \alpha} ) &= \tr( (\rho_{S_n}^{k}(\sigma^{-1}) \otimes \rho_{S_n}^{d}(\sigma^{-1})) \circ (\mathcal{X}^{\otimes \alpha} \otimes \mathcal{U}^{\otimes \alpha}) ) \\
		&= \tr( \rho_{S_n}^k(\sigma^{-1}) \circ \mathcal{X}^{\otimes \alpha} ) \; \tr( \rho_{S_n}^d(\sigma^{-1}) \circ \mathcal{U}^{\otimes \alpha} ) \\
		&= p_{\sigma, \alpha}(\mathcal{X}) p_{\sigma, \alpha}(\mathcal{U}).
	\end{align*}
	since trace is conjugation-invariant. Substituting this into (\ref{eqn:pencil-expansion-step}) completes the proof.
\end{proof}

\begin{lem}\label{lem:target-expansion} Let $\mathcal{X}$, $\mathcal{Y}$ be $g$ tuples of square matrices with ${\textbf{rad}}(\mathcal X), {\textbf{rad}}(\mathcal Y)<1$. Then 
  \begin{equation}\label{eqn:target-expansion}
  \det \left(I\otimes I -\sum_{j=1}^g X_j\otimes \overline{Y_j}\right)^{-1} = \sum_{n=0}^\infty \frac{1}{n!} \sum_{|\alpha|=n} \binom{n}{\alpha} \sum_{\sigma\in S_n} p_{\sigma, \alpha}(\mathcal{X})\overline{p_{\sigma, \alpha}(\mathcal{Y)}}
  \end{equation}
  and the series on the right is absolutely convergent in the sense that
  \[
  \sum_{n=0}^\infty \frac{1}{n!} \left| \sum_{|\alpha|=n} \binom{n}{\alpha} \sum_{\sigma\in S_n} p_{\sigma, \alpha}(\mathcal{X})\overline{p_{\sigma, \alpha}(\mathcal{Y)}}\right|<\infty
  \]
  
\end{lem}
\begin{proof}
  Since each of $\mathcal{X}, \mathcal{Y}$ has outer spectral radius less than $1$, it follows that $rad(\sum_{j=1}^g X_j\otimes \overline{Y_j})<1$ by Proposition~\ref{prop:CS-for-outer}. Applying (\ref{eqn:dual-cauchy-redux}) with $A = \sum_{j=1}^g X_j\otimes\overline{Y_j}$ we obtain
  \[
  \det \left(I\otimes I -\sum_{j=1}^g X_j\otimes \overline{Y_j}\right)^{-1} = \sum_{n=0}^\infty \frac{1}{n!} \sum_{\sigma\in S_n} \tr\left(\rho_{S_n}(\sigma^{-1})\circ \left(\sum_{j=1}^g X_j\otimes \overline{Y_j}\right)^{\otimes n}\right).
  \]
 which is absolutely convergent as in Lemma~\ref{lem:cauchy-identities-redux}. The identity
  \[
  \sum_{\sigma\in S_n} \tr\left(\rho_{S_n}(\sigma^{-1})\circ \left(\sum_{j=1}^g X_j\otimes \overline{Y_j}\right)^{\otimes n}\right) = \sum_{|\alpha|=n} \binom{n}{\alpha}\sum_{\sigma\in S_n} p_{\sigma, \alpha}(\mathcal{X})\overline{p_{\sigma,\alpha}(\mathcal{Y})}
  \]
  follows as in the proof of Lemma~\ref{lem:pencil-expansion-matrix}.
\end{proof}

So far we have only written out the determinants found in $(\ref{eqn:pencil-limit-conjecture})$ in terms of the trace monomials $p_{\sigma,\alpha}$. Next, we apply the Lemma~\ref{lem:pencil-expansion-matrix} to our integral expression in the left hand side of $(\ref{eqn:pencil-limit-conjecture})$ to obtain something of the form
\begin{align*}
  \int_{U(d)^g} & \det( L_\mathcal{X}(\mathcal{U}) ) \overline{ \det( L_\mathcal{Y}(\mathcal{U}) ) } \, d\mathcal{U} = \\
  &= \sum_{n=0}^{kd} \sum_{m=0}^{k'd} \sum_{|\alpha| = n} \sum_{|\beta| = m} \sum_{\sigma \in S_n} \sum_{\tau \in S_m} ( \cdots ) \int_{ U(d)^g } p_{\sigma,\alpha}(\mathcal{U}) \overline{ p_{\tau,\beta}(\mathcal{U}) } \, d\mathcal{U}
\end{align*}

Basic symmetry of the measure $d\mathcal{U}$ reveals that the integrals $\int p_{\sigma,\alpha}(\mathcal{U}) \overline{ p_{\tau,\beta}(\mathcal{U}) }$ vanish unless $\alpha = \beta$. Indeed, in one variable the Haar measure $dU$ is invariant to the change of coordinates $U \mapsto zU$ for a unimodular complex number $z$. So if some $\alpha_j \neq \beta_j$, we apply such a change in the $j^{th}$ coordinate to obtain
\begin{equation*}
	\int_{ U(d)^g } p_{\sigma,\alpha}(\mathcal{U}) \overline{ p_{\tau,\beta}(\mathcal{U}) } \, d\mathcal{U} = z^{\alpha_j - \beta_j} \int_{ U(d)^g } p_{\sigma,\alpha}(\mathcal{U}) \overline{ p_{\tau,\beta}(\mathcal{U}) } \, d\mathcal{U}.
\end{equation*}
The fact that this is true for all unimodular $z$ means the integral vanishes, and so the earlier expansion reduces to
\begin{gather}\label{eqn:expanded-integral}
	\int_{U(d)^g} \det( L_\mathcal{X}(\mathcal{U}) ) \overline{ \det( L_\mathcal{Y}(\mathcal{U}) ) } \, d\mathcal{U}  \\
	= \sum_{n=0}^{ \min ( kd , k'd ) } \frac{1}{(n!)^2} \sum_{ |\alpha| = n } \binom{n}{\alpha}^2  \sum_{ \sigma , \tau \in S_n } \sign(\sigma) \sign(\tau) p_{\sigma,\alpha}(\mathcal{X}) \overline{ p_{\tau,\alpha}(\mathcal{Y}) } \int_{ U(d)^g } p_{\sigma,\alpha}(\mathcal{U}) \overline{ p_{\tau,\alpha}(\mathcal{U}) } \, d\mathcal{U} . \nonumber
\end{gather}

It is clear that in order to proceed, we should try to understand the correlations between the generalized power sum polynomials $p_{\sigma, \alpha}(\mathcal U)$, that is, we wish to understand the integrals
\begin{equation*}
	\int_{ U(d)^g } p_{\sigma,\alpha}(\mathcal{U}) \overline{ p_{\tau,\alpha}(\mathcal{U}) } \, d\mathcal{U}.
\end{equation*}
It turns out that these expressions can become quite complicated at finite sizes $d$, but the in the limit we have useful orthogonality relations. These were discovered independently (in quite different contexts from ours) in \cite{Radulescu-2006} and \cite{Mingo-Sniady-Speicher-2007}  More recently these numbers have also appeared in a topological context \cite{Magee-Puder-2019}.  The limit turns out to be either $0$, or a positive integer with a helpful group-theoretic interpretation.

The next lemma re-expresses the result from \cite[Theorem 4.1]{Radulescu-2006} in a way that is useful to us:

\begin{thm}\label{thm:asymptotic-orthogonality}
	
	Let $\orb_\alpha(\sigma)$ and $\stab_\alpha(\sigma)$ be the orbit and stabilizer of $\sigma \in S_n$ under the action of conjugation by $S_\alpha$. Then
	\begin{equation}\label{eqn:asymptotic-orthogonality}
		\lim_{d \to \infty} \int_{ U(d)^g } p_{\sigma,\alpha}(\mathcal{U}) \overline{ p_{\tau,\alpha}(\mathcal{U}) } \, d\mathcal{U} = {\bf 1}_{\orb_\alpha(\sigma)}(\tau) | \stab_\alpha(\sigma) | .
	\end{equation}
	\noindent (with ${\bf1}_E$ being the indicator function of a set $E$.) In other words, the limit is nonzero if and only if there exists $\gamma \in S_\alpha$ such that $\sigma = \gamma \tau \gamma^{-1}$, in which case the value is the number of $\gamma \in S_\alpha$ such that $\sigma = \gamma \sigma \gamma^{-1}$.
	
\end{thm}

\begin{proof}
	
	For some word $w_i \in \mathbb{W}_{g}(n)$ let $W_i$ be the product of matrices obtained by replacing each letter $\ell_j$ of $w_i$ with the corresponding $U_j$. R\u{a}dalescu's Theorem 4.1 \cite{Radulescu-2006} proves that
	\begin{equation}\label{eqn:radalescu-limit}
		\lim_{d \to \infty} \int_{U(d)^g} \prod_{i=1}^{p} \tr( W_i )^{b_i} \overline{ \tr( W_i )^{c_i} } \, d\mathcal{U}
	\end{equation}
	is non-zero if and only if all $b_i = c_i$, in which case the non-zero value of this limit is $b_1! \cdots b_p! j_1^{b_1} \cdots j_p^{b_p}$ where $j_i = j(W_i)$ is the number of cyclic rotations of $W_i$ that leave $W_i$ invariant.

	Proposition~\ref{thm:kostant} tells us
	\begin{equation*}
		p_{\sigma,\alpha}(\mathcal{U}) = \prod_{i=1}^{p} \tr( W_i )^{b_i} \;,\; p_{\tau,\alpha}(\mathcal{U}) = \prod_{i=1}^{p} \tr( W_i )^{c_i}
	\end{equation*}
	for some words $w_i$ and non-negative integers $b_i,c_i$.
	
	Recalling the partition $\mathcal{P} = (P_1 , \ldots , P_g)$ used to define $S_\alpha$, we obtain the above trace products by taking the disjoint cycle form of $\sigma$ (and $\tau$) and replacing all the elements of $P_j$ with the matrix $U_j$, then multiplying and taking traces. Since this only depends on $\mathcal{P}$ -- rather than each individual element of $\{1 , 2 , \ldots , n\}$ -- it follows that $p_{\sigma , \alpha}(\mathcal{U}) = p_{\tau, \alpha}(\mathcal{U})$ (i.e. all $b_i = c_i$) if and only if $\sigma \in \orb_\alpha(\tau)$.

	To see how $b_1! \cdots b_p! j_1^{b_1} \cdots j_p^{b_p} = | \stab_\alpha(\sigma) |$, one need only imitate the derivation of the value of $z_\lambda$ \cite[Proposition 1.3.2]{Stanley-2012}. As an outline, fix a disjoint cycle form of $\sigma$, conjugate by every $\gamma \in S_\alpha$ to produce $\alpha!$ disjoint cycle forms, then consider which of these are equivalent via commuting different cycles and rotating each cycle to get the number of distinct forms. Because $\gamma \in S_\alpha$, it follows that $$| \orb_\alpha(\sigma)| = \alpha! / (b_1! \cdots b_p! j_1^{b_1} \cdots j_p^{b_p}),$$ so $b_1! \cdots b_p! j_1^{b_1} \cdots j_p^{b_p} = | \stab_\alpha(\sigma) |$ by the orbit-stabilizer theorem.
	
\end{proof}
We can now establish the formal convergence: the next corollary states that as $d\to \infty$, for each $n$ the $n^{th}$ homogeneous term in the expansion (\ref{eqn:expanded-integral}) converges to the corresponding $n^{th}$ degree term in (\ref{eqn:target-expansion}).
\begin{cor}\label{cor:termwise-limit} Let $\mathcal{X}$ and $\mathcal{Y}$ be $g$-tuples of square matrices. For each nonnegative integer $n$ we have
	\begin{align*}
		\lim_{d\to \infty} \; \frac{1}{(n!)^2} &\sum_{ |\alpha| = n } \binom{n}{\alpha}^2  \sum_{ \sigma , \tau \in S_n } \sign(\sigma)\sign(\tau) p_{\sigma,\alpha}(\mathcal{X}) \overline{ p_{\tau,\alpha}(\mathcal{Y}) } \int_{ U(d)^g } p_{\sigma,\alpha}(\mathcal{U}) \overline{ p_{\tau,\alpha}(\mathcal{U}) }  d\mathcal{U} \\
		= \frac{1}{n!} &\sum_{ |\alpha| = n } \binom{n}{\alpha} \sum_{ \sigma \in S_n } p_{\sigma,\alpha}(\mathcal{X}) \overline{ p_{\sigma,\alpha}(\mathcal{Y})}
  \end{align*}
  
\end{cor}
\begin{proof}
  We have by Theorem~\ref{thm:asymptotic-orthogonality}
  \begin{align*}
    \lim_{d\to \infty} \; &\frac{1}{(n!)^2} \sum_{ |\alpha| = n } \binom{n}{\alpha}^2  \sum_{ \sigma , \tau \in S_n } \sign(\sigma)\sign(\tau) p_{\sigma,\alpha}(\mathcal{X}) \overline{ p_{\tau,\alpha}(\mathcal{Y}) } \int_{ U(d)^g } p_{\sigma,\alpha}(\mathcal{U}) \overline{ p_{\tau,\alpha}(\mathcal{U}) }  d\mathcal{U} \\
    &= \frac{1}{(n!)^2} \sum_{ |\alpha| = n } \binom{n}{\alpha}^2 \sum_{ \sigma \in S_n } \sum_{ \tau \in S_n } p_{\sigma,\alpha}(\mathcal{X}) \overline{ p_{\tau,\alpha}(\mathcal{Y}) } {\bf 1}_{\orb_\alpha(\sigma)}(\tau) | \stab_\alpha(\sigma) | \\
    &\overset{(\ast)}{=} \frac{1}{(n!)^2} \sum_{ |\alpha| = n } \binom{n}{\alpha}^2 \sum_{ \sigma \in S_n } \sum_{ \tau \in \text{orb}_\alpha(\sigma) } p_{\sigma,\alpha}(\mathcal{X}) \overline{ p_{\sigma,\alpha}(\mathcal{Y}) } | \stab_\alpha(\sigma) | \\
    &= \frac{1}{(n!)^2} \sum_{ |\alpha| = n } \binom{n}{\alpha}^2 \sum_{ \sigma \in S_n } |\text{orb}_\alpha(\sigma)| |\text{stab}_\alpha(\sigma)| p_{\sigma,\alpha}(\mathcal{X}) \overline{ p_{\sigma,\alpha}(\mathcal{Y}) } \\
    &= \frac{1}{n!} \sum_{ |\alpha| = n } \binom{n}{\alpha} \sum_{ \sigma \in S_n } \left[ \frac{1}{n!} \binom{n}{\alpha} | S_\alpha | \right] p_{\sigma,\alpha}(\mathcal{X}) \overline{ p_{\sigma,\alpha}(\mathcal{Y}) } \\
    &= \frac{1}{n!} \sum_{ |\alpha| = n } \binom{n}{\alpha} \sum_{ \sigma \in S_n } p_{\sigma,\alpha}(\mathcal{X}) \overline{ p_{\sigma,\alpha}(\mathcal{Y}) } .
  \end{align*}
 \noindent Note in $(\ast)$ we are using the fact that $\tau \in \text{orb}_\alpha(\sigma)$ implies $p_{\sigma , \alpha}(\mathcal{Y}) = p_{\tau, \alpha}(\mathcal{Y})$.
\end{proof} 

We conclude this section by showing that the formal convergence can be made rigorous provided we have adequate bounds on the integral, independent of $d$. We will make use of this idea in proving special cases of the conjecture in later sections. 

\begin{prop}\label{prop:L2-bound-is-sufficient}
  Suppose that for each $0\leq r<1$ and each integer $k\geq 1$ there exists a number $C(r,k)$ with the following property: for all systems of $k\times k$ matrices $\mathcal{X}=(X_1, \dots, X_g)$ satisfying $\|\mathcal{X}\|_{row}\leq r$, we have
  \begin{equation}\label{eqn:local-L2-bound}
    \sup_{d\geq 1} \int_{U(d)^g} \left|\det( L_\mathcal{X}(\mathcal{U}) )\right|^2\, d\mathcal{U} \leq C(r,k).
  \end{equation}
Then Conjecture~\ref{eqn:pencil-limit-conjecture} holds. 
\end{prop}

\begin{proof}  In Conjecture~\ref{eqn:pencil-limit-conjecture} we are assuming the tuple $\mathcal X, \mathcal Y$ have outer spectral radius strictly less than $1$, though by Proposition~\ref{prop:rota-for-outer} it suffices to prove the conjecture under the stronger hypothesis $\|\mathcal X\|_{row}, \|\mathcal Y\|_{row}<1$, which we impose from now on. We define $\Omega_k\vcentcolon=\{\mathcal X\in M^g_{k\times k} : \|\mathcal X\|_{row}<1\}$, which we call the {\it row ball (at level $k$)}. We may identify $M_{k\times k}^g$ with $\mathbb C^{gk^2}$ by enumerating the entries of the matrices $X_j$, so that $\Omega_k$ is identified with an open subset of $\mathbb C^{gk^2}$.  
  We examine the expression
  \[
  f_d(\mathcal{X})\vcentcolon= \int_{U(d)^g} \det( L_\mathcal{X}(\mathcal{U}) ) \overline{ \det( L_\mathcal{Y}(\mathcal{U}) ) } \, d\mathcal{U}
  \]
  where we consider $\mathcal{Y}$ to be held fixed, and we allow $\mathcal{X}$ to vary over the row ball $\Omega_k$ at level $k$. Then for each $d$, the scalar valued function $f_d(\mathcal{X})$ is a polynomial in the $g k^2$ complex variables $\{x_{ij}^{(k)}:1\leq i,j \leq k; \ 1\leq k\leq g\}$ corresponding to each entry of each of the matrices $X_1, \dots, X_g$.  The hypothesis (\ref{eqn:local-L2-bound}), together with the Cauchy-Schwarz inequality, then implies that the sequence of polynomials $(f_d(\mathcal{X}))_{d=1}^\infty$ is uniformly bounded on each compact subset of $\Omega_k$. It follows from Montel's theorem that every subsequence of $f_d$ has a subsequence which converges uniformly on compact subsets of $\Omega_k$, to some holomorphic function $f$ (which a priori will depend on the choice of subsequence, but we will show that it does not). Indeed, any $f$ holomorphic in $\Omega_k$ has a unique expansion into homogeneous polynomials $f(\mathcal{X}) = \sum_{n=1}^\infty f^{(n)}(\mathcal{X})$, where each $f^{(n)}$ is homogeneous of degree $n$.
  
  Equation~(\ref{eqn:expanded-integral}) tells us that
  \begin{equation*}
  	f_{d}^{(n)}(\mathcal{X}) = \frac{1}{(n!)^2} \sum_{ |\alpha| = n } \binom{n}{\alpha}^2  \sum_{ \sigma , \tau \in S_n } \sign(\sigma)\sign(\tau) p_{\sigma,\alpha}(\mathcal{X}) \overline{ p_{\tau,\alpha}(\mathcal{Y}) } \int_{ U(d)^g } p_{\sigma,\alpha}(\mathcal{U}) \overline{ p_{\tau,\alpha}(\mathcal{U}) }  d\mathcal{U}
  \end{equation*}
  or $0$ when $\min(kd , k'd) \leq n$. From Corollary~\ref{cor:termwise-limit} we see that in each degree $n$, we have \begin{equation*}
  	\lim_{d\to \infty} f_d^{(n)}(\mathcal{X})=\frac{1}{n!} \sum_{ |\alpha| = n } \binom{n}{\alpha} \sum_{ \sigma \in S_n } p_{\sigma,\alpha}(\mathcal{X}) \overline{ p_{\sigma,\alpha}(\mathcal{Y}) }.
  \end{equation*} By Lemma~\ref{lem:target-expansion}, this is exactly the homogeneous term of degree $n$ appearing in the expansion of the (rational) function $f(\mathcal X)=\det( I_{k} \otimes I_{k'} - \sum X_j \otimes \overline{Y_j} )^{-1}$, which is holomorphic in $\Omega_k$ by our spectral radius assumption. On the other hand, whenever $f_d\to f$ locally uniformly in the domain $\Omega_k$, it follows from the multivariable Cauchy integral formula that the homogeneous terms also converge: $f_d^{(n)}\to f^{(n)}$ for each $n$. We conclude that every subsequence of $f_d(\mathcal{X})$ has a subsequence converging to $\det( I_{k} \otimes I_{k'} - \sum X_j \otimes \overline{Y_j} )^{-1}$, which proves the proposition. 
\end{proof}

\section{Scalar Coefficient Pencils}\label{sec:scalar}

In this section, we consider (\ref{eqn:pencil-limit-conjecture}) for $\mathcal{X} , \mathcal{Y} \in \mathbb{C}^g$ i.e. for scalar-coefficient pencils. This affords us an explicit expression analogous to (\ref{eqn:HS-ident}) which we can easily take the limit of.

\begin{thm}\label{thm:main}
	
	For any $x , y \in \mathbb{C}^g$, we have
	
	\begin{equation}\label{eqn:main-integral}
		\int_{ U(d)^g } \det ( L_x( \mathcal{U} ) ) \; \overline{ \det ( L_y( \mathcal{U} ) ) } \, d\mathcal{U} = \sum_{n=0}^d \sum_{ |\alpha| = n } c(d, \alpha) \binom{n}{\alpha} x^\alpha \overline{y^\alpha}
	\end{equation}
	
	\noindent where
	
	\begin{equation}\label{eqn:main-coefficient}
		c(d, \alpha) = \binom{d}{n} \binom{n}{\alpha} \prod_{j=1}^g \binom{d}{\alpha_j}^{-1}.
	\end{equation}
\end{thm}
Since the corollary about the large $d$ limit follows quickly, we state and prove it first, then return to the proof of the theorem. 

\begin{cor}\label{cor:main-limit}
	
	For any $x, y \in \mathbb{C}^g$ such that $\|x\|_2 , \|y\|_2 < 1$, we have
	
	\begin{equation}\label{eqn:main-limit}
		\lim_{d \to \infty} \int_{ U(d)^g } \det ( L_x( \mathcal{U} ) ) \; \overline{ \det ( L_y( \mathcal{U} ) ) } \, d\mathcal{U} = \frac{1}{1 - \langle x , y \rangle} .
	\end{equation}

\end{cor}

\begin{proof}
	
	Let us expand the terms in (\ref{eqn:main-coefficient}).
	
	\begin{align*}
		c(d, \alpha) &= \binom{d}{n} \binom{n}{\alpha} \prod_{j=1}^g \binom{d}{\alpha_j}^{-1} \\
		&= \frac{ \frac{d (d-1) \cdots (d - n + 1 ) }{ n! } }{ \frac{d (d-1) \cdots (d - \alpha_1 + 1 ) }{ \alpha_1 ! } \cdots \frac{d (d-1) \cdots (d - \alpha_g + 1 ) }{ \alpha_g ! } } \cdot \frac{n!}{\alpha_1!\alpha_2!\cdots\alpha_g!} \\
		&= \frac{d (d-1) \cdots (d - n + 1 ) }{ [d (d-1) \cdots (d - \alpha_1 + 1)] \cdots [d (d-1) \cdots (d - \alpha_g + 1)] }.
	\end{align*}
	
	Both the numerator and denominator are a product of $n$ integers. The factors of the numerator begin at $d$ and decrease consecutively for $n$ steps. But the factors of the denominator begin at $d$, decrease for $\alpha_1$ steps, return to $d$, decrease for $\alpha_2$ steps, and so on. Additionally, the numerator and denominator are both monic polynomials in $d$ of degree $n$. Together, we see that $0\leq c(d, \alpha)\leq 1$ for all $d$ and $\alpha$ and that $\lim_{d\to \infty} c(d, \alpha)=1$ for all $\alpha$.
	
	Since the series
	\[
	\sum_{n=0}^\infty \sum_{|\alpha|=n} \binom{n}{\alpha} x^\alpha\overline{y}^\alpha = \frac{1}{1-\langle x,y\rangle}
	\]
	is absolutely convergent for $\|x\|_2, \|y\|_2<1$, the result follows from Theorem~\ref{thm:main} and dominated convergence.
	
\end{proof}
We now turn to the proof of Theorem~\ref{thm:main}.  The strategy is to expand the determinants as sums of trace polynomials, and then reduce the integrals of these trace polynomials to combinatorial expressions which can be evaluated explicitly. We will require a sequence of lemmas.

\begin{lem}\label{lem:pencil-expansion}
	
	\begin{gather}\label{eqn:pencil-integral-pruned}
		\int_{U(d)^g} \det(L_x(\mathcal{U})\overline{\det(L_y(\mathcal{U})}\, d\mathcal{U} \\
		= \sum_{n=0}^d \frac{1}{(n!)^2} \sum_{ |\alpha| = n } \binom{n}{\alpha}^2 x^\alpha \overline{y}^\alpha  \sum_{ \sigma , \tau \in S_n } \sign(\sigma)\sign(\tau) \int_{ U(d)^g } p_{\sigma,\alpha}(\mathcal{U}) \overline{ p_{\tau,\alpha}(\mathcal{U}) } \, d\mathcal{U} \nonumber
	\end{gather}
	
\end{lem}
\begin{proof}
	Since $x$ and $y$ are scalar tuples, we have $k=1$ in (\ref{eqn:general-power-sum-poly}). All matrix/tensor products become scalar products, thus one checks easily that $p_{\sigma, \alpha}(x) = x^\alpha$, so the formula follows from $(\ref{eqn:expanded-integral})$. 
\end{proof}

To proceed, we must recall (\ref{eqn:Q-formula}). Specifically, we consider the one-dimensional (hence irreducible) $S_n$-representation $\rho_{S_n}^{\varepsilon} : \sigma \to \sign(\sigma)$, whose corresponding partition is $\varepsilon = (1 , 1 , \ldots , 1)$ \cite[p. 116, Example 1]{Macdonald-2015}. Then $Q_\varepsilon$ is given by
\begin{equation*}
	Q_\varepsilon \vcentcolon= \frac{1}{n!} \sum_{\sigma \in S_n} \sign(\sigma) \sigma .
\end{equation*}
\noindent For each $\tau \in S_n$, we observe
\begin{align*}
	Q_\varepsilon \tau &= \frac{1}{n!} \sum_{\sigma \in S_n} \sign(\sigma) \sigma \tau \\
	&= \frac{1}{n!} \sum_{\sigma \in S_n} \sign(\sigma \tau^{-1}) \sigma \\
	&= \sign(\tau) \frac{1}{n!} \sum_{\sigma \in S_n} \sign(\sigma) \sigma \\
	&= \sign(\tau) Q_\varepsilon 
\end{align*}

\noindent by re-indexing the sum via $\sigma \mapsto \sigma \tau^{-1}$. Furthermore,
\begin{align*}
	\rho_{S_n}^d (Q_\varepsilon)^T &= \frac{1}{n!} \sum_{\sigma \in S_n} \sign(\sigma) \rho_{S_n}^d(\sigma)^T \\
	&= \frac{1}{n!} \sum_{\sigma \in S_n} \sign(\sigma) \rho_{S_n}^d(\sigma)^{-1} \\
	&= \frac{1}{n!} \sum_{\sigma \in S_n} \sign(\sigma^{-1}) \rho_{S_n}^d(\sigma^{-1}) \\
	&= \rho_{S_n}^d (Q_\varepsilon)
\end{align*}

\noindent since each $\rho_{S_n}^d(\sigma)$ is a permutation matrix. This also means $\rho_{S_n}^d (Q_\varepsilon)$ has real entries, hence it is self-adjoint and symmetric. Using the definition of $Q_\varepsilon$ and summing over $\sigma$ and $\tau$, we may rewrite (\ref{eqn:pencil-integral-pruned}) as
\begin{equation}\label{eqn:pencil-integral-centalt}
	(\ref{eqn:pencil-integral-pruned}) = \sum_{n=0}^d \sum_{ |\alpha| = n } \binom{n}{\alpha}^2 x^\alpha \overline{y}^\alpha \int_{ U(d)^g } \tr( Q_\varepsilon \circ \mathcal{U}^{\otimes \alpha} ) \overline{ \tr( Q_\varepsilon \circ \mathcal{U}^{\otimes \alpha} ) }\, d\mathcal{U} .
\end{equation}
Technically, we should write $\rho_{S_n}^d (Q_\varepsilon)$ in the above trace, but again we are suppressing the dependence on $n$ and $d$. Ultimately we shall compute the value of this integral, but achieving this requires some more rearranging.

\begin{lem}
	Define an operator $\mathbb{E}_\alpha$ on $\morph(\mathbb{C}^d)^{\otimes n}$ by
	\begin{gather*}
		\mathbb{E}_{\alpha}(X) \vcentcolon= \int_{U(d)^g} \mathcal{U}^{\otimes \alpha} X (\mathcal{U}^{\otimes \alpha})^\ast \, d\mathcal{U} \\
	\end{gather*}
	
	Then for any $T_1, T_2\in \morph(\mathbb{C}^d)^{\otimes n}$ we have 
	\begin{equation}\label{eqn:trace-integral}
		\int_{ U(d)^g } \tr( T_1\circ \mathcal{U}^{\otimes \alpha} ) \overline{ \tr( T_2 \circ \mathcal{U}^{\otimes \alpha} ) }\, d\mathcal{U} = \tr ( \mathbb{E}_\alpha \circ L_{T_1} \circ R_{T_2^*} ).
	\end{equation}
	
\end{lem}

\begin{proof}
  We have
  \begin{align*}
    \tr( T_1\circ \mathcal{U}^{\otimes \alpha} ) \overline{ \tr( T_2 \circ \mathcal{U}^{\otimes \alpha} )} &=\overline{ \tr( T_2 \circ \mathcal{U}^{\otimes \alpha})} \tr( T_1\circ \mathcal{U}^{\otimes \alpha} )\\
    &= \tr[(\overline{T_2}\otimes T_1)\circ  (\mathcal{\overline{U}}^{\otimes \alpha}\otimes  \mathcal{U}^{\otimes \alpha})]
  \end{align*}
  The matrix $(\overline{T_2}\otimes T_1)$ is the matrix of the linear map $L_{T_1}\circ R_{T_2^*}:X\to T_1XT_2^*$ by $(\ref{eqn:left-right})$. Similarly $(\mathcal{\overline{U}}^{\otimes \alpha}\otimes  \mathcal{U}^{\otimes \alpha})$ is the matrix of $X\to  \mathcal{U}^{\otimes \alpha} X (\mathcal{U}^{\otimes \alpha})^*$. Integrating over $U$ gives the result.
\end{proof}

We are interested in the case when $T_1=T_2=Q_\varepsilon$, a central projection (hence self-adjoint); we will write $C_{Q_\varepsilon} \vcentcolon= L_{Q_\varepsilon} \circ R_{Q_\varepsilon}$. 

To compute $\tr ( \mathbb{E}_\alpha \circ C_{Q_\varepsilon} )$, we will first examine more closely the map $\mathbb E_\alpha$. It will turn out that this map is a conditional expectation onto the subalgebra generated by the Young subgroup $S_\alpha$ in the representation $\rho_{S_n}^d$. The calculation is a straightforward generalization of the calculation of the expectation $\mathbb E$ onto the image of the full group $S_n$ under $\rho_{S_n}^d$ as carried out by Collins and \'{S}niady \cite{Collins-Sniady-2006}. They prove that for any integer $j \geq 1$, the map
\begin{equation}\label{eqn:e-map}
	\mathbb{E}_j(X) = \int_{U(d)} U^{\otimes j} X (U^{\otimes j})^\ast \, d\mathcal{U}
\end{equation}

\noindent is an orthogonal projection of $\morph(\mathbb{C}^d)^{\otimes j}$ onto $\rho_{S_j}^d(\mathbb{C}[S_j])$. The orthogonality is with respect to the tracial inner product $\langle X , Y \rangle = \tr( Y^\ast X)$. This projection is related to another operator on $\morph(\mathbb{C}^d)^{\otimes j}$, namely
\begin{equation}\label{eqn:phi-map}
	\Phi_j(X) = \sum_{\sigma \in S_j} \tr ( X \circ \rho_{S_j}^d(\sigma^{-1}) ) \cdot \rho_{S_j}^d(\sigma)
\end{equation}

\noindent by the equation $\Phi_j(X) = \Phi_j(I_{d^j}) \cdot \mathbb{E}_j(X)$ \cite[Proposition 2.3.3]{Collins-Sniady-2006}. In fact $\Phi_j(I_{d^j})$ is invertible, thus 
\begin{equation}\label{eqn:e-and-phi-maps}
	\mathbb{E}_j(X) = \Phi_j(I_{d^j})^{-1} \circ \Phi_j(X).
\end{equation}

\noindent The following two lemmas will give us a way of generalizing $\mathbb{E}$ and $\Phi$ to multi-indices.

\begin{lem}
	
	Given $\alpha \in \mathbb{N}^g$ with $|\alpha| = n$, define $\Phi_{\alpha} \vcentcolon= \Phi_{\alpha_1} \otimes \cdots \otimes \Phi_{\alpha_g}$. Then
	\begin{equation}\label{eqn:phi-map-multi}
		\Phi_{\alpha}(X) = \sum_{\gamma \in S_\alpha} \tr ( X \circ \rho_{S_n}^d(\gamma^{-1}) ) \cdot \rho_{S_n}^d(\gamma) .
	\end{equation}
	
\end{lem}

\begin{proof}
	
	We shall prove the result assuming $g = 2$ for ease of notation, but the same technique applies for any value of $g$. Suppose that $\gamma \in S_\alpha$, which means $\gamma = \sigma \tau$ where $\sigma \in S_{\alpha_1}$ and $\tau \in S_{\alpha_2}$. Using the isomorphism $(\mathbb{C}^d)^{\otimes n} = (\mathbb{C}^d)^{\otimes \alpha_1} \otimes (\mathbb{C}^d)^{\otimes \alpha_2}$, it follows that 
	\begin{equation*}
		\rho_{S_n}^d(\gamma) = \rho_{S_{\alpha_1}}^d(\sigma) \otimes \rho_{S_{\alpha_2}}^d(\tau) 
	\end{equation*} 
	since $\gamma$ acts like $\sigma$ on the first $\alpha_1$ tensor factors of $(\mathbb{C}^d)^{\otimes n}$ and like $\tau$ on the last $\alpha_2$ factors. Thus, for $X_i \in \morph(\mathbb{C}^d)^{\otimes \alpha_i}$ ($i=1,2$), we see
	\begin{align*}
		\Phi_{(\alpha_1 , \alpha_2)}(X_1 \otimes X_2) &= \Phi_{\alpha_1}(X_1) \otimes \Phi_{\alpha_2}(X_2) \\
		&= \left( \sum_{\sigma \in S_{\alpha_1}} \tr ( X_1 \circ \rho_{S_{\alpha_1}}^d(\sigma^{-1}) ) \cdot \rho_{S_{\alpha_1}}^d(\sigma) \right) \otimes \left( \sum_{\tau \in S_{\alpha_2}} \tr ( X_2 \circ \rho_{S_{\alpha_2}}^d(\tau^{-1}) ) \cdot \rho_{S_{\alpha_2}}^d(\tau) \right) \\
		&= \sum_{\substack{\sigma \in S_{\alpha_1} \\ \tau \in S_{\alpha_2}}} \tr \left( ( X_1 \circ \rho_{S_{\alpha_1}}^d(\sigma^{-1}) ) \otimes ( X_2 \circ \rho_{S_{\alpha_2}}^d(\tau^{-1}) ) \right) \cdot \rho_{S_n}^d(\sigma\tau) \\
		&= \sum_{\substack{\sigma \in S_{\alpha_1} \\ \tau \in S_{\alpha_2}}} \tr \left( ( X_1 \otimes X_2 ) \circ ( \rho_{S_n}^d(\sigma^{-1}\tau^{-1}) ) \right) \cdot \rho_{S_n}^d(\sigma\tau) \\
		&= \sum_{\gamma \in S_\alpha} \tr ( (X_1 \otimes X_2) \circ \rho_{S_n}^d(\gamma^{-1}) ) \cdot \rho_{S_n}^d(\gamma) .
	\end{align*}
	Note that $\sigma$ and $\tau$ commute with each other, so $\sigma^{-1}\tau^{-1} = \gamma^{-1}$ yields the final equality. Since $\morph(\mathbb{C}^d)^{\otimes n}$ is spanned by such $X_1 \otimes X_2$, the result follows by linearity.
	
\end{proof}

\begin{lem}
	
	For any multinomial $\alpha$,
	\begin{equation*}
		\mathbb{E}_\alpha = \mathbb{E}_{\alpha_1} \otimes \cdots \otimes \mathbb{E}_{\alpha_g} ,
	\end{equation*}
	Thus $\mathbb{E}_\alpha$ is the orthogonal projection of $\morph(\mathbb{C}^d)^{\otimes n}$ onto $\rho_{S_n}^d(\mathbb{C}[S_\alpha])$.
	
\end{lem}

\begin{proof}
	
	For any matrix $A$, let $[A]_{i,j}$ denote its $(i,j)$ entry.
	
	Again, assume $g = 2$ for simplicity. The matrix $$(\mathbb{E}_{\alpha_1} \otimes \mathbb{E}_{\alpha_2})(X_1 \otimes X_2) = \mathbb{E}_{\alpha_1}(X_1) \otimes \mathbb{E}_{\alpha_2}(X_2)$$ is the block matrix whose $(i,j)$ entry is $[\mathbb{E}_{\alpha_1}(X_1)]_{i,j} \mathbb{E}_{\alpha_2}(X_2)$. The $(k , l)$ entry of this matrix is then $[\mathbb{E}_{\alpha_1}(X_1)]_{i,j} [\mathbb{E}_{\alpha_2}(X_2)]_{k,l}$. Now
	\begin{align*}
		[\mathbb{E}_{\alpha_1}(X_1)]_{i,j} [\mathbb{E}_{\alpha_2}(X_2)]_{k,l} &= \left[ \int_{U(d)} U_1^{\otimes \alpha_1} X_1 (U_1^{\otimes \alpha_1})^\ast \, dU_1 \right]_{i,j} \left[ \int_{U(d)} U_2^{\otimes \alpha_2} X_2 (U_2^{\otimes \alpha_2})^\ast \, dU_2 \right]_{k,l} \\
		&=	\int_{U(d)} \int_{U(d)} \left[ U_1^{\otimes \alpha_1} X_1 (U_1^{\otimes \alpha_1})^\ast \right]_{i,j} \left[ U_2^{\otimes \alpha_2} X_2 (U_2^{\otimes \alpha_2})^\ast \right]_{k,l} \, dU_1 dU_2
	\end{align*}
	
	\noindent But the integrand is exactly the $((i, j), (k, l))$ entry of
	\begin{equation*}
		( U_1^{\otimes \alpha_1} X_1 (U_1^{\otimes \alpha_1})^\ast ) \otimes ( U_2^{\otimes \alpha_2} X_2 (U_2^{\otimes \alpha_2})^\ast ) = ( U_1^{\otimes \alpha} \otimes U_2^{\otimes \alpha_2} ) ( X_1 \otimes X_2 ) ( U_1^{\otimes \alpha} \otimes U_2^{\otimes \alpha_2} )^\ast .
	\end{equation*}
	
	Thus,
	\begin{align*}
		(\mathbb{E}_{\alpha_1} \otimes \mathbb{E}_{\alpha_2})(X_1 \otimes X_2) &= \int_{U(d)} \int_{U(d)} ( U_1^{\otimes \alpha_1} \otimes U_2^{\otimes \alpha_2} ) ( X_1 \otimes X_2 ) ( U_1^{\otimes \alpha_1} \otimes U_2^{\otimes \alpha_2} )^\ast \, dU_1 dU_2 \\
		&= \int_{U(d)^2} \mathcal{U}^{\otimes \alpha} (X_1 \otimes X_2) (\mathcal{U}^{\otimes \alpha})^\ast \, d\mathcal{U} \\
		&= \mathbb{E}_{\alpha}(X_1 \otimes X_2).
	\end{align*}
	
	\noindent Self-adjointness, idempotence, and orthogonality of $\mathbb{E}_\alpha$ all follow since the individual $\mathbb{E}_{\alpha_j}$ have these properties. Furthermore, the range of $\mathbb{E}_\alpha$ is the tensor product of the ranges of the $\mathbb{E}_{\alpha_i}$, namely $\bigotimes_{j=1}^g \rho_{S_{\alpha_i}}^{d}( \mathbb{C}[S_{\alpha_i}] ) =\rho_{S_n}^d(\mathbb{C}[S_\alpha])$. 
	
\end{proof}

\begin{cor}
	
	For a multinomial $\alpha$ of weight $n$, 
	
	\begin{equation}\label{eqn:e-and-phi-maps-multi}
		\mathbb{E}_\alpha(X) = \Phi_\alpha(I)^{-1} \circ \Phi_\alpha(X).
	\end{equation}
	
\end{cor}

\begin{proof}
	Tensor $(\ref{eqn:e-and-phi-maps})$ over the $\alpha_j$.
\end{proof}

\begin{prop}\label{prop:trace-projection}
	
	\begin{equation}\label{eqn:trace-projection}
		\tr ( \mathbb{E}_\alpha \circ C_{Q_\varepsilon} ) = \frac{ \tr( \rho_{S_n}^d(Q_\varepsilon) ) }{ \prod_{j=1}^g \tr( \rho_{S_{\alpha_j}}^d(Q_\varepsilon) ) }
	\end{equation}

\end{prop}

\noindent We remind the reader that $\mathbb{E}_\alpha , C_{Q_\varepsilon} \in \morph(\morph(\mathbb{C}^d)^{\otimes n})$ while $\rho_{S_n}^{d} (Q_\varepsilon) \in \morph(\mathbb{C}^d)^{\otimes n}$.

\begin{proof}
	
	For any finite-dimensional vector space $V$ and linear $T: V \to V$, it is true that $\tr(T) = \tr(T|_{\text{range}(T)})$ -- simply complete any basis for $\text{range}(T)$ and consider the resulting block decomposition of $T$. Hence, in computing $\tr ( \mathbb{E}_\alpha \circ C_{Q_\varepsilon} )$ we need only consider the restriction of this operator to its range. As we saw in Section~\ref{sec:representations}, $\rho_{S_n}^d$ is injective when $n \leq d$, thus the set $\rho_{S_n}^d(S_\alpha)$ is a basis for the range of $\mathbb{E}_\alpha$. 
	
	 For $v \in \mathbb{C}[S_n]$, let $[v]_\sigma$ denote the $\sigma$-coefficient of $v$ i.e. $v = \sum_{\sigma} [v]_\sigma \; \sigma$. Then
	\begin{align*}
		\tr ( \mathbb{E}_\alpha \circ C_{Q_\varepsilon} ) &= \sum_{\gamma \in S_\alpha} [ \mathbb{E}_\alpha C_{Q_\varepsilon}( \gamma ) ]_{ \gamma } \\
		&= \sum_{\gamma \in S_\alpha} [ \mathbb{E}_\alpha(Q_\varepsilon \gamma Q_\varepsilon) ]_{ \gamma } \\
		&= \sum_{\gamma \in S_\alpha} [ \mathbb{E}_\alpha((Q_\varepsilon)^2 \sign(\gamma)) ]_{ \gamma } \\
		&= \sum_{\gamma \in S_\alpha} \sign(\gamma) [ \mathbb{E}_\alpha(Q_\varepsilon) ]_{ \gamma } .
	\end{align*}
	
	\noindent Next,
	\begin{align*}
		\mathbb{E}_\alpha(\rho_{S_n}^d(Q_\varepsilon)) &\overset{(\ref{eqn:e-and-phi-maps-multi})}{=} \Phi_\alpha(I)^{-1} \circ \Phi_\alpha(\rho_{S_n}^d(Q_\varepsilon)) \\
		&\overset{(\ref{eqn:phi-map-multi})}{=} \Phi_\alpha(I)^{-1} \circ \sum_{\gamma \in S_\alpha} \tr( \rho_{S_n}^d(Q_\varepsilon) \; \gamma^{-1} ) \; \gamma \\
		&= \Phi_\alpha(I)^{-1} \circ \sum_{\gamma \in S_\alpha} \tr( \rho_{S_n}^d(Q_\varepsilon) ) \sign(\gamma) \; \gamma \\
		&= \tr( \rho_{S_n}^d(Q_\varepsilon) ) \; \Phi_\alpha(I)^{-1} \circ \sum_{\gamma \in S_\alpha} \sign(\gamma) \; \gamma \\
		&= \tr( \rho_{S_n}^d(Q_\varepsilon) ) \; \bigotimes_{j = 1}^g \left( \Phi_{\alpha_j}(I)^{-1} \circ \sum_{\gamma^{(j)} \in S_{\alpha_j}} \sign(\gamma^{(j)}) \; \gamma^{(j)} \right) \\
		&= \tr( \rho_{S_n}^d(Q_\varepsilon) ) \; \bigotimes_{j = 1}^g \left( \frac{1}{ \tr( \rho_{S_{\alpha_j}}^d(Q_\varepsilon)) } \Phi_{\alpha_j}(I)^{-1} \circ \sum_{\gamma^{(j)} \in S_{\alpha_j}} \tr( \rho_{S_{\alpha_j}}^d(Q_\varepsilon)) \sign(\gamma^{(j)}) \; \gamma^{(j)} \right) \\
		&= \frac{ \tr( \rho_{S_n}^d(Q_\varepsilon) ) }{ \prod_{j=1}^g \tr( \rho_{S_{\alpha_j}}^d(Q_\varepsilon)) } \bigotimes_{j = 1}^g \left( \Phi_\alpha(I)^{-1} \circ \Phi_{\alpha_j}(\rho_{S_{\alpha_j}}^d(Q_\varepsilon)) \right) \\
		&\overset{(\ref{eqn:e-and-phi-maps})}{=} \frac{ \tr( \rho_{S_n}^d(Q_\varepsilon) ) }{ \prod_{j=1}^g \tr( \rho_{S_{\alpha_j}}^d(Q_\varepsilon)) } \; \bigotimes_{j = 1}^g \left( \mathbb{E}_{\alpha_j} (\rho_{S_{\alpha_j}}^d(Q_\varepsilon)) \right) \\
		&= \frac{ \tr( \rho_{S_n}^d(Q_\varepsilon) ) }{ \prod_{j=1}^g \tr( \rho_{S_{\alpha_j}}^d(Q_\varepsilon)) } \; \bigotimes_{j = 1}^g Q_{\alpha_j}\\
		&= \frac{ \tr( \rho_{S_n}^d(Q_\varepsilon) ) }{ \prod_{i=1}^g \tr( \rho_{S_{\alpha_j}}^d(Q_\varepsilon)) } \; \frac{1}{ \alpha_1! \cdots \alpha_g! } \; \sum_{\gamma \in S_\alpha} \sign(\gamma) \gamma .
	\end{align*}
	
	Since $\alpha_1! \cdots \alpha_g! = |S_\alpha|$, we see
	\begin{align*}
		\tr ( \mathbb{E}_\alpha \circ C_{Q_\varepsilon} ) &= \sum_{\gamma \in S_\alpha} \sign(\gamma) [ \mathbb{E}_\alpha(Q_\varepsilon) ]_{ \gamma } \\
		&= \sum_{\gamma \in S_\alpha} \sign(\gamma) \frac{ \tr( \rho_{S_n}^d(Q_\varepsilon) ) }{ \prod_{i=1}^g \tr( \rho_{S_{\alpha_j}}^d(Q_\varepsilon)) } \frac{1}{|S_\alpha|} \sign(\gamma) \\
		&= \frac{ \tr( \rho_{S_n}^d(Q_\varepsilon) ) }{ \prod_{i=1}^g \tr( \rho_{S_{\alpha_j}}^d(Q_\varepsilon)) } \frac{1}{|S_\alpha|} \sum_{\gamma \in S_\alpha} 1 \\
		&= \frac{ \tr( \rho_{S_n}^d(Q_\varepsilon) ) }{ \prod_{i=1}^g \tr( \rho_{S_{\alpha_j}}^d(Q_\varepsilon)) } .
	\end{align*}
		
\end{proof}

\begin{prop}\label{thm:trace-centalt}
	
	For $n \leq d$,
	\begin{equation}\label{eqn:trace-centalt}
		\tr( \rho_{S_n}^d(Q_\varepsilon) ) = \binom{d}{n} .
	\end{equation}
	
\end{prop}

\begin{proof}
	
	Recalling Equations~(\ref{eqn:character-of-permutation-representation-direct}) and (\ref{eqn:sign}), we see
	\begin{align*}
		\tr( \rho_{S_n}^d(Q_\varepsilon) ) &= \frac{1}{n!} \sum_{\sigma \in S_n} \sign(\sigma) \chi_{S_n}^{d}(\sigma) \\
		&= \frac{1}{n!} \sum_{\sigma \in S_n} (-1)^{n-c(\sigma)} d^{c(\sigma)} \\
		&= \frac{1}{n!} \sum_{j = 1}^n s(n,j) d^{j} \\
		&= \frac{1}{n!} d ( d - 1 ) \cdots ( d - ( n - 1 ) ) \\
		&= \frac{1}{n!} \cdot \frac{d!}{(d - n)!} \\
		&= \binom{d}{n} .
	\end{align*}
	where $s(n,j)$ is a Stirling number of the first kind. The value $s(n,j)$ is known to be the (signed) number of elements of $S_n$ which have $j$ disjoint cycles and also the coefficient of $d^j$ in the falling factorial $d ( d - 1 ) \cdots ( d - ( n- 1 ) )$; see Section 1.3.3 to 1.3.4 of Stanley \cite{Stanley-2012} for a proof.
	
\end{proof}

We finally finish the proof of Theorem~\ref{thm:main}, 
\begin{align*}
	&\int_{ U(d)^g } \det ( L_x( \mathcal{U} ) ) \; \overline{ \det ( L_y( \mathcal{U} ) ) } \, d\mathcal{U} \\
	\overset{(\ref{eqn:pencil-integral-centalt})}{=} &\sum_{n=0}^d \sum_{ |\alpha| = n } \binom{n}{\alpha}^2 x^\alpha \overline{y}^\alpha \int_{ U(d)^g } \tr( Q_\varepsilon \circ \mathcal{U}^{\otimes \alpha} ) \overline{ \tr( Q_\varepsilon \circ \mathcal{U}^{\otimes \alpha} ) }\, d\mathcal{U} \\
	\overset{(\ref{eqn:trace-integral})}{=} &\sum_{n=0}^d \sum_{ |\alpha| = n } \binom{n}{\alpha}^2 x^\alpha \overline{y}^\alpha \tr ( \mathbb{E}_\alpha \circ C_{Q_\varepsilon} ) \\
	\overset{(\ref{eqn:trace-projection})}{=} &\sum_{n=0}^d \sum_{ |\alpha| = n } \binom{n}{\alpha}^2 x^\alpha \overline{y}^\alpha \frac{ \tr( \rho_{S_n}^d(Q_\varepsilon) ) }{ \prod_{j=1}^g \tr( \rho_{S_{\alpha_i}}^d(Q_\varepsilon) ) } \\
	\overset{(\ref{eqn:trace-centalt})}{=} &\sum_{n=0}^d \sum_{ |\alpha| = n } \binom{n}{\alpha}^2 \binom{d}{n} \prod_{j=1}^g \binom{d}{\alpha_j}^{-1} x^\alpha \overline{y}^\alpha \\
	\overset{(\ref{eqn:main-coefficient})}{=} &\sum_{n=0}^d \sum_{ |\alpha| = n } c(d, \alpha) \binom{n}{\alpha} x^\alpha \overline{y}^\alpha .
\end{align*}

\section{Higher moments and triangular pencils}\label{sec:triangular}

In this section we give a proof of Conjecture~(\ref{eqn:pencil-limit-conjecture}) under the additional assumption that the coefficient matrices $X_1, \dots, X_g$, and separately the coefficients $Y_1, \dots, Y_g$, can be put into simultaneous upper (or lower) triangular form. This will evidently include the case that the systems of coefficients $\mathcal{X}$ and $\mathcal{Y}$ each commute among themselves. In fact, we will see that the triangular case can be reduced to the case where all the coefficient matrices are scalar multiplies of the identity matrix. Nevertheless, even the case of scalar multiplies of the identity seems substantially more difficult then the case of genuinely scalar coefficients treated in Section~\ref{sec:scalar}. We will see that this amounts to controlling higher moments of the genuinely scalar pencils. More precisely, the case of scalar multiples of the identity (of size $k\times k$) is essentially equivalent to the problem of bounding the $k^{th}$ moments of $|\det(I+\sum_{j=1}^g x_jU_j)|^2$ for $\|x\|\leq r<1$, independently of the size of $U$. After explaining this reduction, we proceed in broadly the same way as in the scalar case in Section~\ref{sec:scalar}, except that the resulting combinatorial expressions will be significantly more complicated, and we will settle for bounds rather than explicit calculations of the coefficients (which seem out of reach). 

\begin{thm}\label{thm:true-for-triangular} Conjecture~(\ref{eqn:pencil-limit-conjecture}) holds when the coefficients $X_j, Y_j$ are upper triangular. 
\end{thm}
From elementary properties of determinants, if the $X_j$ are upper triangular, then the expression $\det\left(I\otimes I+\sum_{j=1}^g X_j\otimes U_j\right)$ is unchanged after replacing the $X_j$ by their corresponding diagonals. Thus, the upper triangular case immediately reduces to the diagonal case. Concretely, the theorem then implies
\begin{cor}
For vectors $x_l = (x_l^{(1)}, \dots, x_l^{(g)})$, $l=1, \dots, k$, and $y_m = (y_m^{(1)}, \dots, y_m^{(g)})$, $m=1, \dots k^\prime$ with $\|x_l\|_2, \|y_m\|_2<1$,  
  \[
  \lim_{d\to \infty} \int_{U(d)^g} \prod_{l=1}^k \det(I_d+\sum_{j=1}^{g} x_l^{(j)} U_j)\prod_{m=1}^{k^\prime}\overline{ \det(I_d+\sum_{j=1}^{g} y_m^{(j)} U_j)}\, d\mathcal{U} = \prod_{l=1}^k\prod_{m=1}^{k^\prime} \frac{1}{1-\langle x_l, y_m\rangle_{\mathbb C^g}}
  \]

\end{cor}

We first show that the conjecture will follow from suitable $L^2$ bounds on the determinants of the pencil, which are independent of the size $d$.

\begin{prop}\label{prop:uniform-L2-bound}
	
	Suppose that for each $0\leq r<1$ and each integer $k\geq 1$ there is a constant $C(r,k)$ such that
	\begin{equation}\label{eqn:uniform-L2-bound}
		\sup_{d\geq1} \int_{U(d)^g} |\det(I_k \otimes I_d + \sum_{j=1}^{g} x_j I_k \otimes U_j)|^2 \, d\mathcal{U} \leq C(r,k)
	\end{equation}
	for all scalars $x_1, \dots, x_g$ with $|x_1|^2+\cdots+|x_g|^2\leq r$. Then Conjecture~(\ref{eqn:pencil-limit-conjecture}) holds for upper triangular coefficients $\mathcal{X}$ and $\mathcal{Y}$.
	
\end{prop}

\begin{proof}
	
	As noted above, we can assume from the outset that the tuples $\mathcal X$, $\mathcal Y$ are diagonal. So, let $X_1, \dots X_g$ be $k\times k$ diagonal matrices and let $Y_1, \dots, Y_g$ be $k^\prime\times k^\prime$ diagonal matrices, and let $x_l^{(j)}$ be the $l^{th}$ diagonal entry of $X_j$, similarly for the $Y_j$. For $l=1, \dots, k$ let $x_l\in\mathbb{C}^g$ be the vector $x_l = (x_l^{(1)}, \dots, x_l^{(g)})$, and similarly define vectors $y_m \in \mathbb{C}^g$, $m = 1, \dots, k^\prime$ from $\mathcal{Y}$. By the hypothesis of Conjecture~(\ref{eqn:pencil-limit-conjecture}), the spectral radii of $\sum X_j\otimes \overline{X_j}$ and $\sum Y_j\otimes \overline{Y_j}$ are strictly less than $1$, say both are less than some fixed $r<1$.  The matrix $\sum X_j\otimes \overline{X_j}$ is diagonal, and each diagonal entry has the form $\langle x_{l_1}, x_{l_2}\rangle_{\mathbb C^g}$ for pairs of indices $l_1, l_2$ in $\{1, \dots, k\}$. It follows in particular that $\|x_l\|^2\leq r<1$ for each $l=1, \dots, k$, that is, $\sum_{j=1}^g |x_l^{(j)}|^2\leq r$, and similarly for the vectors $y_m$ built from $\mathcal{Y}$.  We then have
	\begin{equation}\label{eqn:diagonal-determinant-expansion}
		\det(I \otimes I +\sum_{j=1}^{g} X_j\otimes U_j) = \prod_{l=1}^k \det(I+\sum_{j=1}^{g} x_l^{(j)}U_j)
	\end{equation}
	and
	\begin{equation*}
		\det(I \otimes I +\sum_{j=1}^{g} Y_j\otimes U_j) = \prod_{m=1}^{k^\prime} \det(I+\sum_{j=1}^{g} y_m^{(j)}U_j)
	\end{equation*}
	We have obviously for any $x\in\mathbb{C}^g$
	\begin{equation*}
		|\det(I \otimes I +\sum_{j=1}^{g} x_j I_k\otimes U_j)|^2 = |\det(I +\sum_{j=1}^{g} x_j U_j)|^{2k}.
	\end{equation*}
	Let us temporarily introduce the notation
	\begin{equation*}
		\Delta(x,\mathcal{U}) \vcentcolon= \det(I +\sum_{j=1}^{g} x_j U_j).
	\end{equation*}
	so that the previous equation becomes
	\begin{equation}\label{eqn:tensor-to-2k-moment}
		|\det(I \otimes I +\sum_{j=1}^{g} x_j I_k\otimes U_j)|^2 = |\Delta(x, \mathcal{U})|^{2k}.
	\end{equation}
	From our hypothesis (\ref{eqn:uniform-L2-bound}) about the existence of $C(r,k)$, we conclude that the function $\Delta(x,U)$ has finite $2k^{th}$ moments (that is, belongs to $L^{2k}(d\mathcal{U})$) for all $x$ in a fixed ball of radius $r<1$, and all finite $k$, and the $L^{2k}$ norms are bounded by constant which depends on $r$ and $k$, but not on $d$. We now claim that for diagonal tuples $\mathcal X$ of size $k$, satisfying ${\textbf{rad}(\mathcal X)}\leq r<1$
	\begin{equation}\label{eqn:L2-bound-triangular}
		\int_{U(d)^g} |\det(I \otimes I+\sum_{j=1}^{g} X_j\otimes U_j)|^2 \, d\mathcal{U} \leq C(r,k)
	\end{equation}
	where $C(r,k)$ is the same constant appearing in (\ref{eqn:uniform-L2-bound}). Indeed, we express the determinant as a product of determinants $\Delta(x,\mathcal{U})$ and apply the multivariable H\"older inequality. By (\ref{eqn:diagonal-determinant-expansion}) we have
	\begin{equation*}
		\det(I \otimes I+\sum_{j=1}^{g} X_j\otimes U_j) = \prod_{l=1}^k \Delta(x_l,\mathcal{U})
	\end{equation*}
	and thus
	\begin{align*}
		\int_{U(d)^g} |\det(I \otimes I+\sum_{j=1}^{g} X_j\otimes U_j)|^2\, d\mathcal U &= \int_{U(d)^g} \prod_{l=1}^k |\Delta(x_l,\mathcal{U})|^2 \, d\mathcal U \\
		&= \left\| \prod_{l=1}^k \Delta(x_l,\mathcal{U}) \right\|^2_2 \\
		&\leq \prod_{l=1}^k \left\| \Delta(x_l,\mathcal{U}) \right\|^2_{2k} \\
		&\leq C(r,k).
	\end{align*}
	
	The first inequality is the multivariable H\"older inequality; the second inequality follows because (\ref{eqn:tensor-to-2k-moment}) allows us to rewrite the hypothesis (\ref{eqn:uniform-L2-bound}) as $$\|\Delta(x,\mathcal{U})\|_{2k} \leq C(r,k)^{1/2k}.$$ Now we can complete the proof by appeal to Montel's theorem: indeed, we have satisfied the hypotheses of Proposition~\ref{prop:L2-bound-is-sufficient} with the added restriction that $\mathcal{X}$ is also upper triangular.
	
\end{proof}

Proposition~\ref{prop:uniform-L2-bound} tells us that in order to prove Conjecture (\ref{eqn:pencil-limit-conjecture}) for upper triangular $\mathcal{X},\mathcal{Y}$, it suffices to show the bound (\ref{eqn:local-L2-bound}) when each $X_j$ is a scalar multiple of the $k \times k$ identity matrix i.e. it suffices to prove (\ref{eqn:uniform-L2-bound}). This is exactly what we shall accomplish in the remainder of this section.

The initial manipulations will imitate the $k=1$ case treated in Section~\ref{sec:scalar}, but it will turn out that we will have to consider central projections $Q_\lambda$ onto summands corresponding to more general irreducible representations of $S_n$, rather than only the sign representation as was needed in the scalar case. We obtain explicit (but rather more complicated) combinatorial expressions for the coefficients in the homogeneous expansion, and the desired $L^2$ bound will ultimately follow from combinatorial arguments which bound these coefficients independently of $d$. 

\begin{prop} For fixed $k\geq 1$ and complex numbers $x_1, \dots, x_g$ we have
	
	\begin{align}\label{eqn:integral-to-trace-identity-matrix}
		\int_{U(d)^g} &|\det(I_k \otimes I_d + \sum_{j=1}^{g} x_j I_k \otimes U_j)|^2 \, d\mathcal{U} = \\
		&\sum_{n=0}^{kd} \sum_{|\alpha|=n} |x|^{2\alpha}\binom{n}{\alpha}^2  \sum_{ \substack{ \lambda \vdash n \\ wd(\lambda)\leq k } } \left(\frac{s_{\lambda^*}(k)}{\chi_\lambda(1)}\right)^2 \tr[\mathbb E_\alpha \circ C_{Q_\lambda}] \nonumber 
	\end{align}
  
\end{prop}

\begin{proof}

	We begin by invoking (\ref{eqn:expanded-integral}) with $X_j =x_j I$. Observe that
	\begin{align*}
		p_{\sigma, \alpha}(\mathcal{X}) &= \tr( \rho_{S_k}^d(\sigma^{-1}) \circ \mathcal{X}^{\otimes \alpha} ) \\
		&= x^{\alpha} \tr( \rho_{S_k}^d(\sigma^{-1}) ) \\
		&\overset{(\ref{eqn:character-of-permutation-representation-direct})}{=} x^{\alpha} k^{c(\sigma)}
	\end{align*}\\
	which then yields
	
	\begin{align*}
		(\ref{eqn:integral-to-trace-identity-matrix}) &= \sum_{n=0}^{kd} \frac{1}{(n!)^2} \sum_{|\alpha| = n} \binom{n}{\alpha}^2 \sum_{ \sigma , \tau \in S_n } \sign(\sigma) \sign(\tau) x^{\alpha}k^{c(\sigma)} \overline{x^{\alpha}}k^{c(\tau)} \int_{ U(d)^g } p_{\sigma,\alpha}(\mathcal{U}) \overline{ p_{\tau,\alpha}(\mathcal{U}) } \, d\mathcal{U} \\
		&= \sum_{n=0}^{kd} \sum_{|\alpha| = n} |x|^{2\alpha} \binom{n}{\alpha}^2 \left[ \frac{1}{(n!)^2} \sum_{ \sigma , \tau \in S_n } \sign(\sigma) k^{c(\sigma)} \sign(\tau) k^{c(\tau)} \int_{ U(d)^g } p_{\sigma,\alpha}(\mathcal{U}) \overline{ p_{\tau,\alpha}(\mathcal{U}) } \, d\mathcal{U} \right] .
	\end{align*}
	Proving the proposition requires rewriting the expression in square brackets.
		
	We invoke (\ref{eqn:character-of-permutation-representation-decomposed}) and (\ref{eqn:character-of-permutation-representation-direct}) along with the fact that $\text{sgn}(\sigma) \chi_\lambda(\sigma) = \chi_{\lambda^*}(\sigma)$ \cite[Theorems 35.1, 35.2, 37.4]{Bump-2004} \cite[Pg 116, Example 2]{Macdonald-2015} to obtain
	\begin{align*}
		\sign(\sigma) k^{c(\sigma)} &= \sum_{\substack{\lambda \vdash n \\ ht(\lambda)\leq k}} s_\lambda(k) \sign(\sigma)\chi_\lambda(\sigma) \\
		&= \sum_{\substack{\lambda \vdash n \\ ht(\lambda)\leq k}} s_\lambda(k) \chi_{\lambda^\ast}(\sigma) \\
		&= \sum_{\substack{\lambda \vdash n \\ wd(\lambda)\leq k}} s_{\lambda^*}(k) \chi_\lambda(\sigma)
	\end{align*}
	where we have changed the summation index $\lambda \to \lambda^*$.
	
	Recalling the formula for the projection $Q_\lambda$ (\ref{eqn:Q-formula}), (and noting that the character $\text{sgn}(\sigma)k^{c(\sigma)}$ is unchanged on replacing $\sigma$ by $\sigma^{-1}$) we have
	\begin{align*}
		&\phantom{=} \frac{1}{n!^2}\sum_{\sigma, \tau\in S_n} \text{sgn}(\sigma)k^{c(\sigma)}\text{sgn}(\tau)k^{c(\tau)}\int \tr(\sigma \circ U^\alpha)\overline{\tr(\tau \circ U^\alpha)}\, d\mathcal{U} \\
		&=  \frac{1}{n!^2}\sum_{\sigma, \tau\in S_n} \sum_{ \substack{ \lambda \vdash n \\ wd(\lambda)\leq k } } s_{\lambda^*}(k) \chi_\lambda(\sigma^{-1})\sum_{ \substack{ \mu \vdash n \\ wd(\mu)\leq k } } s_{\mu^*}(k) \chi_\mu(\tau^{-1})\int \tr(\sigma \circ U^\alpha)\overline{\tr(\tau \circ U^\alpha)}\, d\mathcal{U} \\
		&=  \sum_{\lambda, \mu} \frac{s_{\lambda^*}(k)}{\chi_\lambda(1)}\frac{s_{\mu^*}(k)}{\chi_\mu(1)} \int \tr(Q_\lambda \circ U^\alpha)\overline{\tr(Q_\mu \circ U^\alpha)}\, d\mathcal{U}.
	\end{align*}
	Our full integral (\ref{eqn:integral-to-trace-identity-matrix}) is now 
	\begin{equation*}
		\sum_{n=0}^{kd} \sum_{|\alpha|=n} |x|^{2\alpha}\binom{n}{\alpha}^2  \sum_{\lambda, \mu} \frac{s_{\lambda^*}(k)}{\chi_\lambda(1)}\frac{s_{\mu^*}(k)}{\chi_\mu(1)} \int \tr(Q_\lambda \circ U^\alpha)\overline{\tr(Q_\mu \circ U^\alpha)}\, d\mathcal{U}.
	\end{equation*}
	We can rewrite as in (\ref{eqn:trace-integral})
	\begin{equation*}
		\int \tr(Q_\lambda \circ U^\alpha)\overline{\tr(Q_\mu \circ U^\alpha)}\, d\mathcal{U} = \tr(\mathbb E_\alpha \circ L_{Q_\lambda} \circ R_{Q_\mu})
	\end{equation*}
	By orthogonality of the distinct characters discussed in Proposition~\ref{prop:properties-of-Q}, we have $Q_\lambda X Q_\mu = Q_\lambda Q_\mu X = 0$ whenever $\lambda\neq \mu$ and $X\in \mathbb C[S_n]$, hence $L_{Q_\lambda} \circ R_{Q_\mu} = 0$ (for $\lambda\neq \mu$) when restricted to $\rho^{d}_{S_n}(\mathbb{C}[S_n])$. As discussed at the beginning of the proof of Proposition~\ref{prop:trace-projection}, this restriction has no effect on the above trace, so it follows that our full integral (\ref{eqn:integral-to-trace-identity-matrix}) is now equal to 
	\begin{equation*}
		\sum_{n=0}^{kd} \sum_{|\alpha|=n} |x|^{2\alpha}\binom{n}{\alpha}^2  \sum_{ \substack{ \lambda \vdash n \\ wd(\lambda)\leq k } } \left(\frac{s_{\lambda^*}(k)}{\chi_\lambda(1)}\right)^2 \tr(\mathbb E_\alpha \circ C_{Q_\lambda})
	\end{equation*}
	as desired.
	
\end{proof}
	
Our next major task is to obtain (as we did in Propositions~\ref{prop:trace-projection}~and~\ref{thm:trace-centalt}) a combinatorial expression for $\tr(\mathbb E_\alpha \circ C_{Q_\lambda})$. At this point, to avoid overburdening the notation, we shall work in $g=2$ variables. Once we have proved the full theorem in that case, we will indicate the modifications necessary for $g>2$ (this will be straightforward).

\begin{lem}\label{thm:key-trace-formula}
	Let $n\geq 1$ be an integer, let $\alpha=(\alpha_1, \alpha_2)$ with $\alpha_1+\alpha_2=n$ be a multi-index of order $n$, and let $\lambda$ be a partition of $n$. We have
	\begin{equation}\label{eqn:trace-centproj}
		\tr(\mathbb E_\alpha\circ C_{Q_\lambda}) =  \sum_{\substack{\mu\vdash \alpha_1 \\ \nu\vdash\alpha_2}} c_{\mu\nu}^\lambda  \frac{s_\lambda(d)}{s_\mu(d)s_\nu(d)} \chi_\mu(1)^2\chi_\nu(1)^2
	\end{equation}
	when $ht(\lambda) \leq d$, otherwise $\tr(\mathbb E_\alpha\circ C_{Q_\lambda}) = 0$.
\end{lem}

\begin{proof}
	
	Using the fact that $Q_\lambda$ is a central projection, we may imitate the argument in the proof of Proposition~\ref{prop:trace-projection} to check that
	\begin{equation}
		\tr(\mathbb E_\alpha\circ C_{Q_\lambda}) = \alpha! [\mathbb{E}_\alpha(Q_\lambda)]_{(1)}
	\end{equation}
	Indeed, recalling equations (\ref{eqn:e-and-phi-maps-multi}) and (\ref{eqn:phi-map-multi}), which say $\mathbb{E}_\alpha = \Phi_\alpha(I)^{-1} \circ \Phi_\alpha$ and $\Phi_\alpha(X) = \sum_{\delta \in S_\alpha} \tr(X\delta^{-1}) \delta$, we have 
	\begin{align*}
		\Phi_\alpha(C_{Q_\lambda}(\gamma)) &= \sum_{\delta \in S_\alpha} \tr(Q_\lambda \gamma Q_\lambda\delta^{-1}) \delta \\
		&=\sum_{\delta \in S_\alpha} \tr(Q_\lambda \gamma\delta^{-1}) \delta \\
		&= \sum_{\delta \in S_\alpha} \tr(Q_\lambda \delta^{-1}) \delta\gamma \\
		&=\Phi_\alpha(Q_\lambda) \circ \gamma.
	\end{align*}
	It follows that $\mathbb{E}_\alpha (C_{Q_\lambda}(\gamma))= (\mathbb E_\alpha(Q_\lambda)) \circ \gamma$. Hence, the coefficient of $\gamma$ in the expansion of $\mathbb E_\alpha (C_{Q_\lambda}(\gamma))$ is equal to the coefficient of the identity element $(1) \in S_\alpha$ in the expansion of $\mathbb{E}_\alpha(Q_\lambda)$. Summing over $\gamma\in S_\alpha$ yields ${\tr(\mathbb E_\alpha\circ C_{Q_\lambda}) =\alpha! [\mathbb{E}_\alpha(Q_\lambda)]_{(1)}}$ as claimed. 
	
	So like we did in Proposition~\ref{prop:trace-projection}, we need to compute $\mathbb E_\alpha(Q_\lambda)$. The situation now is more complicated since we will need to apply the splitting rule (\ref{eqn:splitting-rule}). The result will be a kind of splitting rule for the central projections $Q_\lambda$, which will decompose $\mathbb E_\alpha(Q_\lambda)$ as a linear combination of central projections $Q_\mu\otimes Q_\nu$ in the group algebra $\mathbb C[S_\alpha]$.
	
	To continue, recall (\ref{eqn:tensor-representation-decomp})
	\begin{equation*}
		(\mathbb{C}^d)^{\otimes n} = \bigoplus_{\lambda\vdash n, \ ht(\lambda)\leq d} s_\lambda(d) V_{S_n}^{\lambda}.
	\end{equation*}
	which told us how to decompose $(\mathbb{C}^d)^{\otimes n}$ into $S_n$-irreps.  Since $\rho_{S_n}^{d}(Q_\lambda)$ is the central projection onto the component $s_\lambda(d) V^{\lambda}_{S_n}$, it follows that $\tr(Q_\lambda \gamma) = s_\lambda(d) \chi_\lambda(\gamma)$ when $ht(\lambda)\leq d$, otherwise $\tr(Q_\lambda \gamma) = 0$ because then $Q_\lambda$ projects onto a component of $(\mathbb{C}^d)^{\otimes n}$ with multiplicity $0$.
		
	Thus we can compute 
	\begin{align*}
		\Phi_\alpha(Q_\lambda) &= \sum_{\gamma \in S_\alpha} \tr(Q_\lambda\gamma^{-1}) \gamma\\
		&= s_\lambda(d) \sum_{\gamma\in S_\alpha} \chi_\lambda(\gamma^{-1}) \gamma \\ &\overset{(\ref{eqn:splitting-rule})}{=} s_\lambda(d) \sum_{\gamma_1\in S_{\alpha_1}, \gamma_2\in S_{\alpha_2}} \sum_{\mu,\nu} c_{\mu\nu}^\lambda \chi_\mu(\gamma_1^{-1})\chi_\nu(\gamma_2^{-1})\gamma_1\otimes\gamma_2
	\end{align*}
	
	Multiplying and dividing by $s_\mu(d), s_\nu(d)$ and applying the formula for $\Phi(Q)$ recursively we obtain
	\begin{equation*}
		\Phi_\alpha(Q_\lambda) = \sum_{\mu, \nu} c_{\mu\nu}^\lambda  \frac{s_\lambda(d)}{s_\mu(d)s_\nu(d)} \Phi_{\alpha_1}(Q_\mu) \otimes \Phi_{\alpha_2}(Q_\nu)
	\end{equation*}
	so that now multiplying by the $\Phi_\alpha(I)^{-1}$ terms and performing manipulations very similar to those used in the proof of Proposition~\ref{prop:trace-projection} yields
	\begin{equation*}
		\mathbb E_\alpha(Q_\lambda) =  \sum_{\mu, \nu} c_{\mu\nu}^\lambda  \frac{s_\lambda(d)}{s_\mu(d)s_\nu(d)} Q_\mu\otimes Q_\nu .
	\end{equation*}

	Using the formula~(\ref{eqn:Q-formula}) for the $Q's$ and extracting the coefficient $[\mathbb{E}_\alpha(Q_\lambda)]_{(1)}$, and then multiplying by $\alpha!$, we get
	\begin{equation*}
		\tr(\mathbb E_\alpha\circ C_{Q_\lambda}) =  \sum_{\mu, \nu} c_{\mu\nu}^\lambda  \frac{s_\lambda(d)}{s_\mu(d)s_\nu(d)} \chi_\mu(1)^2\chi_\nu(1)^2.
	\end{equation*}
	
\end{proof}

Substituting (\ref{eqn:trace-centproj}) into (\ref{eqn:integral-to-trace-identity-matrix}), the integral is now equal to
\begin{equation}\label{eqn:main-term}
	\sum_{n=0}^{kd}  \binom{n}{\alpha} |x|^{2\alpha} \left[ \binom{n}{\alpha}\ \sum_{ \substack{ ht(\lambda)\leq d \\ wd(\lambda)\leq k } } \frac{s_{\lambda^*}(k)^2}{\chi_\lambda(1)^2} \sum_{\mu, \nu} c_{\mu\nu}^\lambda  \frac{s_\lambda(d)}{s_\mu(d)s_\nu(d)} \chi_\mu(1)^2\chi_\nu(1)^2 \right].
\end{equation}
We turn to bounding the expression inside the square brackets in (\ref{eqn:main-term}).

\begin{lem}\label{thm:schur-bound}
	If $\lambda$ is a partition of $n$ with $ht(\lambda) \leq k$, then
	\begin{equation*}
		s_{\lambda}(k) \leq n^{\binom{k}{2}} = n^{\frac{k^2}{2}-\frac{k}{2}}.
	\end{equation*}
\end{lem}

\begin{proof} Our hypothesis is $ht(\lambda) \leq k$, which means $\lambda$ has at most $k$ parts. From the Stanley hook-content formula (\ref{eqn:schur-hook-and-content}) we have for any partition $\lambda$ 
  \[
  s_{\lambda}(k)= \prod_{u\in\lambda} \frac{k+c_u}{h_u}
  \]
  where $c_u$ is the content and $h_u$ is the hook length of the cell $u$. On the other hand also from \cite[Lemma 7.21.1]{Stanley-2024} we have for any partition $\lambda$ with $k$ parts
  \[
  \prod_{u\in\lambda} h(u) = \frac{\prod_{i=1}^k \mu_i!}{\prod_{1\leq i<j\leq k}(\mu_i-\mu_j)}\quad \text{and}\quad \prod_{u\in\lambda} (k+c_u) = \prod_{i=1}^k \frac{\mu_i!}{(k-i)!}.
  \]
  where $\mu_i = \lambda_i+k-i$. Thus
  \[
  s_{\lambda}(k) = \frac{\prod_{1\leq i<j\leq k}(\mu_i-\mu_j)}{\prod_{i=1}^k(k-i)!} \leq \prod_{1\leq i<j\leq k}(\mu_i-\mu_j)
  \]
Each number $\mu_i-\mu_j$ is at most $n$, and there are $\binom{k}{2}$ pairs $1\leq i<j\leq k$. The conclusion follows.    
\end{proof}

Apply this lemma to $\lambda^\ast$ in (\ref{eqn:main-term}), we are left to estimate
\begin{equation}\label{eqn:character-and-schur-sum}
	\binom{n}{\alpha} \sum_{ \substack{ ht(\lambda)\leq d \\ wd(\lambda)\leq k } } \sum_{\mu, \nu} c_{\mu\nu}^\lambda  \frac{s_\lambda(d)}{s_\mu(d)s_\nu(d)} \frac{\chi_\mu(1)^2\chi_\nu(1)^2}{\chi_\lambda(1)^2}
\end{equation}
where $\mu, \nu$ are partitions of $\alpha_1, \alpha_2$ respectively. We first apply the Stanley hook-content formula to make the $d$ dependence more explicit. For a partition $\lambda$ we let $C_\lambda(d)$ denote what we will call the {\em content polynomial}
\begin{equation}\label{eqn:content-poly-definition}
	C_\lambda(d) = \prod_{u\in \lambda} (d+c_u).	
\end{equation}

\begin{lem}
	
	\begin{equation}\label{eqn:reduction-to-C(d)}
		\binom{n}{\alpha} \sum_{ \substack{ ht(\lambda)\leq d \\ wd(\lambda)\leq k } } \sum_{\mu, \nu} c_{\mu\nu}^\lambda  \frac{s_\lambda(d)}{s_\mu(d)s_\nu(d)} \frac{\chi_\mu(1)^2\chi_\nu(1)^2}{\chi_\lambda(1)^2} = \sum_{ \substack{ ht(\lambda)\leq d \\ wd(\lambda)\leq k } } \sum_{\mu, \nu} c_{\mu\nu}^\lambda \frac{C_\lambda(d)}{C_\mu(d)C_\nu(d)} \frac{\chi_\mu(1)\chi_\nu(1)}{\chi_\lambda(1)}
	\end{equation}
  
\end{lem}

\begin{proof}
	Recall $(\ref{eqn:schur-hook-and-content})$ and $(\ref{eqn:character-hook-length})$:
	\begin{equation*}
		s_\lambda(d)= \prod_{u\in\lambda} \frac{d+c_u}{h_u}, \quad \chi_\lambda(1) = \frac{n!}{\prod_{u\in\lambda} h_u}.
	\end{equation*}
	Dividing these immediately yields
	\begin{equation*}
		\frac{s_\lambda(d)}{\chi_\lambda(1)} =\frac{C_\lambda(d)}{n!}
	\end{equation*}
	and likewise for $\mu,\nu$, hence the conclusion follows.
\end{proof}

The splitting rule (\ref{eqn:splitting-rule}) applied to $\gamma = (1)$ rearranges to
\[
\sum_{\mu, \nu} c_{\mu\nu}^\lambda \frac{\chi_\mu(1)\chi_\nu(1)}{\chi_\lambda(1)} =1.
\]
Thus, each term associated to fixed $\lambda$ in the sum in the right hand side of (\ref{eqn:reduction-to-C(d)}) is a convex combination of the rational functions of $d$
\[
r_{\mu,\nu}^\lambda(d) \vcentcolon= \frac{C_\lambda(d)}{C_\mu(d)C_\nu(d)}
\]
which we call the {\em content ratios}. The following is the key estimate in our proof: 
\begin{thm}\label{thm:content-ratio-bound}
  Fix the following:
  \begin{itemize}
  \item integers $k,d\geq 1$,
  \item an integer $k \leq n \leq kd$,
  \item a partition $\lambda\vdash n$ with at most $d$ parts, and each part of size at most $k$ (i.e. $ht(\lambda) \leq d$ and $wd(\lambda) \leq k$),
  \item partitions $\mu, \nu$ such that the Littlewood-Richardson coefficient $c^\lambda_{\mu\nu}$ is nonzero.
  \end{itemize}
  Then

  \[
  r_{\mu,\nu}^\lambda(d) \leq (n+1)^{k^2}.
  \]
\end{thm}
\begin{proof} See Section~\ref{sec:content-ratio-lemma}. 
\end{proof}

\begin{proof}[Proof of Theorem~\ref{thm:true-for-triangular}] By Proposition~\ref{prop:uniform-L2-bound}, it suffices to fix integer $k\geq 1$ and real $0\leq r< 1$ and prove there is a constant $C(r,k)$ so that the $L^2$ bound (\ref{eqn:uniform-L2-bound}) holds. By Equations~(\ref{eqn:integral-to-trace-identity-matrix}), (\ref{eqn:trace-centproj}), (\ref{eqn:reduction-to-C(d)}) we have that for each fixed $d\geq 1$, the integral is equal to

\begin{align}\label{eqn:thing-to-bound}
	\int_{U(d)^g} |\det(I_k \otimes I_d + &\sum_{j=1}^{g} x_j I_k \otimes U_j)|^2 \, d\mathcal{U} = \\
	\sum_{n=0}^{kd} \binom{n}{\alpha} |x|^{2\alpha} &\sum_{ \substack{ \lambda\vdash n \\ wd(\lambda)\leq k \\ ht(\lambda) \leq d } } s_{\lambda^*}(k)^2\sum_{ \substack{ \mu \vdash \alpha_1 \\ \nu \vdash \alpha_2 } } c_{\mu\nu}^\lambda \frac{\chi_\mu(1)\chi_\nu(1)}{\chi_\lambda(1)} r_{\mu,\nu}^\lambda(d) \nonumber
\end{align}

The terms in the sum with $0\leq n < k$ are uniformly bounded in $d$, since they converge as $d\to \infty$ by Corollary~\ref{cor:termwise-limit}, and there is a fixed finite number of them. We will bound the portion of the sum for $n\geq k$ by the tail of a convergent series. By Theorem~\ref{thm:content-ratio-bound} and the fact that a convex combination of nonnegative real numbers is bounded by the largest one, we have for fixed $\lambda\vdash n, wd(\lambda)\leq k, ht(\lambda) \leq d$
\[
\sum_{\mu, \nu} c_{\mu\nu}^\lambda \frac{\chi_\mu(1)\chi_\nu(1)}{\chi_\lambda(1)} r_{\mu,\nu}^\lambda(d) \leq (n+1)^{k^2}
\]
when $n\geq k$. By Lemma~\ref{thm:schur-bound} the quantity $s_{\lambda^*}(k)^2$ is bounded by $n^{k^2-k}$. The number of partitions $\lambda$ of $n$ into parts of size at most $k$ is trivially bounded by $n^k$. So altogether we have
\[
\sum_{ \substack{ \lambda\vdash n \\ wd(\lambda)\leq k \\ ht(\lambda) \leq d } } s_{\lambda^*}(k)^2\sum_{\mu, \nu} c_{\mu\nu}^\lambda \frac{\chi_\mu(1)\chi_\nu(1)}{\chi_\lambda(1)} r_{\mu,\nu}^\lambda(d)\leq n^{k}\cdot n^{k^2-k} \cdot (n+1)^{k^2} \leq C_k\cdot n^{2k^2}.
\]
where $C_k$ is a constant depending only on $k$. Applying this bound to the expression (\ref{eqn:thing-to-bound}), we have for $\|x\|_2\leq r<1$
\begin{align*}
  \sum_{n=k}^{kd}  \sum_{|\alpha|=n}\binom{n}{\alpha} |x|^{2\alpha} &\sum_{ \substack{ \lambda\vdash n \\ wd(\lambda)\leq k \\ ht(\lambda) \leq d } } s_{\lambda^*}(k)^2\sum_{\mu, \nu} c_{\mu\nu}^\lambda \frac{\chi_\mu(1)\chi_\nu(1)}{\chi_\lambda(1)} r_{\mu,\nu}^\lambda(d)\\
  &\leq C_k\cdot \sum_{n=k}^{kd} \sum_{|\alpha|=n}\binom{n}{\alpha} |x|^{2\alpha}(n^{2k^2})\\
  &=  C_k\cdot \sum_{n=k}^\infty \|x\|_2^{2n}(n^{2k^2})\\
  &\leq C_k\cdot \sum_{n=k}^\infty n^{2k^2}r^{2n}<\infty.
\end{align*}

\end{proof}

\subsection{The $g>2$ case} Beginning at Lemma~\ref{thm:key-trace-formula} we assumed $g=2$ to keep the statements and proofs easily readable. However, everything works out when $g>2$ as long as we apply the general splitting rule $(\ref{eqn:splitting-rule-general})$ where appropriate. Indeed, Lemma~\ref{thm:key-trace-formula} becomes

\begin{lem}\label{thm:key-trace-formula-general} Let $n\geq 1$ be an integer, let $\alpha=(\alpha_1, \ldots, \alpha_g)$ be a multi-index of order $n$, and let $\lambda$ be a partition of $n$. We have
	\begin{equation}\label{eqn:key-trace-formula-general}
		\tr(\mathbb E_\alpha\circ C_{Q_\lambda}) =  \sum_{\mu^i\text{'s},\nu^i\text{'s}} c^{\lambda}_{\mu^1 , \nu^1} \left( \prod_{i=2}^{g-2} c^{\nu^{i-1}}_{\mu^i, \nu^i} \right) c^{\nu^{g-2}}_{\mu^{g-1},\mu^{g}} \frac{s_\lambda(d)}{\prod_{i=1}^{g} s_{\mu^i}(d)} \prod_{i=1}^{g} \chi_{\mu^i}(1)^2
	\end{equation}
	when $ht(\lambda) \leq d$, otherwise $\tr(\mathbb E_\alpha\circ C_{Q_\lambda}) = 0$.
\end{lem}

Likewise, we have an analogue to (\ref{eqn:reduction-to-C(d)}), namely
\begin{lem}
	
	\begin{align}\label{eqn:reduction-to-C(d)-general}
		\binom{n}{\alpha} &\sum_{ \substack{ ht(\lambda)\leq d \\ wd(\lambda)\leq k } } \sum_{\mu^i\text{'s},\nu^i\text{'s}} c^{\lambda}_{\mu^1 , \nu^1} \left( \prod_{i=2}^{g-2} c^{\nu^{i-1}}_{\mu^i, \nu^i} \right) c^{\nu^{g-2}}_{\mu^{g-1},\mu^{g}} \frac{s_\lambda(d)}{\prod_{i=1}^{g} s_{\mu^i}(d)} \frac{\prod_{i=1}^{g} \chi_{\mu^i}(1)^2}{\chi_\lambda(1)^2} \\
		= &\sum_{ \substack{ ht(\lambda)\leq d \\ wd(\lambda)\leq k } } \sum_{\mu^i\text{'s},\nu^i\text{'s}} c^{\lambda}_{\mu^1 , \nu^1} \left( \prod_{i=2}^{g-2} c^{\nu^{i-1}}_{\mu^i, \nu^i} \right) c^{\nu^{g-2}}_{\mu^{g-1},\mu^{g}} \frac{C_\lambda(d)}{\prod_{i=1}^{g} C_{\mu^i}(d)} \frac{\prod_{i=1}^{g} \chi_{\mu^i}(1)}{\chi_\lambda(1)} \nonumber
	\end{align}
	
\end{lem}

Again, evaluating the splitting rule $(\ref{eqn:splitting-rule-general})$ on the identity permutation $(1) \in S_n$ tells us that
\begin{equation*}
	\sum_{\mu^i\text{'s},\nu^i\text{'s}} c^{\lambda}_{\mu^1 , \nu^1} \left( \prod_{i=2}^{g-2} c^{\nu^{i-1}}_{\mu^i, \nu^i} \right) c^{\nu^{g-2}}_{\mu^{g-1},\mu^{g}} \frac{\prod_{i=1}^{g} \chi_{\mu^i}(1)}{\chi_\lambda(1)} = 1
\end{equation*}
so we are still trying to bound a convex combination of the terms
\begin{equation*}
	\frac{C_\lambda(d)}{\prod_{i=1}^{g} C_{\mu^i}(d)}.
\end{equation*}
This is a product of the already defined content ratios $r^{\lambda}_{\mu,\nu}$, so we finally have

\begin{thm}\label{thm:content-ratio-bound-general}
	Fix the following:
	\begin{itemize}
		\item integers $k,d\geq 1$,
		\item an integer $k \leq n \leq kd$,
		\item a multi-index $\alpha = (\alpha_1 , \ldots , \alpha_g)$ with $|\alpha| = n$,
		\item a partition $\lambda\vdash n$ with at most $d$ parts, and each part of size at most $k$ (i.e. $ht(\lambda) \leq d$ and $wd(\lambda) \leq k$),
		\item partitions $\mu \vdash \alpha_i$ $(i=1 , \ldots , g)$, $\nu_i \vdash \alpha_{i+1} + \cdots \alpha_g$ $(i=1 , \ldots , g-2)$ such that the Littlewood-Richardson coefficients $c^\lambda_{\mu^1,\nu^1}$, $c^{\nu^{g-2}}_{\mu^{g-1},\mu^{g}}$, and $c^{\nu^{i-1}}_{\mu^i, \nu^i}$ $(i=2, \ldots , g-2)$ are nonzero.
	\end{itemize}
	Then
	
	\[
	\frac{C_\lambda(d)}{\prod_{i=1}^{g} C_{\mu^i}(d)} \leq (n+1)^{(g-1)k^2}.
	\]
\end{thm}
\begin{proof}
	
	This is essentially a corollary of Theorem~\ref{thm:content-ratio-bound}:
	
	\begin{align*}
		\frac{C_\lambda(d)}{\prod_{i=1}^{g} C_{\mu^i}(d)} &= \frac{C_{\lambda}(d)}{C_{\mu^1}(d) C_{\nu^1}(d)} \left( \prod_{i=2}^{g-2} \frac{C_{\nu^{i-1}}(d)}{C_{\mu^i}(d) C_{\nu^i}(d)} \right) \frac{C_{\nu^{g-2}}(d)}{C_{\mu^{g-1}}(d) C_{\mu^g}(d)} \\
		&= r^{\lambda}_{\mu^1 , \nu^1} \left( \prod_{i=2}^{g-2} r^{\nu^{i-1}}_{\mu^i, \nu^i} \right) r^{\nu^{g-2}}_{\mu^{g-1},\mu^{g}} \\
		&\leq \prod_{i=1}^{g-1} (n+1)^{k^2} \\
		&= (n+1)^{(g-1)k^2}
	\end{align*}
\end{proof}

\noindent Finally, Theorem~\ref{thm:true-for-triangular} holds for any $g$ with the same proof as before by applying Theorem~\ref{thm:content-ratio-bound-general} in place of Theorem~\ref{thm:content-ratio-bound}.

\section{Content ratio lemma}\label{sec:content-ratio-lemma}

This section is devoted to the proof of Theorem~\ref{thm:content-ratio-bound}. For fixed $\lambda$ under the hypotheses of the theorem, we will
\begin{enumerate}[i)]
	\setlength\itemsep{-0.5em}
	\item Find which $\nu$ maximizes $r_{\mu,\nu}^\lambda(d)$ for fixed $\mu$ and $d$. \\
	\item Show that the \emph{pair} $\mu,\nu$ which maximizes $r_{\mu,\nu}^\lambda(d)$ for fixed $d$ has a special form. \\
	\item Show that $r^{\lambda}_{\mu,\nu}(d) \leq (n+1)^{k^2}$ for $\mu,\nu$ of the special form.
\end{enumerate} 

Essentially, starting from arbitrary $\lambda,\mu,\nu$, we will successively modify $\mu,\nu$ in a way that only increases content ratio until it is somehow obvious that the bound in the theorem holds. 

We will need to make use of a partial order on partitions known as \textit{dominance order}, typically denoted $\trianglelefteq$. For two partitions $\omega, \Omega$, we write $\omega \trianglelefteq \Omega$ to mean
\begin{equation}\label{eqn:dominance-order}
	\sum_{i=1}^{t} \omega_i \leq \sum_{i=1}^{t} \Omega_i, \quad t=1,2,\ldots,p .
\end{equation}
This is written as if $\omega$, $\Omega$ have the same number of parts, but the shorter one can be padded out with parts of size $0$ if necessary. Dominance order typically assumes that $\omega$, $\Omega$ partition the same number, but this isn't a necessary restriction for us.

Since we are trying to obtain a bound on $r^{\lambda}_{\mu,\nu}(d)$, throughout this section we will manipulate many inequalities involving content polynomials $C_{\Lambda}(D)$ which shall require multiplying or dividing by the $D + c_u$ defining $C_{\Lambda}(D)$ (recall (\ref{eqn:content-poly-definition})). As long as the height of the partition $\Lambda$ is less than the input $D$, these $D + c_u$ terms will be positive, so our inequalities will never reverse.

We begin with a monotonicity property of the content polynomials:
\begin{lem}\label{lem:content-polynomial-monotone}
	
	The content polynomial $C_\omega(d)$ is dominance-monotone in $\omega$. That is, $\omega \trianglelefteq \Omega$ implies $C_\omega(d) \leq C_\Omega(d)$ when $ht(\omega), ht(\Omega) \leq d$, and the inequality is strict when $\omega \neq \Omega$.
	
\end{lem}

\begin{proof}
	
	Let $p = \max\{ ht(\omega), ht(\Omega) \}$ and note $p \leq d$. First, we observe that a content polynomial can be split into a product of two smaller content polynomials. That is, for $t \in \{ 1 , 2 , \ldots , p - 1\}$, we have
	\begin{align}\label{eqn:content-poly-split}
		C_\omega(d) &= \prod_{i=1}^{p}\prod_{j=1}^{\omega_i} (d + j - i) \\
		&= \prod_{i=1}^{t}\prod_{j=1}^{\omega_i} (d + j - i) \cdot \prod_{i=t+1}^{p}\prod_{j=1}^{\omega_i} (d + j - i) \nonumber \\
		&= C_{(\omega_1 , \ldots , \omega_t)}(d) \cdot \prod_{i=1}^{p-t}\prod_{j=1}^{\omega_{t+i}} (d + j - (t+i)) \nonumber \\
		&= C_{(\omega_1 , \ldots , \omega_t)}(d) \cdot C_{(\omega_{t+1} , \ldots , \omega_p)}(d-t). \nonumber
	\end{align}
	Assuming $ht(\omega) \geq t$, we have
	\begin{equation*}
		ht(\omega_{t+1} , \ldots , \omega_p) \leq ht(\omega) - t \leq d - t 
	\end{equation*}
	and so all factors in $C_{(\omega_{t+1} , \ldots , \omega_p)}(d-t)$ are positive. Otherwise, $C_{(\omega_{t+1} , \ldots , \omega_p)}(d-t)$ is an empty product and thus equal to $1$.
	
	Proceeding by induction on the number of rows of $\Omega$, for $p = 1$ we see $\omega \trianglelefteq \Omega$ is the same as $\omega_1 \leq \Omega_1$, in which case
	\begin{align*}
		C_\omega(d) &= \prod_{j=1}^{\omega_1} (d + j - 1) \\
		&\leq \prod_{j=1}^{\omega_1} (d + j - 1) \cdot \prod_{j = \omega_1 + 1}^{\Omega_1} (d + j - 1) \\
		&= \prod_{j=1}^{\Omega_1} (d + j - 1) \\
		&= C_\Omega(d)
	\end{align*}
	If $\omega_1 < \Omega_1$, then the term for which $j = \omega_1 + 1$ is strictly larger than $1$, yielding strict inequality $C_\omega(d) < C_\Omega(d)$
	
	From now on take $\omega \neq \Omega$. Assume the lemma holds for all partitions with $p-1$ or fewer rows. The definition of $\omega \trianglelefteq \Omega$ gives us $(\omega_1 , \ldots , \omega_t) \trianglelefteq (\Omega_1 , \ldots , \Omega_t)$ for any $t \in \{ 1 , 2 , \ldots , p - 1\}$, but let us examine what happens if we further assume that there exists a $t_0 \leq p-1$ for which $\sum_{i=1}^{t_0} \omega_i = \sum_{i=1}^{t_0} \Omega_i$. Combining this with the definition of $\omega \trianglelefteq \Omega$ immediately yields $(\omega_{t_0+1} , \ldots , \omega_p) \trianglelefteq (\Omega_{t_0+1} , \ldots , \Omega_p)$, and so
	\begin{gather*}
		C_{(\omega_1 , \ldots , \omega_{t_0})}(d) \leq C_{(\Omega_1 , \ldots , \Omega_{t_0})}(d) \\
		C_{(\omega_{t_0+1} , \ldots , \omega_p)}(d-t_0) \leq C_{(\Omega_{t_0+1} , \ldots , \Omega_{p})}(d-t_0)
	\end{gather*}\\
	by our induction hypothesis. Since $\omega \neq \Omega$, one inequality must be strict, so multiplying gives $C_\omega(d) < C_\Omega(d)$ by (\ref{eqn:content-poly-split}).
	
	It remains to examine the situation where no such $t_0$ exists, meaning
	\begin{equation*}
		\sum_{i=1}^{t} \omega_i < \sum_{i=1}^{t} \Omega_i \text{ for all } t = 1 , \ldots , p - 1, \text{ and } \sum_{i=1}^{p} \omega_i \leq \sum_{i=1}^{p} \Omega_i .
	\end{equation*}
	We will now update $\omega$ by decreasing its smallest (nonzero) part $\omega_r$ by $1$ and increasing its largest part $\omega_1$ by $1$, then show that this operation strictly increases $C_\omega(d)$, and that the new partition continues to be dominated by $\Omega$. Repeating this process $(\Omega_1 - \omega_1)$-many times, we eventually arrive at a $\Lambda \trianglelefteq \Omega$ such that $\Lambda_1 = \Omega_1$, placing us in the earlier case (with $t_0 = 1$). Thus $C_\omega(d) < C_{\Lambda}(d) \leq C_\Omega(d)$ where the first inequality is by construction and the second is by $\Lambda_1 = \Omega_1$.
	
	Given $\omega$, let $\hat\omega = (\omega_1 + 1 , \omega_2 , \ldots , \omega_{r-1} , \omega_r - 1)$. Let us check the content polynomials: the factors are identical except that the term $d+\omega_r-r$ in $C_\omega(d)$ (contributed by the box removed from the bottom row of $\omega$) gets replaced by $d+(\omega_1+1)-1=d+\omega_1$ (contributed by the new box added in the top row) in $C_{\hat\omega} (d)$, and since $\omega$ is a partition we must have $\omega_1 \geq \omega_r > \omega_r-r$. Equivalently, $d+\omega_r-r < d+(\omega_1+1)-1=d+\omega_1$ and so $C_\omega(d) < C_{\hat\omega}(d)$. 
	
	Finally, we show $\hat\omega \trianglelefteq \Omega$. This means $\sum_{i=1}^{t} \hat\omega_i \leq \sum_{i=1}^{t} \Omega_i \text{ for all } t = 1 , \ldots , p$. For $t < p$
	\begin{align*}
		\sum_{i=1}^{t} \hat\omega_i = 1 + \sum_{i=1}^{t} \omega_i < 1 + \sum_{i=1}^{t} \Omega_i
	\end{align*}
	by assumption, hence $\sum_{i=1}^{t} \hat\omega_i \leq \sum_{i=1}^{t} \Omega_i$. Finally, $\sum_{i=1}^{p} \hat\omega_i = \sum_{i=1}^{p} \omega_i \leq \sum_{i=1}^{p} \Omega_i$.
	
\end{proof}

We can now prove part $i)$ of our outline:
\begin{cor}\label{cor:maximal-nu}
	
	Given $\lambda \vdash n$ with $ht(\lambda) \leq d$ and $\mu,\nu$ such that $c^{\lambda}_{\mu,\nu} \neq 0$, we have
	\begin{equation*}
		r^{\lambda}_{\mu,\nu}(d) \leq r^{\lambda}_{\mu,\text{rows}(\lambda / \mu)}(d) .
	\end{equation*}

\end{cor}

The partition $\text{rows}(\lambda / \mu)$ is obtained by listing the numbers $( \lambda_i - \mu_i )_{i=1}^{p}$ in descending order. McNamara \cite[Proposition 3.1]{McNamara-2008} proves that the Littlewood-Richardson coefficient $c^{\lambda}_{\mu,\text{rows}(\lambda / \mu)}$ is nonzero, and that $\text{rows}(\lambda / \mu) \trianglelefteq \nu$ for every $\nu$ satisfying $c^{\lambda}_{\mu,\nu} \neq 0$.

\begin{proof}
	
	\begin{align*}
		\text{rows}(\lambda / \mu) \trianglelefteq \nu &\Longrightarrow C_{\text{rows}(\lambda / \mu)}(d) \leq C_{\nu}(d) \\
		&\Longrightarrow \frac{1}{C_{\nu}(d)} \leq \frac{1}{C_{\text{rows}(\lambda / \mu)}(d)} \\
		&\Longrightarrow \frac{C_\lambda(d)}{C_\mu(d) C_\nu(d)} \leq \frac{C_\lambda(d)}{C_\mu(d) C_{\text{rows}(\lambda / \mu)}(d)} \\
		&\Longrightarrow r^{\lambda}_{\mu,\nu}(d) \leq r^{\lambda}_{\mu,\text{rows}(\lambda / \mu)}(d) .
	\end{align*}
	
\end{proof}

Next we prove $ii)$, showing that the maximizing $\mu,\nu$ have a special form. This special form is $\lambda_i = \mu_i + \nu_i$ for $i = 1 , \ldots , p$ which we write as $\lambda = \mu + \nu$. This implies $\nu = \text{rows}(\lambda/\mu)$ and so $c^{\lambda}_{\mu,\nu} \neq 0$.

\begin{prop}\label{thm:special-form-gives-max-ratio}
	
	For $\lambda \vdash n$ with $ht(\lambda) \leq d$, any $\mu,\nu$ which satisfy
	\begin{equation*}
		r^{\lambda}_{\mu,\nu}(d) = \max \{ r^{\lambda}_{\hat\mu,\hat\nu}(d) : \hat\mu,\hat\nu \text{ such that } c^{\lambda}_{\hat\mu, \hat\nu} \neq 0 \}
	\end{equation*}
	necessarily satisfy $\lambda = \mu + \nu$.
	
\end{prop}

Appendix A1.3 of \cite{Stanley-2024} gives many combinatorial methods for computing the Littlewood-Richardson coefficients. While the actual value isn't important to us, a necessary condition for $c^{\lambda}_{\mu, \nu} \neq 0$ is that $\mu,\nu$ are both subpartitions of $\lambda$, so for fixed $\lambda$ there are finitely many $\mu,\nu$ to consider and the maximum $r^{\lambda}_{\mu,\nu}(d)$ has to be attained by some pair.

We prove the contrapositive: given $\mu,\nu$ such that $c^{\lambda}_{\mu, \nu} \neq 0$ and $\lambda \neq \mu + \nu$, we shall update $\mu,\nu$ to new partitions $\hat\mu, \hat\nu$ in a way that strictly increases the content ratio and still satisfies $c^{\lambda}_{\hat\mu, \hat\nu} \neq 0$. 

\begin{proof}
	
	When $( \lambda_i - \mu_i )_{i=1}^{p}$ is a weakly decreasing sequence (a partition), we clearly have $\lambda = \mu + \text{rows}(\lambda / \mu)$. Thus $\lambda \neq \mu + \nu$ just means $\nu \neq \text{rows}(\lambda / \mu)$. In this case, applying Corollary \ref{cor:maximal-nu} proves (the contrapositive of) the proposition.
	
	Now assume $( \lambda_i - \mu_i )_{i=1}^{p}$ isn't decreasing, meaning that there exists at least one index $i_0$ such that
	\begin{equation}\label{eqn:non-decreasing}
		\lambda_{i_0} - \mu_{i_0} < \lambda_{i_0+1} - \mu_{i_0+1} .
	\end{equation}
	Notice that this restriction forces $\mu_{i_0} > \mu_{i_0+1}$ since $\lambda_{i_0}-\lambda_{i_0+1}\geq 0$. Moreover, every $\nu$ will automatically satisfy $\lambda \neq \mu + \nu$, so without loss of generality, take $\nu = \text{rows}(\lambda/\mu)$.
	
	We will update $\mu$ to $\hat\mu$ in one of two possible ways:
	\begin{enumerate}[(A)]
		\setlength\itemsep{-1em}
		\item Increase $\mu_{i_0+1}$ by 1, or \\
		\item decrease $\mu_{i_0}$ by 1.
	\end{enumerate}
	These provide us with our updated $\hat\mu$, and in either cases we set $\hat\nu = \text{rows}(\lambda / \hat\mu)$. We should make sure that (A) and (B) are actually possible. The only way for Update (A) to be impossible is if $\mu_{i_0+1} = \lambda_{i_0+1}$ already, in which case $(\ref{eqn:non-decreasing})$ becomes $\lambda_{i_0} < \mu_{i_0}$, which is absurd since $\mu$ is a subpartition of $\lambda$. The only way for Update (B) to be impossible is if $\mu_{i_0} = 0$, but we've already observed that $\mu_{i_0} > \mu_{i_0+1} \geq 0$.
	
	Now let us investigate each update.
	
	\textbf{Update (A):} Recall that the content ratio is given by
	\begin{equation*}
		r_{\mu,\nu}^\lambda(d) = \frac{C_\lambda(d)}{C_\mu(d)C_\nu(d)}.
	\end{equation*}
	
	Because the $d + c_u$ factors are all strictly positive, we see that the inequality $r^{\lambda}_{\mu,\nu}(d) < r^{\lambda}_{\hat\mu,\hat\nu}(d)$ is equivalent to $C_{\hat\mu}(d) C_{\hat\nu}(d) < C_\mu(d) C_\nu(d)$, which is

	\begin{equation*}
		\prod_{i=1}^{p}\prod_{j=1}^{\hat\mu_i} (d + j - i) \cdot \prod_{i=1}^{p}\prod_{j=1}^{\hat\nu_i} (d + j - i) < \prod_{i=1}^{p}\prod_{j=1}^{\mu_i} (d + j - i) \cdot \prod_{i=1}^{p}\prod_{j=1}^{\nu_i} (d + j - i)
	\end{equation*}\\
	Since $\hat\mu_{i_0+1} = \mu_{i_0+1} + 1$ and $\hat\mu_i = \mu_i$ for $i \neq i_0+1$, there is a great deal of cancellation in the above inequality, resulting in
	
	\begin{equation*}
		(d + (\mu_{i_0+1} + 1) - (i_0+1)) \cdot \prod_{i=1}^{p}\prod_{j=1}^{\hat\nu_i} (d + j - i) < \prod_{i=1}^{p}\prod_{j=1}^{\nu_i} (d + j - i) .
	\end{equation*}\\
	The parts of $\nu$ are just the numbers $\lambda_1 - \mu_1 , \ldots , \lambda_p - \mu_p$ listed in descending order. Meanwhile the parts of $\hat\nu$ are the exact same list except that $\lambda_{i_0+1} - \mu_{i_0+1}$ has been decreased by 1. This means $\nu$ and $\hat\nu$ only differ by a single block in a single row, say row $q$, where $\hat\nu_q = \nu_q - 1$. Evidently, $q$ must be the last row of $\nu$ whose length is $\lambda_{i_0+1} - \mu_{i_0+1}$, which in particular tells us that $\nu_q = \lambda_{i_0+1} - \mu_{i_0+1}$, although we shall not use this quite yet. The earlier inequality is now further equivalent to
	
	\begin{equation*}
		d + (\mu_{i_0+1} + 1) - (i_0+1) < d + \nu_q - q
	\end{equation*}\\
	or just
	
	\begin{equation}\label{eqn:update-A}
		\mu_{i_0+1} - i_0 < \nu_q - q .
	\end{equation}\\
	To summarize, what we have just shown is that if $\hat\mu$ is obtained by Update (A), then $r^{\lambda}_{\mu,\nu}(d) < r^{\lambda}_{\hat\mu,\hat\nu}(d)$ if and only if $(\ref{eqn:update-A})$ holds.
	
	\textbf{Update (B):} As before, $r^{\lambda}_{\mu,\nu}(d) < r^{\lambda}_{\hat\mu,\hat\nu}(d)$ is equivalent to
	
	\begin{equation*}
		\prod_{i=1}^{p}\prod_{j=1}^{\hat\mu_i} (d + j - i) \cdot \prod_{i=1}^{p}\prod_{j=1}^{\hat\nu_i} (d + j - i) < \prod_{i=1}^{p}\prod_{j=1}^{\mu_i} (d + j - i) \cdot \prod_{i=1}^{p}\prod_{j=1}^{\nu_i} (d + j - i)
	\end{equation*}\\
	This time, $\hat\mu_{i_0} = \mu_{i_0} - 1$ and $\hat\mu_i = \mu_i$ for $i \neq i_0$, so the above inequality is the same as
	
	\begin{equation*}
		\prod_{i=1}^{p}\prod_{j=1}^{\hat\nu_i} (d + j - i) < (d + \mu_{i_0} - i_0) \cdot \prod_{i=1}^{p}\prod_{j=1}^{\nu_i} (d + j - i) .
	\end{equation*}\\
	Again, $\nu$ and $\hat\nu$ differ by a single block in a single row $q'$ where $\hat\nu_{q'} = \nu_{q'} + 1$. Now, $q'$ must be the first row of $\nu$ of length $\lambda_{i_0} - \mu_{i_0}$, and $\nu_{q'} = \lambda_{i_0} - \mu_{i_0}$. Proceeding, the latest inequality equivalent to
	
	\begin{equation*}
		d + \nu_{q'} + 1 - q' < d + \mu_{i_0} - i_0
	\end{equation*}\\
	or just
	
	\begin{equation}\label{eqn:update-B}
		\nu_{q'} + 1 - q' < \mu_{i_0} - i_0 .
	\end{equation}\\
	So if $\hat\mu$ is obtained by Update (B), then $r^{\lambda}_{\mu,\nu}(d) < r^{\lambda}_{\hat\mu,\hat\nu}(d)$ if and only if $(\ref{eqn:update-B})$ holds.
	
	We claim that either Update (A) or (B) increases the content ratio, so assume to the contrary that neither does, which is to say that both (\ref{eqn:update-A}) and (\ref{eqn:update-B}) are false:
	\begin{gather*}
		\mu_{i_0+1} - i_0 \geq \nu_q - q \\
		\nu_{q'} + 1 - q' \geq \mu_{i_0} - i_0.
	\end{gather*}
	Adding these two inequalities together and rearranging terms results in
	\begin{equation*}
		\nu_{q'} - \nu_q + q - q' + 1 \geq \mu_{i_0} - \mu_{i_0+1} \overset{(\ref{eqn:non-decreasing})}{>} \lambda_{i_0} - \lambda_{i_0+1}
	\end{equation*}
	Now, substituting in the values of $\nu_q , \nu_{q'}$ yields
	\begin{align}\label{eqn:contradiction}
		&\lambda_{i_0} - \mu_{i_0} - (\lambda_{i_0+1} - \mu_{i_0+1}) + q - q' + 1 > \lambda_{i_0} - \lambda_{i_0+1} \nonumber \\
		\Longrightarrow (&\mu_{i_0+1} - \mu_{i_0}) + (q - q') + 1 > 0 \nonumber \\
		\Longrightarrow (&\mu_{i_0+1} - \mu_{i_0}) + (q - q') \geq 0.
	\end{align}
	
	We've already established that $\mu_{i_0+1} - \mu_{i_0}$ is a strictly negative number. It turns out that $q - q'$ is too; recall that $q$ is the last row of $\nu$ with length $\lambda_{i_0+1} - \mu_{i_0+1}$, and $q'$ is the first row of $\nu$ with length $\lambda_{i_0} - \mu_{i_0}$. Since $\lambda_{i_0+1} - \mu_{i_0+1}$ is the larger number by $(\ref{eqn:non-decreasing})$, it occurs in an earlier (higher) row than $\lambda_{i_0} - \mu_{i_0}$, hence $q < q'$. Which is to say $q - q'$ is also a negative number, so (\ref{eqn:contradiction}) is impossible and the proof is done.
	
\end{proof}

\textbf{Example)} Consider $\lambda = (4,4,3)$ and $\mu = (3,2,0)$. In the following Young diagram of $\lambda$, the boxes of $\mu$ are shaded.
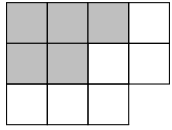
\begin{equation*}
	\ytableausetup{nobaseline,boxsize=1.5em}
	\begin{ytableau}
		*(lightgray) & *(lightgray) & *(lightgray) & \\
		*(lightgray) & *(lightgray) & & \\
		 & & \\
	\end{ytableau}
\end{equation*} \\
Clearly $\nu = \text{rows}(\lambda/\mu) = (3,2,1)$, and $r^{\lambda}_{\mu,\nu}(3) = 2$. Since $(\lambda_i - \mu_i)_{i=1}^3$ is not already decreasing, we must find a row where the sequence of differences increases and apply update $(A)$ or $(B)$. Let us choose $i_0 = 1$.

\textbf{Update (A):} Increase $\mu_2$ so that $\hat\mu = (3,3,0)$, yielding
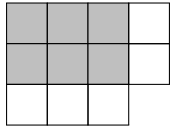
\begin{equation*}
	\ytableausetup{nobaseline,boxsize=1.5em}
	\begin{ytableau}
		*(lightgray) & *(lightgray) & *(lightgray) & \\
		*(lightgray) & *(lightgray) & *(lightgray) & \\
		& & \\
	\end{ytableau}
\end{equation*} \\
for which $\hat\nu = (3,1,1)$ and $r^{\lambda}_{\hat\mu,\hat\nu}(3) = 1.5$.

\textbf{Update (B):} Decrease $\mu_1$ so that $\hat\mu = (2,2,0)$, yielding
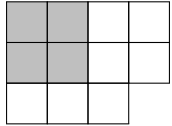
\begin{equation*}
	\ytableausetup{nobaseline,boxsize=1.5em}
	\begin{ytableau}
		*(lightgray) & *(lightgray) & & \\
		*(lightgray) & *(lightgray) & & \\
		& & \\
	\end{ytableau}
\end{equation*} \\
for which $\hat\nu = (3,2,2)$ and $r^{\lambda}_{\hat\mu,\hat\nu}(3) = 5$.

Update (A) failed to increase content ratio compared to our original $\lambda,\mu,\nu$, but Update (B) succeeded. Although the results of (B) still do not satisfy $\lambda = \hat\mu + \hat\nu$, we could repeatedly apply this process to eventually arrive at the desired form.

We are ready to prove part $iii)$ of the initial outline.

\begin{prop}\label{thm:special-form-bound}
	
	Let $\lambda$ be a partition of $n$ such that $ht(\lambda) \leq d$ and $wd(\lambda) \leq k$. If $\mu,\nu$ are subpartitions such that $\lambda = \mu + \nu$, then $r^{\lambda}_{\mu,\nu}(d) \leq (n+1)^{k^2}$ for all $n\geq k$.
	
\end{prop}

\begin{proof}
	To begin, we observe
	\begin{align}
		r^{\lambda}_{\mu,\nu}(d) &= \frac{ \prod_{i=1}^{p} \prod_{j=1}^{\mu_i + \nu_i} (d + j - i) }{ \prod_{i=1}^{p} \prod_{j=1}^{\mu_i} (d + j - i) \cdot \prod_{i=1}^{p} \prod_{j=1}^{\nu_i} (d + j - i) } \label{eqn:ratio} \\
		&= \frac{ \prod_{i=1}^{p} \prod_{j=\mu_i + 1}^{\mu_i + \nu_i} (d + j - i) }{ \prod_{i=1}^{p} \prod_{j=1}^{\nu_i} (d + j - i) } \label{eqn:ratio-cancel-mu} \\
		&\text{Change of variables in numerator: } \hat{\jmath} = j - \mu_i \nonumber \\
		&= \frac{ \prod_{i=1}^{p} \prod_{j=1}^{\nu_i} (d + j + \mu_i - i) }{ \prod_{i=1}^{p} \prod_{j=1}^{\nu_i} (d + j - i) } \nonumber \\
		&\leq \frac{ \prod_{i=1}^{p} \prod_{j=1}^{\nu_i} (d + j - i + k) }{ \prod_{i=1}^{p} \prod_{j=1}^{\nu_i} (d + j - i) } \nonumber
	\end{align}
	where the last inequality holds by the assumption that every part of $\lambda$, and hence every part of $\mu$, has size at most $k$.
	
	At this point, the product $\prod_{i=1}^{p} \prod_{j=1}^{\nu_i}$ says to multiply over the boxes of $\nu$ across rows then down columns. Denoting the heights of columns of $\nu$ by $h_1 , \ldots , h_k$, we can flip the indexing to $\prod_{j=1}^{k} \prod_{i=1}^{h_j}$ to go down columns then across rows. We then have
	\begin{align}\label{eqn:ratio-comparison}
		r^{\lambda}_{\mu,\nu}(d) &\leq \frac{ \prod_{j=1}^{k} \prod_{i=1}^{h_j} (d + j - ( i - k )) }{ \prod_{j=1}^{k} \prod_{i=1}^{h_j} (d + j - i) } \nonumber \\ 
		&\text{Change of variables in numerator: } \hat{\imath} = i - k \nonumber \\
		&= \prod_{j=1}^{k} \frac{ \prod_{i=-k + 1}^{-k + h_j} (d + j - i) }{ \prod_{i=1}^{h_j} (d + j - i) } 
	\end{align}

	We split the product over $j$ into two parts, consisting of those columns $h_j$ for which $h_j\leq k$ (in which there will be no cancellation between numerator and denominator) and those for which $h_j>k$ (in which there will be cancellation). 	
	
	\textbf{Case 1 (no cancellation):} For fixed $j$ assume $h_j \leq k$ i.e. $-k + h_j \leq 0$. Then
	\begin{align*}
		&\hphantom{=}\frac{ \prod_{i=-k + 1}^{-k + h_j} (d + j - i) }{ \prod_{i=1}^{h_j} (d + j - i) } \\
		&\text{Change of variables in denominator: } \hat{\imath} = i - k \\
		&= \prod_{i=-k + 1}^{-k + h_j} \frac{ d + j - i }{ d + j - (i + k) } \\
		&= \prod_{i=-k + 1}^{-k + h_j} \left( 1 + \frac{ k }{ d - i - k + j } \right) \\
	\end{align*}
	
	The upper bound of $i \leq -k + h_j$ implies $d - h_j \leq d - i - k$. But $d - h_j \geq 0$ and $j \geq 1$, so $d - i - k + j \geq 1$. It follows that
	\begin{align*}
		\frac{ \prod_{i=-k + 1}^{-k + h_j} (d + j - i) }{ \prod_{i=1}^{h_j} (d + j - i) } &= \prod_{i=-k + 1}^{-k + h_j} \left( 1 + \frac{ k }{ d - i - k + j } \right) \\
		&\leq \prod_{i=-k + 1}^{-k + h_j} ( 1 + k ) \\
		&= ( 1 + k )^{h_j} \\
		&\leq (1 + n)^{k}.
	\end{align*}
	assuming $n\geq k$. 	
	
	\textbf{Case 2 (cancellation):} For fixed $j$ assume $h_j > k$. Then terms will cancel between the numerator and denominator, yielding
	\begin{align*}
		\frac{ \prod_{i=-k + 1}^{-k + h_j} (d + j - i) }{ \prod_{i=1}^{h_j} (d + j - i) } &= \frac{ \prod_{i=-k + 1}^{0} (d + j - i) }{ \prod_{i=-k + h_j + 1}^{h_j} (d + j - i) } \\
		&\text{Change of variables in denominator: } \hat{\imath} = i - h_j \\
		&= \prod_{i=-k + 1}^{0} \frac{ d + j - i }{ d + j - (i + h_j) } \\
		&= \prod_{i=-k + 1}^{0} \left( 1 + \frac{ h_j }{ d - h_j + j - i } \right) \\
	\end{align*}
		
	\noindent Now $d - h_j \geq 0$ since $ht(\nu) \leq ht(\lambda) \leq d$. Also $j - i \geq 1$ since $j \geq 1$ while $i \leq 0$. Hence $d - h_j + j - i \geq 1$ and
	\begin{align*}
		\frac{ \prod_{i=-k + 1}^{-k + h_j} (d + j - i) }{ \prod_{i=1}^{h_j} (d + j - i) } &= \prod_{i=-k + 1}^{0} \left( 1 + \frac{ h_j }{ d - h_j + j - i } \right) \\
		&\leq \prod_{i=-k + 1}^{0} ( 1 + h_j ) \\
		&= (1+h_j)^{k} \\
		&\leq (1+n)^{k}.
	\end{align*}
	
	Thus for the whole product (\ref{eqn:ratio-comparison}) we obtain, assuming $n\geq k$, 
	\begin{align*}
		r^{\lambda}_{\mu,\nu}(d) &\leq \prod_{j=1}^{k} \frac{ \prod_{i=-k + 1}^{-k + h_j} (d + j - i) }{ \prod_{i=1}^{h_j} (d + j - i) } \\
		&\leq \prod_{j=1}^{k} (1+n)^{k} \\
		&= (1+n)^{k^2}.
	\end{align*}
	
\end{proof}

\begin{proof}[Proof of Theorem~\ref{thm:content-ratio-bound}]
	Combine Propositions \ref{thm:special-form-gives-max-ratio} and \ref{thm:special-form-bound}.
\end{proof}

\section{Further discussion and remarks}\label{sec:remarks}

\subsection{Determinants of more general polynomials} It turns out that Conjecture~\ref{eqn:pencil-limit-conjecture} implies a stronger version of itself, where we can replace the linear pencils $L_{\mathcal X}(\mathcal U)$ with more general (noncommutative) polynomial functions of $\mathcal U$ (still allowing matrix coefficients). One must find some condition on the polynomials to stand in for the spectral radius condition on the pencil; the following notion of {\em stability} turns out to be suitable:  a noncommutative polynomial $p \in \matrixspace{d}(\mathbb{C}\langle x_1 , \ldots , x_g \rangle)$ will be called {\em stable} if $\det(p(\mathcal Z)) \neq 0$ whenever $\mathcal Z$ is a $g$-tuple of matrices with $\| \mathcal Z \|_{row} \leq 1$.

\begin{prop}
	
Let $p$ be a noncommutative polynomial with $p(0)=I_d$.  Then $p$ is stable if and only if there exists $\mathcal{X} \in \matrixspace{k}^g$ with ${\bf rad}(\mathcal{X}) < 1$ so that
	\begin{align}\label{eqn:det-rep}
		\det(p(\mathcal{Z})) &= \det ( I_k \otimes I_d - \sum_{j=1}^g X_j \otimes Z_j )  \\
		&= \det ( L_\mathcal{X}(\mathcal{Z}) ) . \nonumber
	\end{align}
	
\end{prop}

\begin{proof}
  Suppose $p$ is such that (\ref{eqn:det-rep}) holds. Since the tuple $\mathcal X$ has ${\bf rad}(\mathcal X)<1$, it is jointly similar to a strict column contraction, i.e. there exists an invertible $S$ and tuple $\mathcal Y$ such that $S^{-1}X_jS =Y_j$ for each $j$, and for which $\| Y_j^*Y_j\|<1$. Clearly (\ref{eqn:det-rep}) still holds with $\mathcal Y$ in place of $\mathcal X$. Since $\mathcal Z$ is a row contraction and $\mathcal Y$ is a column contraction, the sum $\sum Y_j\otimes Z_j$ is a strict contraction, and hence has spectral radius strictly less than $1$. (i.e. ${\|\sum Y_j\otimes Z_j\|\leq \|\sum Y_j^*Y_j\|^{1/2} \|\sum Z_jZ_j^*\|^{1/2} <1}$.)  Therefore $I \otimes I - \sum Y_j \otimes Z_j$ (and hence ${I \otimes I - \sum X_j \otimes Z_j}$) is invertible for all $\mathcal Z$ in the closed row ball, and so by (\ref{eqn:det-rep}) $p$ is stable.

  Conversely, suppose $p$ is stable and $p(0)=1$. Consider the nc rational function $r(\mathcal Z)\vcentcolon=p(\mathcal Z)^{-1}$, this function is regular on the closed row ball. Any nc rational function has a {\em realization} $r(\mathcal Z)=w^*(I \otimes I-\sum X_j\otimes Z_j)^{-1}y$, which we choose to be {\em minimal}, i.e. the $\mathcal X$ tuple has smallest size $k$ among all possible realizations. It follows then from \cite{Volcic-2017} that the domain of $r$ is equal to the set of $\mathcal Z$ for which ${\det(I \otimes I -\sum X_j\otimes Z_j)\neq 0}$; which by construction is exactly the set where $\det(p(\mathcal Z))\neq 0$. By \cite[Theorem A]{Jury-Martin-Shamovich-2021}, since $r$ is regular in the closed row ball, any minimal tuple $\mathcal X$ must have ${\bf rad}(\mathcal X)<1$. (The result in \cite{{Jury-Martin-Shamovich-2021}} is stated and proved only for rational functions with scalar coefficients, but the proof goes through in the matrix coefficient case.) On the other hand, by \cite[Lemma 5.3]{Helton-Klep-Volcic-2018}, for this pencil $L_{\mathcal{X}}$ which appears in the realization of $p(\mathcal Z)^{-1}$, we have $\det p(\mathcal{Z}) = \det L_{\mathcal X}(\mathcal Z)$ for all $\mathcal Z$. This completes the proof.

\end{proof}

Thus, if Conjecture~\ref{eqn:pencil-limit-conjecture} holds, then the limit
\begin{equation}
\lim_{d\to \infty} \int_{\mathcal U(d)^g} \det p(\mathcal U) \overline{\det q(\mathcal U)}\, d\mathcal U
\end{equation}
will exist for all stable nc polynomials $p,q$. If $L_{\mathcal X}$, $L_{\mathcal Y}$ are pencils linearizing $p,q$ as in (\ref{eqn:det-rep}), then the limiting value will be
\[
\det(I \otimes I-\sum X_j\otimes \overline{Y_j})^{-1} = \det p(\overline{Y})^{-1} = \overline{(\det q(X)})^{-1}  
\]
We note that this accords with the one-variable setting: if $p$ is a polynomial normalized to have $p(0)=1$, we may factor $p$ as $p(z):=\prod_{j=1}^k (1+x_jz)$. Letting $X$ be any $k\times k$ matrix with eigenvalues $x_1, \dots, x_k$, we have $\det p(U)=\det (I_k\otimes I_d+X\otimes U)$, so that by (\ref{eqn:classical-limit}) the large $d$ limit exists only when all $|x_j|<1$, which is precisely the condition that $p$ has no zeroes in the unit disk $|z|\leq 1$.

Even in the one-variable case the stability condition cannot be weakened further, as one can see putting $x=y=1$ in (\ref{eqn:HS-ident})---in this case $p$ will have a single zero on the unit circle, and the integral blows up as $d\to \infty$.

\subsection{More on asymptotics}
Our results will imply some asymptotics for the expected value of $|\det (x_0U_0+x_1U1+\cdots +x_gU_g)|^{2k}$ when the associated pencil is {\em conic}. We say the pencil $\sum_{j=0}^g x_j U_j$ is {\em conic} if $|x_0|^2 -|x_1|^2-\cdots -|x_g|^2>0$.

Given a conic pencil, put $\tilde{x_j} = \frac{x_j}{x_0}$ for $j=1, \dots, g$. Then $\sum|\tilde{x_j}|^2<1$. Now if $U_0, U_1, \dots, U_g$ are independent Haar unitaries, it is easy to check that $U_0^*U_1, \dots, U_0^*U_g$ are independent Haar unitaries, and that at size $d\times d$ 
\begin{equation*}
	\int_{U(d)^{g+1}} |\det(x_0U_0+\cdots + x_gU_g)|^{2k} \, d\mathcal{U} = |x_0|^{2dk} \int_{U(d)^{g}} |\det(I+\tilde{x_1}U_1+\cdots +\tilde{x_g}U_g)|^{2k} \, d\mathcal{U}.
\end{equation*}
The limit of the last integral exists by Corollary~\ref{cor:scalar-limit}, and has the value $(1-\|\tilde{x}\|_2^2)^{-k^2}$. Writing $t=\exp(\log t)$ for $t>0$ we obtain the following asymptotics for conic unitary pencils:

\begin{cor}\label{cor:conic-asymptotics}For conic pencils we have the asymptotic
\begin{equation}\label{eqn:szego-asymptotic}
\int_{\mathcal U(d)^{g+1}} |\det(x_0U_0+\cdots + x_gU_g)|^{2k}\, d\mathcal{U} = \exp\left( d\cdot k\cdot c_0  + k^2 c_1 +o(1)\right)
\end{equation}
where
\[
c_0 = \log|x_0|^2 \quad\text{and}\quad c_1 = \log\left(\frac{|x_0|^2}{|x_0|^2-\sum_{j=1}^g|x_j|^2} \right)
\]
\end{cor}
When $g=1$, this result reduces to a simple instance of the Strong Szeg\H{o} Limit Theorem for Toeplitz determinants, so we will refer to these as {\em Szeg\H{o} asymptotics}. (We refer to \cite{Bump-Diaconis-2002} for an explanation of the connection to Toeplitz determinants.) As an example we can obtain Szeg\H{o} asymptotics for characteristic polynomials of linear combinations of unitaries. Let $x_1, \dots, x_g$ be fixed complex numbers and put $U(x):= \sum_{j=1}^g x_j U_j$. Then for all complex numbers $z$ with $|z|^2 > \sum_{j=1}^g |x_j|^2$, we will have
\[
\int_{\mathcal U(d)^g} |\det(zI_d -U(x))|^{2k} \, d\mathcal U \sim \left(\frac{|z|^2}{|z|^2-\sum|x_j|^2}\right)^{k^2} |z|^{2dk} \quad \text{as } d\to \infty.
\]
Thus for example when all the $x_j=1$ we get asymptotics for the moments of the characteristic polynomial for $U_1+U_2+\cdots +U_g$ when $|z|>\sqrt{g}$, namely
\[
\int_{\mathcal U(d)^g} |\det(zI_d -(U_1+\cdots +U_g)|^{2k} \, d\mathcal U \sim \left(\frac{|z|^2}{|z|^2-g}\right)^{k^2} |z|^{2dk} \quad \text{as } d\to \infty.
\]
The restriction $|z|>\sqrt{g}$ may seem an artifact of our theorem, but it is natural in an intrinsic sense: it is known \cite{Basak-Dembo-2013} that the empirical spectral distribution of a sum of $g$ independent Haar unitaries converges, as the size $d\to \infty$, to the so called {\em Brown measure} of the sum $u_1+\cdots +u_g$ where the $u_j$ are {\em freely independent Haar unitaries} in the sense of free probability. This Brown measure is known explicitly, and has support equal to the closed disk of radius $\sqrt{g}$ (when $g\geq 2$) \cite[Example 5.5]{Haagerup-Larsen-2000}. (When $g=1$ the Brown measure of a single Haar unitary is just normalized arc length measure on the circle. We refer \cite{Haagerup-Larsen-2000} and its references for background on the Brown measure.) Thus our corollary says that the even moments of the characteristic polynomials obey Szeg\H{o} asymptotics when $z$ lies {\em outside the support of the Brown measure}.

For another example, consider $x_1U_1+x_2U_2$ with $0<x_1<x_2$. As before if $|z|^2>x_1^2+x_2^2$ we have the Szeg\H{o} asymptotic as in (\ref{eqn:szego-asymptotic}). But also, if instead $|z|^2< x_2^2 -x_1^2$ we can rearrange to $x_1^2+|z|^2<x_2^2$, which is again conic but now with respect to $x_2$ rather than $|z|$. So our result implies that in this regime the asymptotic will be 
\[
\int_{\mathcal U(d)^g} |\det(zI_d -(x_1U_1+x_2U_2))|^{2k} \, d\mathcal U \sim \left(\frac{x_2^2}{x_2^2-x_1^2-|z|^2}\right)^{k^2} x_2^{2dk} \quad \text{as } d\to \infty.
\]
Thus, we get asymptotics for $z$ outside of the annulus $\sqrt{x_2^2-x_1^2}\leq |z|\leq \sqrt{x_1^2+x_2^2}$. This annulus is precisely the support of the Brown measure for the sum of freely independent Haar unitaries $x_1u_1+x_2u_2$  \cite[Example 5.5]{Haagerup-Larsen-2000}.  One expects that this Brown measure should be the limit of the empirical spectral distributions of $x_1U_1+x_2U_2$, though this result does not seem to be known. 

The case of $z$ inside the support of the Brown measure will still be governed by Theorem~\ref{thm:scalar-closed-form} of course, but the asymptotics seem rather more difficult to analyze in this regime, especially for $k>1$. In the single variable setting, the asymptotics for $z$ inside the support of the Brown measure (that is, on the unit circle), are quite different, and are related to the the so-called {\em Fisher-Hartwig asymptotics} of Toeplitz determinants.

\bibliographystyle{plain}
\bibliography{citations.bib}

\end{document}